\documentclass[12pt]{article}
\usepackage{amssymb}
\usepackage{amsmath}
\usepackage{amsthm}
\usepackage{authblk}
\usepackage{color}
\usepackage{comment}
\usepackage{geometry}
\usepackage{graphicx}

\makeatletter
	\@addtoreset{equation}{section}
\makeatother

\let\originalleft\left
\let\originalright\right
\renewcommand{\left}{\mathopen{}\mathclose\bgroup\originalleft}
\renewcommand{\right}{\aftergroup\egroup\originalright}

\begin{document}

\newcommand\cN{\mathcal{N}}
\newcommand\cP{\mathcal{P}}
\newcommand\cQ{\mathcal{Q}}
\newcommand\cS{\mathcal{S}}
\newcommand{\rD}{{\rm D}}
\newcommand{\ee}{\varepsilon}

\newcommand{\myStep}[2]{{\bf Step #1} --- #2\\}

\newtheorem{theorem}{Theorem}
\newtheorem{corollary}[theorem]{Corollary}
\newtheorem{lemma}[theorem]{Lemma}
\newtheorem{proposition}[theorem]{Proposition}

\theoremstyle{definition}
\newtheorem{definition}{Definition}

\theoremstyle{remark}
\newtheorem{remark}{Remark}[section]

\title{From two-dimensional continuous maps to one-dimensional discontinuous maps:
a novel reduction explaining complex bifurcation structures in piecewise-linear families of maps.}
\author[1]{D.J.W.~Simpson\thanks{d.j.w.simpson@massey.ac.nz}}
\author[2]{V.~Avrutin}
\affil[1]{School of Mathematical and Computational Sciences, Massey University, Palmerston North 4410, New Zealand}
\affil[2]{Institute for Systems Theory and Automatic Control, University of Stuttgart, Pfaffenwaldring 9, 70550 Stuttgart, Germany}
\maketitle

\begin{abstract}

Piecewise-linear maps describe dynamical phenomena that switch between distinct states and readily generate complex bifurcation structures due to their strong nonlinearity. We show that two-dimensional continuous piecewise-linear maps near certain codimension-two homoclinic bifurcations are well approximated by a three-parameter family of one-dimensional maps. Each member of the one-dimensional family is discontinuous, because the family is constructed from the first return of iterates to a subset of phase space, and comprised of infinitely many linear pieces, where each piece corresponds to a fixed number of iterations near the saddle associated with the homoclinic bifurcation. The one-dimensional family exhibits period-incrementing, period-adding, bandcount-incrementing, and bandcount-adding structures (all typical for two-piece maps), as well as unique features caused by orbits repeatedly visiting more than two pieces of the map. These structures carry through to the two-dimensional maps with only minor differences in the arrangement of the bifurcations developing with the distance from the codimension-two bifurcations. This leads to a novel and vivid elucidation of the dynamics of the two-dimensional border-collision normal form.

\end{abstract}

\section{Introduction}
\label{sec:}

The purpose of this paper is to advance our understanding of 
the discrete-time dynamics $(x,y) \mapsto f(x,y)$, where
\begin{equation}
f(x,y) =
\begin{cases}
\begin{bmatrix} \tau_L x + y + 1 \\ -\delta_L x \end{bmatrix}, & x \le 0, \\[4mm]
\begin{bmatrix} \tau_R x + y + 1 \\ -\delta_R x \end{bmatrix}, & x \ge 0,
\end{cases}
\label{eq:f}
\end{equation}
and $\tau_L, \delta_L, \tau_R, \delta_R \in \mathbb{R}$ are parameters.
The family \eqref{eq:f} is the two-dimensional border-collision normal form,
except the border-collision parameter, often denoted $\mu$, has been scaled to $1$.
Any continuous, two-piece, piecewise-linear map on $\mathbb{R}^2$ satisfying a certain non-degeneracy condition
can be converted to \eqref{eq:f} via an affine change of coordinates \cite{NuYo92,Si23e}.
Such maps provide leading-order approximations to Poincar\'e maps
of piecewise-smooth ODE systems \cite{DiBu08,Si16},
particularly Filippov systems modelling mechanical apparatus with stick-slip friction \cite{DiKo03,SzOs08},
and hybrid systems modelling electrical devices with switching control strategies \cite{ZhMo03}.
The maps are also used as mathematical models, especially in economics where discrete-time models are preferred
and piecewise-smooth functions reflect the use of different investment strategies,
for example, under different circumstances \cite{PuSu06}.
Families of piecewise-linear maps, such as the Lozi family \cite{Lo78}, which is a subfamily of \eqref{eq:f},
are also used as test-beds for exploring and advancing theoretical aspects of chaos theory.

We are most interested in the {\em attractors} of \eqref{eq:f}, as these govern the long-term behaviour of typical orbits.
In principle, the parameter space of \eqref{eq:f} can be divided into regions according to the nature of its attractors,
where region boundaries are bifurcations at which attractors are created or destroyed or change in a fundamental way.
This has motivated many studies, e.g.~\cite{AvSc12,AvZh16,BaGr99,BaYo98,FaSi23,PaBa02,SiMe08b,SuAv22},
but such a division can probably never be fully achieved because parameter space is $\mathbb{R}^4$ (which is large),
and the dynamics of \eqref{eq:f} can be stupendously complicated
(for every $N \ge 1$ there exist open parameter regions in which \eqref{eq:f} has $N$ coexisting stable periodic solutions \cite{Si14},
and open parameter regions in which \eqref{eq:f} has $N$ coexisting chaotic attractors \cite{Si24c}).

Thus, rather than attempting a complete characterisation of the attractors across parameter space,
it is more valuable to characterise broad aspects of the bifurcation patterns that dominate parameter space.
One such pattern is the distinctive sausage-string geometry
of regions where the map has a stable period-$p$ solution, for some fixed $p$ \cite{Si17c,SuGa08,ZhMo06b}.
This geometry arises when the dynamics is strongly rotational,
and occurs because the piecewise-linearity
forces periodic solutions to cross the switching line
only at codimension-two shrinking points \cite{Si24d}.
Another pattern concerns {\em robust chaos} \cite{BaYo98,Gl17}. 
Regions of robust chaos can be divided into subregions according to the number of connected components of the attractor.
Where these numbers are powers of two, the organisation of the subregions
is dictated by a simple renormalisation operator \cite{GhSi22,GhMc24}.
Third, when a period-$p$ solution loses stability by attaining a stability multiplier of $-1$,
it does not experience a flip (or period-doubling) bifurcation because \eqref{eq:f} lacks the nonlinear terms necessary to
generate a local period-$2 p$ solution.
Instead, this often triggers a sequence
of three bifurcations whereby the attractor changes from period-$p$ to a chaotic attractor with $2 p$ connected components,
then to a chaotic attractor with $p$ connected components,
and then to a chaotic attractor with one connected component.
This sequence is well understood for one-dimensional maps \cite{NuYo95,SuAv16}
and occurs analogously in higher dimensions.

\begin{figure}[b!]
\begin{center}
\includegraphics[width=7cm]{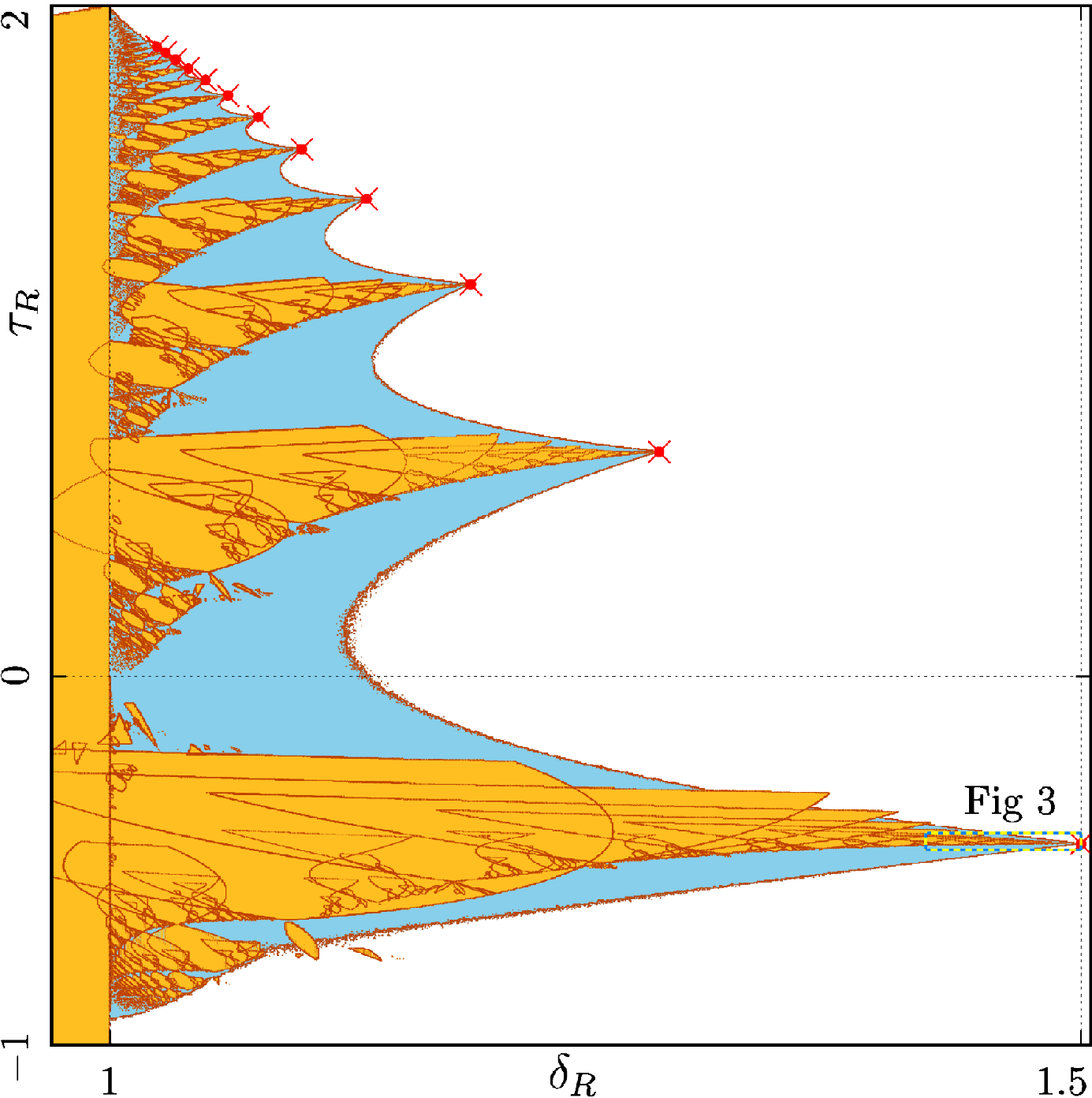}
\caption{
A large-scale bifurcation set of the
two-dimensional border-collision normal form $f$, given by \eqref{eq:f}, with $\tau_L = 2$ and $\delta_L = 0.75$.
In yellow regions $f$ has a periodic attractor,
in blue regions $f$ has a chaotic attractor,
and in white regions $f$ has no attractor.
The orange curves are boundaries of regions where $f$ has a periodic attractor of fixed period.
The red X's indicate parameter points of subsumed homoclinic connections.
Refer to \S\ref{sec:numericalMethods} for a description of the numerical methods used to generate this figure.
\label{fig:bifSetLargeScale}
} 
\end{center}
\end{figure}

Fig.~\ref{fig:bifSetLargeScale} illustrates a fourth pattern.
From each node marked with a red X there issues a sequence
of periodicity regions (yellow) sandwiched between regions of robust chaos (blue).
The nodes are parameter points at which \eqref{eq:f} has a {\em subsumed homoclinic connection} \cite{SiTu17}
where one branch of the unstable set of a saddle periodic solution
is completely contained within its stable set.
For the right-most node in Fig.~\ref{fig:bifSetLargeScale}, this connection is shown in Fig.~\ref{fig:subHCC}a,
where the saddle is the fixed point of the left piece of \eqref{eq:f}.
As another example, Fig.~\ref{fig:subHCC}b shows a subsumed homoclinic connection
for which the saddle is a period-three solution.
Such connections are infinitely degenerate for smooth maps,
but codimension-two for piecewise-linear maps.

\begin{figure}[b!]
\begin{center}
\includegraphics[width=17cm]{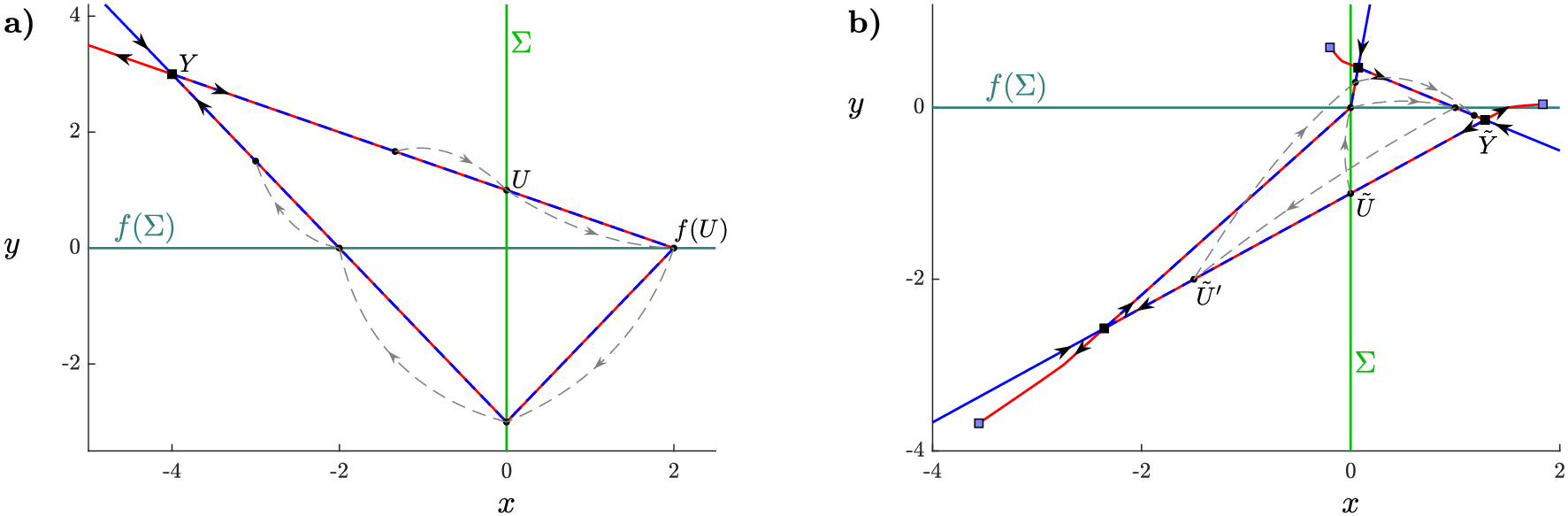}
\caption{
Subsumed homoclinic connections of the two-dimensional border-collision normal form \eqref{eq:f}.
Panel (a) uses $(\tau_L,\delta_L,\tau_R,\delta_R) = (2,0.75,-0.5,1.5)$ corresponding to
the right-most X in Fig.~\ref{fig:bifSetLargeScale}.
The black square $Y$ is a saddle fixed point and its stable and unstable sets are coloured blue and red respectively.
The vertical line is the switching line $\Sigma$;
the horizontal line is its image $f(\Sigma)$;
the dashed curves indicate the action of the map.
Panel (b) uses $\left( \tau_L, \delta_L, \tau_R, \delta_R \right)
= \left( -\frac{23}{33}, \frac{13}{66}, -\frac{5}{2}, 2 \right)$, taken from \cite{Si20}.
The three black squares form a saddle period-three solution ($RLR$-cycle)
and its stable and unstable sets are coloured blue and red respectively.
There is also a stable period-three solution ($RLL$-cycle) indicated with blue squares.
\label{fig:subHCC}
} 
\end{center}
\end{figure}

Previous work \cite{Si20} explained the main sequence of roughly triangular regions $\cP_k'$
issuing from a subsumed homoclinic connection,
seen more clearly for the right-most node in the magnification Fig.~\ref{fig:bifSetBCNF}a.
In each $\cP_k'$ \eqref{eq:f} has an asymptotically stable period-$(k+m)$ solution
with one excursion far from the saddle, where $m \ge 1$ is a constant specific to the homoclinic connection
(Fig.~\ref{fig:bifSetBCNF}a uses $m = 2$).
These regions were explained by performing asymptotic calculations about an arbitrary subsumed homoclinic connection
to obtain formulas for the three bifurcation curves that bound each region.

\begin{figure}[b!]
\begin{center}
\setlength{\unitlength}{1cm}
\begin{picture}(17,7)
\put(.7,0){\includegraphics[height=7cm]{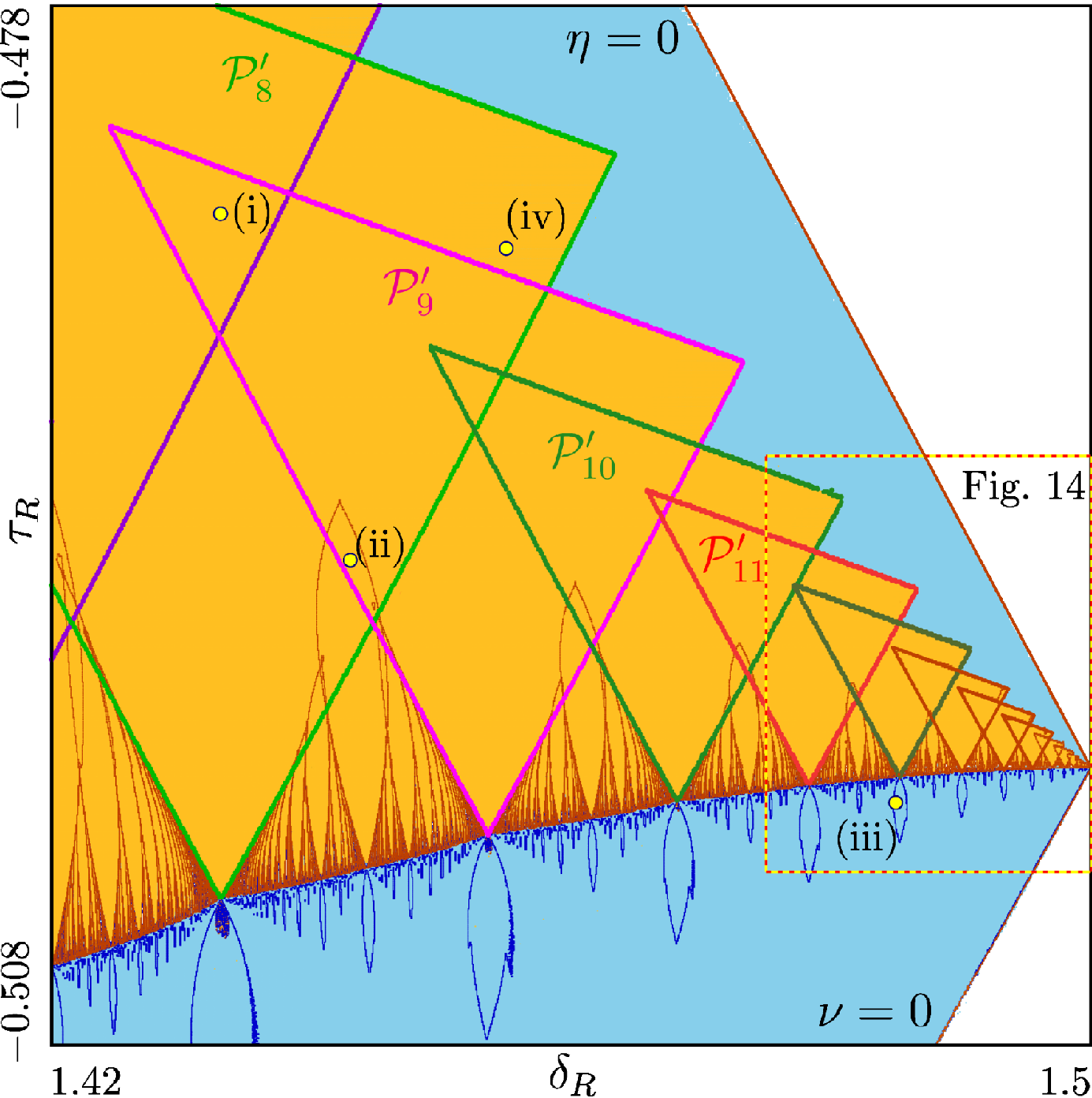}}
\put(10,0){\includegraphics[height=7cm]{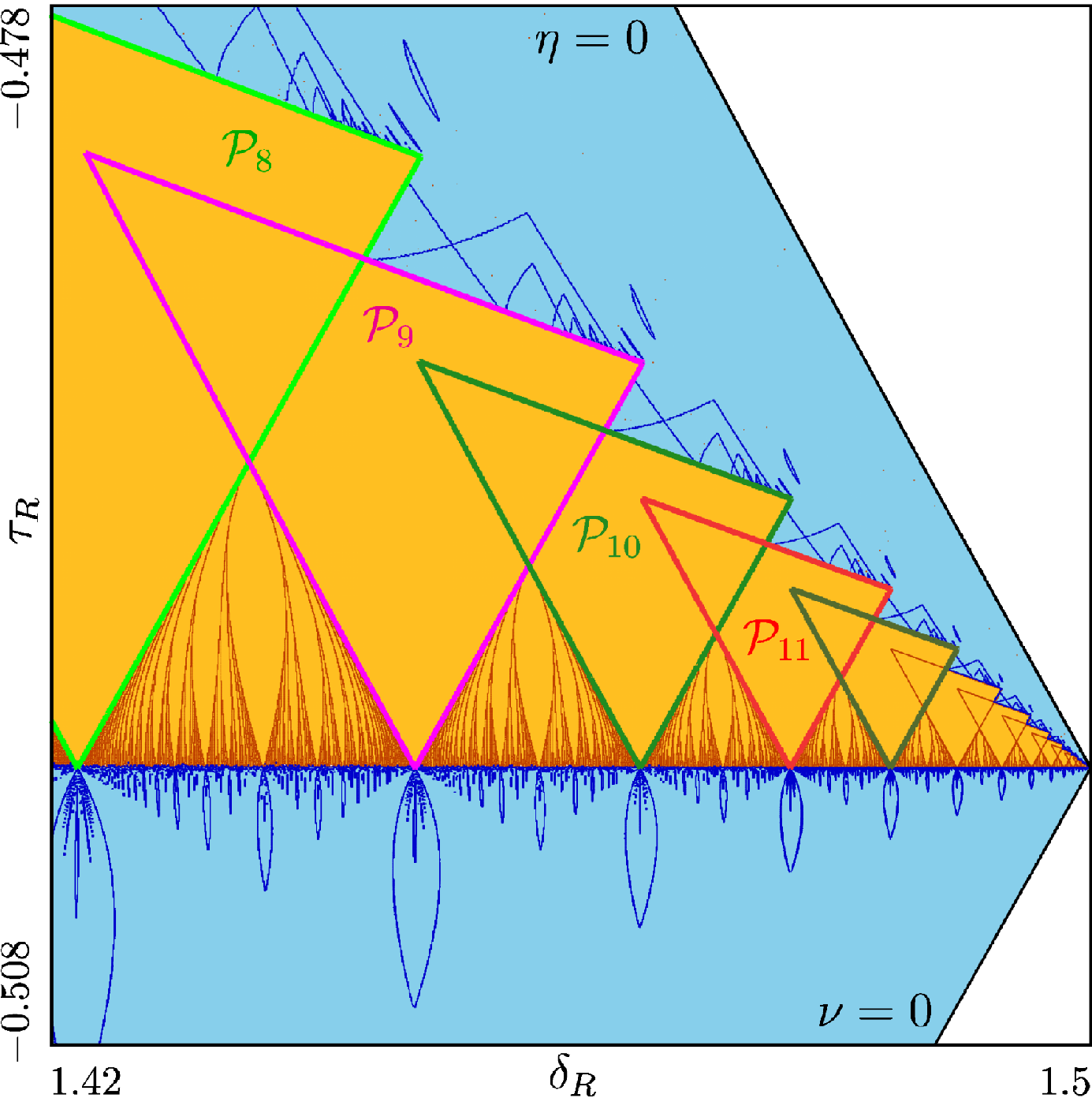}}
\put(0,6.6){{\bf a)}}
\put(9.3,6.6){{\bf b)}}
\end{picture}
\caption{
Panel (a) is a magnification of Fig.~\ref{fig:bifSetLargeScale};
panel (b) is a bifurcation set of the corresponding one-dimensional family $h$, given by \eqref{eq:h}.
Specifically, panel (b) uses $\sigma = 1.5$ and $\eta$ and $\nu$
given by \eqref{eq:eta2} and \eqref{eq:nu2}, with also $\lambda = 0.5$.
For panel (a) [panel (b)],
in yellow regions $f$ [$h$] has a periodic attractor,
in blue regions $f$ [$h$] has a chaotic attractor,
and in white regions $f$ [$h$] has no attractor.
The orange curves are boundaries of regions where $f$ [$h$] has a periodic attractor of fixed period;
the blue curves are boundaries of regions where $f$ [$h$] has a chaotic attractor with a fixed number of connected components.
In (a) the roughly triangular regions $\cP_k'$ are where $f$ has a stable periodic solution of period $p = k+2$;
in (b) the triangular regions $\cP_k$ are where $h$ has a stable fixed point $z_k^* \in I_k$.
Both panels show the homoclinic corner curves $\eta = 0$ and $\nu = 0$.
These curves intersect at the right-most X in Fig.~\ref{fig:bifSetLargeScale}.
In (a) the parameter points (i), (ii), (iii), and (iv)
are examined in Figs.~\ref{fig:basins}a, \ref{fig:basins}b, \ref{fig:bandexA}, and \ref{fig:bandexB} respectively.
\label{fig:bifSetBCNF}
} 
\end{center}
\end{figure}

The aim of this paper is to explain the {\em entire} bifurcation structure
issuing from a subsumed homoclinic connection.
This is achieved by showing that the dynamics is well approximated by a family of one-dimensional maps $h(z;\eta,\nu,\sigma)$,
and performing a comprehensive analysis of this family.
The one-dimensional family has variable $z > 0$,
parameters $\eta, \nu \in \mathbb{R}$ and $\sigma > 1$, and is defined as follows.
For each $k \in \mathbb{Z}$, define the interval
\begin{equation}
I_k = \left[ \sigma^{-k}, \sigma^{-(k-1)} \right),
\label{eq:Ik}
\end{equation}
and the linear function
\begin{equation}
h_k(z;\eta,\nu,\sigma) = \frac{\eta - \nu}{\sigma - 1} \,\sigma^k z + \frac{-\eta + \sigma \nu}{\sigma - 1}.
\label{eq:hk}
\end{equation}
Then
\begin{equation}
h(z;\eta,\nu,\sigma) = h_k(z;\eta,\nu,\sigma), \quad \text{where $k \in \mathbb{Z}$ is such that $z \in I_k$},
\label{eq:h}
\end{equation}
see Fig.~\ref{fig:discMap}.

\begin{figure}[t!]
\begin{center}
\includegraphics[width=15.5cm]{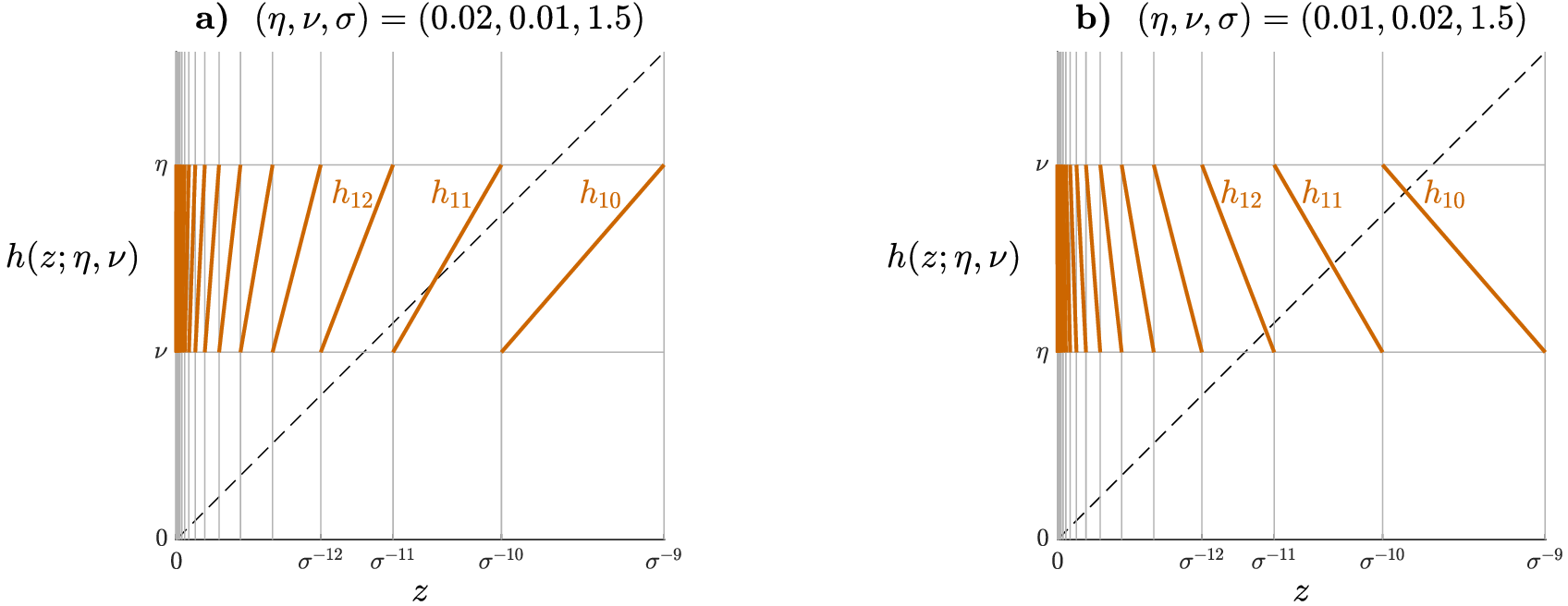}
\caption{
Two instances of the one-dimensional map $h$, given by \eqref{eq:h}.
If $\eta > \nu$, as in (a), then every branch of $h$ is increasing.
If $\eta < \nu$, as in (b), then every branch of $h$ is decreasing.
\label{fig:discMap}
} 
\end{center}
\end{figure}

As an example, panel (b) of Fig.~\ref{fig:bifSetBCNF} shows a bifurcation set of \eqref{eq:h}
using values of $\eta$, $\nu$, and $\sigma$ that correspond to panel (a).
Remarkably, the full bifurcation structure of the two-dimensional family \eqref{eq:f}
is successfully reproduced by the one-dimensional family \eqref{eq:h}
with only minor differences that diminish as we approach the codimension-two point.

The remainder of the paper is organised as follows.
In \S\ref{sec:numericalMethods} we summarise the numerical methods used to generate the bifurcation sets.
Then in \S\ref{sec:derivation} we derive the one-dimensional family \eqref{eq:h}
as an approximation to the dynamics of \eqref{eq:f}.
For simplicity, here we only treat subsumed homoclinic connections
to saddle fixed points of the left piece of the map
(other saddles can be handled in a similar way, see \S\ref{sub:exD} and \cite{Si20}).
The approximation is obtained by considering the first return map $F$
for orbits of \eqref{eq:f} to the third quadrant of the plane.
Proximity to the homoclinic connection ensures that
each iteration of $F$ corresponds to many iterations of \eqref{eq:f} near the fixed point,
so over the course of these iterations the orbit comes close to the unstable subspace associated with the fixed point.
This subspace is one-dimensional, consequently $F$ is usefully approximated by a one-dimensional map.
This map has the form \eqref{eq:h}, where each $h_k$ corresponds to orbits of \eqref{eq:f}
that undergo $k$ iterations in the left half-plane and $m$ iterations in the right half-plane
before returning to the third quadrant.
The approximation is stated precisely as Theorem \ref{th:main},
which is then proved in \S\ref{sec:proof}.

Section \ref{sec:examples} provides three examples.
We first derive explicit expressions for the values of $\sigma$, $\eta$, and $\nu$
in terms of the parameters of \eqref{eq:f}
for the example with $m = 2$ shown in Fig.~\ref{fig:bifSetBCNF}.
We then derive analogous expressions when $m = 3$ and $\delta_L = 0$.
In this case the left piece of \eqref{eq:f} has degenerate range ---
this arises when \eqref{eq:f} models the oscillatory dynamics of a Filippov system
near a grazing-sliding bifurcation \cite{DiKo02,Si25f}.
Lastly we consider parameter values near those of Fig.~\ref{fig:subHCC}b for which the saddle
is a period-three solution, and
show how $\sigma$, $\eta$, and $\nu$ can be defined and evaluated in this situation.
For all three examples we compare a bifurcation set of \eqref{eq:f} to that of
the corresponding subfamily of \eqref{eq:h}.

In \S\ref{sec:1d} we study the basic properties of \eqref{eq:h}.
We first show \eqref{eq:h} is unchanged when $\eta$ and $\nu$ are scaled by $\sigma$,
except that the value of $k$ shifts by one.
We then identify an absorbing interval $J$ that contains all invariant sets of \eqref{eq:h}
and determine the number of pieces of \eqref{eq:h} over $J$.
Lastly we analyse fixed points of \eqref{eq:h}
and compute the triangles $\cP_k$ where $h_k$ has a stable fixed point in $J$.

In \S\ref{sec:fourStructures} we study the bifurcation structures of \eqref{eq:h}
and investigate the degree to which these are realised by \eqref{eq:f}.
The most significant difference is that above the triangles $\cP_k$ in Fig.~\ref{fig:bifSetBCNF}b
the attractor of \eqref{eq:h} is chaotic and can be the union of several disjoint intervals,
whereas above the corresponding regions $\cP_k'$ in Fig.~\ref{fig:bifSetBCNF}a
the attractor of \eqref{eq:f} is chaotic but has only one connected component.
We explain why the attractor of \eqref{eq:f} here has low density instead of gaps,
and argue that this discrepancy often does not occur below each $\cP_k'$.
The parameter space of \eqref{eq:h} can broadly be separated into four areas:
period-incrementing, period-adding, bandcount-adding, and bandcount-incrementing,
and we examine each of these areas in turn.
Where \eqref{eq:h} has at most two pieces over $J$,
its dynamics can be understood from the theory of two-piece, piecewise-linear maps \cite{AvGa19},
while elsewhere \eqref{eq:h} can have additional complexities.
Concluding comments are provided in \S\ref{sec:conc}.

\section{Numerical methods}
\label{sec:numericalMethods}

Here we summarise the methods used to generate Fig.~\ref{fig:bifSetBCNF}.
Other bifurcation sets were generated in the same way.
Other plots, such as one-parameter bifurcation diagrams, were produced via standard techniques.

To numerically compute a bifurcation set, the most basic approach is a Monte-Carlo simulation
whereby forward orbits of the map are computed over a grid of parameter points,
and each point coloured by the nature of the long-term behaviour of the orbit or orbits.
This approach is useful for a preliminary exploration,
but yields an unclear picture where the map has multiple attractors, even if multiple initial values are used.
In particular, the approach often fails to detect attractors with small basins of attraction,
and hence does not accurately resolve bifurcation boundaries where the size of a basin tends to zero.

To mitigate this we take advantage of the piecewise-linear nature of $f$.
For any finite symbolic itinerary $\cS$ comprised of $L$'s and $R$'s,
the periodic solution of $f$ that follows $\cS$ (termed an $\cS$-cycle)
can be calculated by solving the fixed point equation
for the composition of the pieces of $f$ in order specified by $\cS$ \cite{Si16}.
This equation is linear, so the periodic solution (and its stability multipliers)
can be computed directly and accurately without requiring an iterative root-finding algorithm.

So to produce Fig.~\ref{fig:bifSetBCNF}a we first performed a Monte-Carlo simulation over a coarse grid to
identify a set of candidate symbolic itineraries.
For each candidate $\cS$, we then used the corresponding fixed point equation
to identify over a fine ($1200 \times 1200$) grid the parameter region where $f$ has a stable $\cS$-cycle.
Points in this region were identified by growing outwards from potential points found during the coarse grid simulation,
and using the fixed point equation to verify existence and stability.
This allowed us to consider subsets of the full parameter region for each $\cS$, significantly reducing run time.
The orange curves in Fig.~\ref{fig:bifSetBCNF}a are the result of post-processing contour fits to the boundaries of these regions.
In this way the bifurcation boundaries of the periodic solutions are found accurately
and the extent of coexisting attractors is clear.

The blue curves in Fig.~\ref{fig:bifSetBCNF}a bound regions where a chaotic attractor
of $f$ has a fixed number of connected components.
These curves were computed by Eckstein's greatest common divisor algorithm \cite{AvEc07,Ec06}, which proceeds as follows.
For each parameter point we compute an orbit $(x_i,y_i)$ and select a collection of reference points
$\left( x_{i_1}, y_{i_1} \right), \ldots, \left( x_{i_N}, y_{i_N} \right)$ that are assumed to belong to the attractor.
Given $\ee > 0$, we then identify values $k$ for which
$\left\| (x_k,y_k) - \left( x_{i_j}, y_{i_j} \right) \right\| < \ee$ for some $j \in \{ 1,\ldots,N \}$ with $k > i_j$.
Our estimate for the number of connected components of the attractor is
the greatest common divisor of all differences $k - i_j$.
This algorithm relies on the continuity of $f$ and the assumed ergodicity of the attractor \cite{AvEc07,Ec06}.
For Fig.~\ref{fig:bifSetBCNF}a we used $\ee = 10^{-4}$ and $N = 2000$.
Again contour fits were used to produce the blue curves.
The white regions indicate the absence of an attractor
and are simply where the norm of a point in a sample forward orbit exceeded a threshold value.

Fig.~\ref{fig:bifSetBCNF}b was computed with the same techniques,
except, since $h$ is discontinuous, we resorted to a box-counting approach
to estimate the number of connected components.
To do this we partitioned the range of $h$ into $1000$ intervals (boxes) of equal size,
marked those visited by a long forward orbit in the attractor,
and counted the number of connected clusters of marked intervals.
This approach allows for discontinuities,
but is less efficient than the greatest common divisor algorithm.
Hence many iterates were required to obtain clear results (we used $10^6$ iterations),
particularly in areas of parameter space where the attractor has low density over some intervals.

\section{Construction of the one-dimensional approximation}
\label{sec:derivation}

In this section we study the two-dimensional border-collision normal form $f$, given by \eqref{eq:f},
for parameter values near a subsumed homoclinic connection of the fixed point of the left piece of $f$.
We consider the first return of iterates to the third quadrant,
and show how this can be approximated by the one-dimensional family $h$, given by \eqref{eq:h}.

Let $f_L$ and $f_R$ denote the left and right pieces of $f$.
We assume $\delta_L \ge 0$ and $\delta_R > 0$ so that some of the constructions can be simplified.
Note that $\delta_L$ and $\delta_R$ are the determinants of the Jacobian matrices of $f_L$ and $f_R$,
so these are both positive when $f$ is orientation-preserving,
as is usually the case when $f$ approximates the Poincar\'e map of an ODE system.
We also allow $\delta_L = 0$ to accommodate the scenario that $f$ approximates
a return map for a grazing-sliding bifurcation \cite{DiKo02,Si25f}.
For any point $P \in \mathbb{R}^2$ we write $P = (P_1,P_2)$.

\subsection{First return to the third quadrant}
\label{sub:inducedMap}

Let
\begin{align}
\Omega_L &= \left\{ P \in \mathbb{R}^2 \,\middle|\, P_1 < 0 \right\}, &
\Omega_R &= \left\{ P \in \mathbb{R}^2 \,\middle|\, P_1 > 0 \right\},
\nonumber
\end{align}
denote the open left and right half-planes of $\mathbb{R}^2$,
and let $\Sigma = \left\{ P \in \mathbb{R}^2 \,\middle|\, P_1 = 0 \right\}$ denote the switching line.
These three sets form a partition of the phase space of $f$.

With $\delta_R > 0$, if a forward orbit of $f$ exits $\Omega_R$, it arrives in the third quadrant
\begin{equation}
\cQ_3 = \left\{ P \in \mathbb{R}^2 \,\middle|\, P_1 \le 0,\, P_2 < 0 \right\}.
\nonumber
\end{equation}
This is because if $P \in \Omega_R$ and $f(P) \notin \Omega_R$,
then $P_1 > 0$, $f(P)_1 \le 0$, and $f(P)_2 = -\delta_R P_1 < 0$, so $f(P) \in \cQ_3$.
With also $\delta_L \ge 0$, the forward orbit of any point $P \in \cQ_3$
cannot return to $\cQ_3$ until first entering $\Omega_R$.
This is because if $P \notin \Omega_R$,
then $f(P)_2 = -\delta_L P_1 \ge 0$, so $f(P) \notin \cQ_3$.

The {\em first return map} (or {\em induced map}) $F : \cQ_3 \to \cQ_3$ is defined by
\begin{equation}
F(P) = f^n(P), \quad \text{for the smallest $n \ge 1$ for which $f^n(P) \in \cQ_3$},
\label{eq:F}
\end{equation}
where $P \in \cQ_3$ and $F(P)$ is undefined if no such $n$ exists.
In view of the above remarks,
if $n$ exists then $P$ undergoes some $\ell \ge 1$ iterations under $f_L$,
then $r = n - \ell$ iterations under $f_R$ upon reaching $F(P) \in \cQ_3$.
That is, there exist unique $\ell \ge 1$ and $r \ge 1$ (with $\ell + r = n$) such that
\begin{equation}
F(P) = f_R^r \left( f_L^\ell(P) \right),
\label{eq:F2}
\end{equation}
and for each $k \in \{ 1,2,\ldots,\ell+r-1 \}$,
$f^k(P) \in \Omega_L$ if and only if $k < \ell$.
This is illustrated in Fig.~\ref{fig:schem}
for a point $P$ with $\ell = 6$ and $r = 2$.

\begin{figure}[b!]
\begin{center}
\includegraphics[width=8cm]{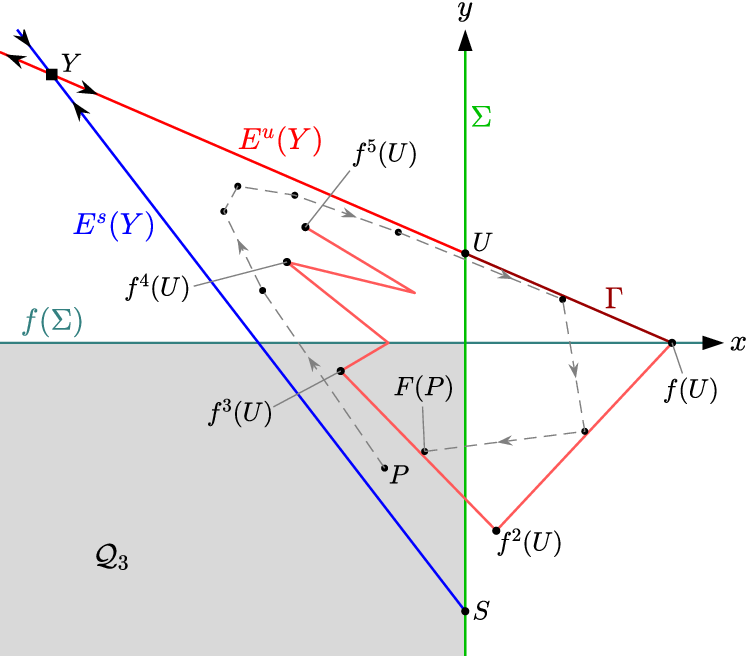}
\caption{
A sketch of the phase space of the two-dimensional border-collision normal form $f$
with $(\tau_L,\delta_L,\tau_R,\delta_R) = (2,0.75,-0.35,1.05)$.
The $x \le 0$ parts of the stable and unstable subspaces $E^s(Y)$ and $E^u(Y)$ are coloured blue and red respectively.
We also show $\Gamma$ (the line segment from $U$ to $f(U)$)
and some of its images under $f$ (these form part of the unstable set $W^u(Y)$).
The third quadrant $\cQ_3$ is shaded,
and we illustrate the forward orbit of a typical point $P \in \cQ_3$ located just above $E^s(Y)$.
\label{fig:schem}
} 
\end{center}
\end{figure}

\subsection{A saddle fixed point}
\label{sub:saddle}

Now suppose $f_L$ has a saddle fixed point in $\Omega_L$.
With $\delta_L \ge 0$ this occurs if and only if $\tau_L > \delta_L + 1$,
and the fixed point is
\begin{equation}
Y = \left( \frac{-1}{\tau_L - \delta_L - 1}, \frac{\delta_L}{\tau_L - \delta_L - 1} \right).
\label{eq:Y}
\end{equation}
The stability multipliers of $Y$, $\lambda \in (0,1)$ and $\sigma > 1$, are the eigenvalues of
the Jacobian matrix $\rD f_L = \begin{bmatrix} \tau_L & 1 \\ -\delta_L & 0 \end{bmatrix}$
(notice $\tau_L = \lambda + \sigma$ and $\delta_L = \lambda \sigma$).
The stable and unstable subspaces for $Y$ are the lines through $Y$ with directions
matching the corresponding eigenvectors of $\rD f_L$:
\begin{align}
E^s(Y) &= \left\{ P \in \mathbb{R}^2 \,\middle|\, \sigma (P_1 - Y_1) + P_2 - Y_2 = 0 \right\}, \\
E^u(Y) &= \left\{ P \in \mathbb{R}^2 \,\middle|\, \lambda (P_1 - Y_1) + P_2 - Y_2 = 0 \right\}.
\end{align}
These lines intersect $\Sigma$ at
\begin{align}
S &= \left( 0, \frac{-\sigma}{\sigma - 1} \right), &
U &= \left( 0, \frac{\lambda}{1 - \lambda} \right),
\label{eq:SU}
\end{align}
see Fig.~\ref{fig:schem}.

Since the unstable stability multiplier $\sigma$ is positive,
the unstable set $W^u(Y)$ has two dynamically independent branches.
The line segment
\begin{equation}
\Gamma = \left\{ (1-\alpha) U + \alpha f(U) \,\middle|\, 0 \le \alpha < 1 \right\}
\nonumber
\end{equation}
is a {\em fundamental domain} for the right branch of $W^u(Y)$,
meaning that every orbit of $f$ in this branch contains exactly one point in $\Gamma$ \cite{KrOs05}.
Therefore this branch can be generated by iterating $\Gamma$ under $f$,
and this is illustrated in Fig.~\ref{fig:schem} which shows $f^k(\Gamma)$ for $k = 1,\ldots,4$.

\subsection{Coordinates relative to the eigendirections}
\label{sub:coords}

Any point $P \in \mathbb{R}^2$ can be expressed in coordinates $(a,b)$ relative to $E^s(Y)$ and $E^u(Y)$ by
\begin{equation}
P = Y + a(S-Y) + b(U-Y).
\label{eq:coordChange}
\end{equation}
By inverting \eqref{eq:coordChange}, and using the above formulas
for $S$, $U$, and $Y$, we obtain
\begin{align}
a(P) &= \frac{\sigma - 1}{\sigma - \lambda} \big( \lambda - (1-\lambda)(\lambda P_1 + P_2) \big), \label{eq:a} \\
b(P) &= \frac{1 - \lambda}{\sigma - \lambda} \big( \sigma + (\sigma-1)(\sigma P_1 + P_2) \big), \label{eq:b}
\end{align}
for any $P \in \mathbb{R}^2$.
In $(a,b)$-coordinates the switching line is $a + b = 1$.

Fig.~\ref{fig:chaoticAttr} provides an example to illustrate $(a,b)$-coordinates.
For the given parameter values,
$f$ appears to have a chaotic attractor (black dots).
Panel (a) uses $(x,y)$-coordinates,
while panel (b) uses $(a,b)$-coordinates over the orange parallelogram shown in panel (a).

\begin{figure}[b!]
\begin{center}
\includegraphics[width=17cm]{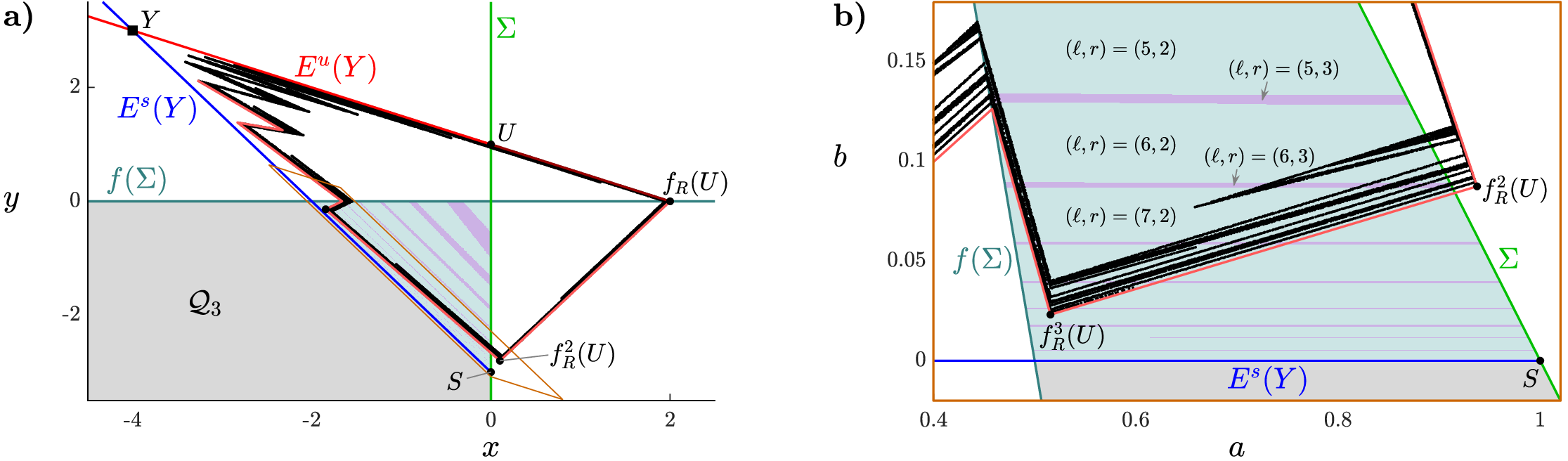}
\caption{
Panel (a) is a phase portrait of two-dimensional border-collision normal form $f$
with $(\tau_L,\delta_L,\tau_R,\delta_R) = (2,0.75,-0.45,1.4)$.
Panel (b) shows part of the same plot in $(a,b)$-coordinates \eqref{eq:coordChange}.
The black dots indicate the attractor of the map (they show several iterates of a typical forward orbit with transient dynamics removed).
The third quadrant $\cQ_3$ is shaded by the value of $r$ in \eqref{eq:F2}:
pale green: $r=2$;
lavender: $r=3$;
grey: $r$ is undefined.
\label{fig:chaoticAttr}
} 
\end{center}
\end{figure}

The forward orbits of points in $\cQ_3$ that lie above $E^s(Y)$ return to $\cQ_3$.
This part of $\cQ_3$ is shaded pale green for points for which $r = 2$ in \eqref{eq:F2},
and lavender for points with $r = 3$.
The green and lavender strips correspond to successively larger values of $\ell$ as we approach $E^s(Y)$.

\subsection{Main result}
\label{sub:main}

In Fig.~\ref{fig:chaoticAttr} the point $f_R^2(U)$ lies relatively close to $S$.
Subsumed homoclinic connections occur when $f_R^m(U) = S$ for some $m \ge 2$,
assuming $f_R^k(U) \in \Omega_R$ for all $k = 1,2,\ldots,m-1$.
This is because, by an inductive argument, in this case each $f^k(\Gamma)$ for $k = 0,1,\ldots,m-1$ is a line segment
that does not cross the switching line, hence its image under $f$ is another line segment.
In particular, $f^m(\Gamma)$ is the line segment from $S$ to $f_R(S)$,
which is the part of the stable subspace $E^s(Y)$ that belongs to $\cQ_3$, and is a subset of the stable set $W^s(Y)$.
Since $\Gamma$ is a fundamental domain for the right branch of $W^u(Y)$,
we can conclude that this branch is contained within $W^s(Y)$.

Theorem \ref{th:main} given below contains the assumption that a subsumed homoclinic connection
occurs at a parameter point $\xi_0$ in the set
\begin{equation}
\Xi = \left\{ (\tau_L,\delta_L,\tau_R,\delta_R) \in \mathbb{R}^4
\,\middle|\, \tau_L > \delta_L + 1,\, 0 \le \delta_L < 1,\, \delta_R > 0 \right\}.
\label{eq:Xi}
\end{equation}
The constraint $\delta_L < 1$ ensures $\lambda \sigma < 1$ at $\xi_0$,
while the remaining constraints in \eqref{eq:Xi} have been mentioned above.
For any nearby parameter point $\xi$, let
\begin{align}
\eta &= b \left( f_R^{m+1}(U) \right), &
\nu &= b \left( f_R^m(U) \right),
\label{eq:nuetaDefn}
\end{align}
and 
\begin{equation}
\ee = {\rm max}[|\eta|,|\nu|],
\label{eq:ee}
\end{equation}
noticing that if $\xi = \xi_0$ then $\ee = 0$.
Also let
\begin{align}
\Psi &= \left\{ P \in \cQ_3 \,\middle|\, 0 < b(P) < 2 \ee \right\}, \label{eq:Psi} \\
\Psi_0 &= \left\{ P \in \Psi \,\middle|\, \sigma^{\ell(P)} b(P) \in [1,\sigma),\, r(P) = m \right\}, \label{eq:Psi0}
\end{align}
where we write $\ell(P)$ and $r(P)$ for the values in \eqref{eq:F2} of a point $P \in \cQ_3$.
Theorem \ref{th:main} shows that most points in $\Psi$ belong to $\Psi_0$,
and that if $P \in \Psi_0$ then $z' \approx h(z)$,
where $z = b(P)$ and $z' = b(F(P))$.
That is, $h$ approximates the $b$-dynamics of the first return map for most points $P$.
The parameter values of $h$ are given by \eqref{eq:nuetaDefn} and the unstable stability multiplier $\sigma$.

\begin{theorem}
Suppose for $\xi_0 \in \Xi$ there exists $m \ge 2$
such that $f_R^m(U) = S$ and $f_R^k(U) \in \Omega_R$ for all $k = 1,2,\ldots,m-1$.
Then there exists a neighbourhood $\cN \subset \mathbb{R}^4$ of $\xi_0$ and constants $C_1, C_2 > 0$
such that for all $\xi \in \cN \cap \Xi$ we have
\begin{align}
1 - \frac{{\rm area} \left( \Psi_0 \right)}{{\rm area}(\Psi)} &< C_1 \ee, \label{eq:fraction} \\
|z' - h(z)| &< C_2 \ee^c, \qquad \text{for all $P \in \Psi_0$}, \label{eq:mainApprox}
\end{align}
where $z = b(P)$, $z' = b(F(P))$, and $c \in \mathbb{R} \cup \{ \infty \}$ is such that $\lambda \sigma^c = 1$.
\label{th:main}
\end{theorem}

\begin{remark}
If $\lambda = 0$ (equivalently $\delta_L = 0$) then $c = \infty$ and $\ee^c = 0$ in \eqref{eq:mainApprox}.
In this case, $z' = h(z)$, i.e.~there is no error in the approximation for points in $\Psi_0$.
This is because if $\lambda = 0$, then, before returning to $\cQ_3$,
the forward orbit any $P \in \cQ_3$ with $b(P) > 0$ lands on $\Gamma$,
from which the formula \eqref{eq:hk} for $h_k$ is constructed.

If instead $\lambda > 0$, then $c \in \mathbb{R}$ is such that $\lambda \sigma^c = 1$.
Since $\lambda \sigma = \delta_L < 1$ at $\xi_0$,
we can assume $\cN$ is chosen small enough that $c > 1$ throughout $\cN \cap \Xi$.
This ensures that the order of the absolute error $|z' - h(z)|$ is higher than the order of $h(z)$ (which is $\ee$).
\label{re:c}
\end{remark}

\begin{remark}
The condition on $\ell(P)$ in \eqref{eq:Psi0} is equivalent to $z \in I_{\ell(P)}$.
This ensures that for $P \in \Psi_0$ we have $h(z) = h_k(z)$, where $k = \ell(P)$.
Also $r(P) = m$ in \eqref{eq:Psi0},
hence each iteration of \eqref{eq:mainApprox}
corresponds to $\ell = k$ iterations of $f_L$ and $r = m$ iterations of $f_R$.
The value of $m$ is fixed, whereas $k$ is given in terms of $z$ by
\begin{equation}
k = K(z,\sigma) = \left\lceil -\frac{\ln(z)}{\ln(\sigma)} \right\rceil,
\label{eq:kFormula}
\end{equation}
obtained by rearranging \eqref{eq:Ik}.
\label{re:k}
\end{remark}

\begin{remark}
Fig.~\ref{fig:approx1d} illustrates the effectiveness of the approximation $z' \approx h(z)$ for the
example of Fig.~\ref{fig:chaoticAttr}.
Here $m = 2$ and $\eta = 0.023125$ and $\nu = 0.0875$.
The dots show pairs $(z,z')$ for points in the attractor, while the orange lines show the map $h$.
Each dot is turquoise if $P \in \Psi_0$, and purple otherwise.
As expected each turquoise dot is close to the corresponding branch $h_k$, that is $z' \approx h(z)$.
The approximation is better for larger values of $k$
because in this case the forward orbit of $P$ experiences more iterations under $f_L$ 
and hence gets closer the line segment $\Gamma$ from which the approximation is derived.
\label{re:example}
\end{remark}

\begin{figure}[b!]
\begin{center}
\includegraphics[width=7.5cm]{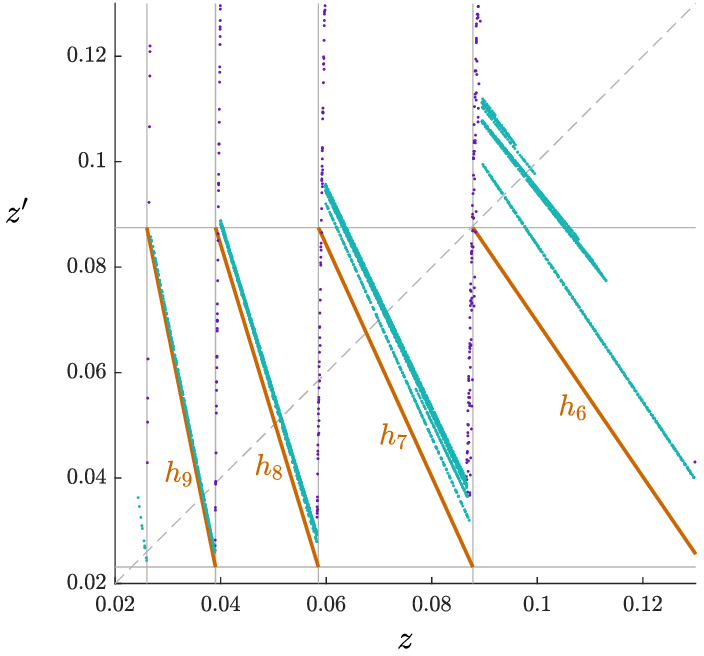}
\caption{
An illustration of the one-dimensional map $h$, given by \eqref{eq:h},
as an approximation to the dynamics of the two-dimensional border-collision normal form $f$, given by \eqref{eq:f},
with the parameter values of Fig.~\ref{fig:chaoticAttr}.
For each $P \in \cQ_3$ in the numerically computed attractor of $f$,
we plot $z' = b(F(P))$ against $z = b(P)$,
where $F(P)$ is the next point in the forward orbit of $P$ that belongs to $\cQ_3$.
The orange lines show the branches $h_k$
with $\sigma = 1.5$ and $\eta = 0.023125$ and $\nu = 0.0875$
given by the formulas \eqref{eq:eta2} and \eqref{eq:nu2}.
\label{fig:approx1d}
} 
\end{center}
\end{figure}

\begin{remark}
If $\eta > 0$ and $\nu > 0$, then forward orbits of $f$ are trapped above $E^s(Y)$ and below $E^u(Y)$, so $f$ has an attractor.
The codimension-one scenarios $\eta = 0$ and $\nu = 0$ are {\em homoclinic corners} \cite{Si16b}
where the unstable set $W^u(Y)$ attains a non-trivial intersection with the stable set $W^s(Y)$.
In Fig.~\ref{fig:bifSetBCNF} these curves are where the attractor is destroyed
because the attractor approaches $U$ as $\eta \to 0$ or $\nu \to 0$,
and with $\eta < 0$ or $\nu < 0$ the forward orbits of points near $U$ enter $\Omega_L$
below $E^s(Y)$ and subsequently diverge.
\label{re:homoclinicCorners}
\end{remark}

\section{Verification of the one-dimensional approximation}
\label{sec:proof}

In this section we prove Theorem \ref{th:main}.
We first give a heuristic explanation, then provide a formal proof.

For any point $P \in \Psi$, the value $b(P)$ is small.
Thus the forward orbit of $P$ experiences many iterations of $f$ before escaping $\Omega_L$.
During this time the orbit approaches $E^u(Y)$,
thus its first point outside $\Omega_L$, namely $f_L^{\ell(P)}(P)$, lies close to the line segment $\Gamma$.
So 
\begin{equation}
f_L^{\ell(P)}(P) \approx (1-\alpha) U + \alpha f(U),
\label{eq:proofOutline10}
\end{equation}
for some $\alpha \in [0,1)$.
Also $b(U) = 1$ and $b(f(U)) = \sigma$,
thus we likely have $b \left( f_L^{\ell(P)}(P) \right) \in [1,\sigma)$,
which is equivalent to $\sigma^{\ell(P)} b(P) \in [1,\sigma)$.

Since $\eta$ and $\nu$ are small, the image of $\Gamma$ under $f_R^m$ lies close to $E^s(Y) \cap \cQ_3$.
Thus most points on and near $\Gamma$ take $m$ iterations to escape $\Omega_R$,
so we likely have $r(P) = m$.
Since $f_R$ is affine, \eqref{eq:proofOutline10} implies
\begin{equation}
f_R^m \left( f_L^{\ell(P)}(P) \right) \approx (1-\alpha) f_R^m(U) + \alpha f_R^{m+1}(U),
\nonumber
\end{equation}
and since $b$ is affine
\begin{equation}
z' \approx (1-\alpha) \nu + \alpha \eta,
\label{eq:proofOutline30}
\end{equation}
using the definition of $\eta$ and $\nu$ and assuming $P \in \Psi_0$.
Finally, from a formula for $\alpha$ in terms of $z = b(P)$,
we obtain $(1-\alpha) \nu + \alpha \eta = h(z)$, resulting in \eqref{eq:mainApprox}.
Most of the effort in the proof is in bounding the set $\Psi_0$
and showing that $C_1$ exists.

\begin{proof}[Proof of Theorem \ref{th:main}]
The proof is completed in four steps.
Given $t_1 > 0$ and $\ee > 0$ sufficiently small, the set
\begin{equation}
\Gamma_\ee = \left\{ Z \in \mathbb{R}^2 \,|\, 0 \le a(Z) \le (2 \ee)^c, 1 + t_1 \ee < b(Z) < \sigma - t_1 \ee \right\}
\label{eq:Gammaee}
\end{equation}
is a fattened version of the line segment $\Gamma$, but with its ends trimmed (except if $c = \infty$, then it is not fattened).
In Step 1 we show that $t_1 > 0$ and the neighbourhood $\cN$ of $\xi_0$ can be chosen so that
\begin{equation}
\text{$r(P) = m$ for all $P \in \cQ_3$ for which $f_L^{\ell(P)}(P) \in \Gamma_\ee$} \,,
\label{eq:GammaeeProperty}
\end{equation}
for all $\xi \in \cN \cap \Xi$.

Given $t_2 > 0$, $\ee > 0$, and $k \ge 0$, let
\begin{equation}
\Phi_k = \left\{ P \in \cQ_3 \,\middle|\, (1 + t_2 \ee) \sigma^{-k} < b(P) < (\sigma - t_2 \ee) \sigma^{-k} \right\}.
\label{eq:Phik}
\end{equation}
In Step 2 we show that $t_2 > 0$ can be chosen so that
\begin{equation}
\text{$\ell(P) = k$ and $f_L^{\ell(P)}(P) \in \Gamma_\ee$ for all $P \in \Psi \cap \Phi_k$,}
\label{eq:PhikProperty}
\end{equation}
for all $\xi \in \cN \cap \Xi$,
reducing the size of $\cN$ if necessary.
This implies $\Psi \cap \left( \bigcup_{k \ge 0} \Phi_k \right) \subset \Psi_0$
from which we obtain \eqref{eq:fraction} in Step 3.
Finally in Step 4 we verify \eqref{eq:mainApprox}.

\myStep{1}{Verify \eqref{eq:GammaeeProperty}.}
Consider the line $G = f_R^{-1} \left( E^s(Y) \right)$
and notice $S \in E^s(Y) \cap G$ (because $S, f_R(S) \in E^s(Y)$).
At $\xi_0$ the point $S$ is the only point of intersection of $E^s(Y)$ and $G$,
for otherwise $E^s(Y)$ and $G$ would coincide and be invariant under $f_R$, and we could not have $S = f_R^m(U)$.
Thus there exists a neighbourhood $\cN \subset \mathbb{R}^4$ of $\xi_0$
such that $E^s(Y) \cap G = \{ S \}$ for all $\xi \in \cN \cap \Xi$.

Throughout this neighbourhood
$f_R^m(U)$ is an order-$\eta$ distance from $G$
and an order-$\nu$ distance from $E^s(Y)$ (by the definition \eqref{eq:nuetaDefn} of $\eta$ and $\nu$).
Thus $f_R^m(U)$ is an order-$\ee$ distance from $S$,
so since $S \in \Sigma$ the line $f_R^{-m}(\Sigma)$ comes within an order-$\ee$ distance of $U$.
At $\xi_0$ this line passes through $U$ but not $f_R(U)$,
thus for all $\xi \in \cN \cap \Xi$ this line does not intersect $\Gamma_\ee$,
assuming $t_1 > 0$ is chosen sufficiently large and reducing the size of $\cN$ if necessary.
Thus for all $\xi \in \cN \cap \Xi$ and $Z \in \Gamma_\ee$,
we have $f_R^m(Z) \in \Omega_L$.
By a similar argument $f_R^k(Z) \in \Omega_R$
for all $k = 0,1,\ldots,m-1$
(increasing $t_1$ and shrinking $\cN$ if necessary),
hence \eqref{eq:GammaeeProperty} holds.

\myStep{2}{Verify \eqref{eq:PhikProperty}.}
Since the coordinates $a$ and $b$ are aligned with the stable and unstable directions of $f_L$,
iterating $f_L$ corresponds to scaling the values of $a$ and $b$ by the eigenvalues $\lambda$ and $\sigma$.
That is,
\begin{align}
a \left( f_L^k(P) \right) &= \lambda^k a(P), \label{eq:aScaling} \\
b \left( f_L^k(P) \right) &= \sigma^k b(P), \label{eq:bScaling}
\end{align}
for all $P \in \mathbb{R}^2$ and $k \ge 0$.
Since $\Omega_L$ consists of all points with $a + b < 1$,
for any $P \in \cQ_3$ the value $\ell = \ell(P)$ is the smallest positive integer for which
$a \left( f_L^\ell(P) \right) + b \left( f_L^\ell(P) \right) \ge 1$.

Set $t_2 = {\rm max} \left[ 2 \sigma^2, t_1 \right]$.
Let $k \ge 0$ and choose any $P \in \Psi \cap \Phi_k$.
Then $b(P) < 2 \ee$ and $\sigma^{-k} < b(P)$, thus $\sigma^{-k} < 2 \ee$.
Hence
\begin{equation}
a \left( f_L^k(P) \right)
= \lambda^k a(P)
\le \sigma^{-k c}
\le (2 \ee)^c,
\label{eq:boundA}
\end{equation}
using also $a(P) < 1$ and $c > 1$ (true if $\cN$ is sufficiently small).
Also, by \eqref{eq:Phik} and \eqref{eq:bScaling},
\begin{equation}
1 + t_2 \ee < b \left( f_L^k(P) \right) < \sigma - t_2 \ee.
\label{eq:boundB}
\end{equation}
From \eqref{eq:boundA} and \eqref{eq:boundB},
\begin{align}
a \left( f_L^{k-1}(P) \right) + b \left( f_L^{k-1}(P) \right)
< (2 \ee \sigma)^c + \frac{\sigma - t_2 \ee}{\sigma}
< 1
\label{eq:boundC}
\end{align}
using $t_2 \ge 2 \sigma^2$ to obtain the last inequality.
It follows that $a \left( f_L^j(P) \right) + b \left( f_L^j(P) \right) < 1$
for all $j = 0,1,\ldots,k-1$.
Also $a \left( f_L^k \right) + b \left( f_L^k \right) > 1 + t_2 \ee > 1$,
therefore $\ell(P) = k$.
Finally, \eqref{eq:boundA} and \eqref{eq:boundB} imply $f_L^{\ell(P)}(P) \in \Gamma_\ee$ (because $t_2 \ge t_1$),
hence \eqref{eq:PhikProperty} holds.

\myStep{3}{Verify \eqref{eq:fraction}.}
By \eqref{eq:GammaeeProperty} and \eqref{eq:PhikProperty},
any $P \in \Psi \cap \Phi_k$ has $\ell(P) = k$ and $r(P) = m$.
Also $\sigma^k b(P) \in [1,\sigma)$ by \eqref{eq:boundB},
so any $P \in \Psi \cap \Phi_k$ belongs to $\Psi_0$.

For all $k \ge 0$, let $\Phi_k' = \left\{ P \in \cQ_3 \,\middle|\, \sigma^{-k} \le b(P) < \sigma^{-(k-1)} \right\}$.
For any $\xi \in \cN \cap \Xi$
there exists $k^*$ such that $\Psi$ is the union of all $\Phi_k'$ with $k > k^*$,
and part of $\Phi_{k^*}'$ (cut off at $b = 2 \ee$).
Moreover, $\Psi_0$ contains all $\Phi_k$ with $k > k^*$.
So since
\begin{equation}
1 - \frac{{\rm area}(\Phi_k)}{{\rm area}(\Phi_k')} \to \frac{2 t_2 \ee}{\sigma - 1} \qquad \text{as $k \to \infty$,}
\nonumber
\end{equation}
we have \eqref{eq:fraction}, say with $C_1 = \frac{4 t_2}{\sigma - 1}$.

\myStep{4}{Verify \eqref{eq:mainApprox}.}
Given $P \in \Psi_0$, let $Q \in E^u(Y)$ be such that
\begin{equation}
b \left( f_L^\ell(P) \right) = b(Q),
\label{eq:proof7}
\end{equation}
and write
\begin{equation}
Q = (1-\alpha) U + \alpha f(U).
\label{eq:proof8}
\end{equation}
Since $b$ is affine with $b(U) = 1$ and $b(f(U)) = \sigma$,
we have $b(Q) = 1 - \alpha + \alpha \sigma$.
So from \eqref{eq:bScaling} and \eqref{eq:proof7},
\begin{equation}
\sigma^\ell b(P) = 1 - \alpha + \alpha \sigma,
\nonumber
\end{equation}
and by solving for $\alpha$ we obtain
\begin{equation}
\alpha = \frac{\sigma^\ell b(P) - 1}{\sigma - 1}.
\label{eq:proof12}
\end{equation}
By \eqref{eq:proof8},
\begin{equation}
f_R^m(Q) = (1-\alpha) f_R^m(U) + \alpha f_R^{m+1}(U),
\nonumber
\end{equation}
and by \eqref{eq:nuetaDefn},
\begin{equation}
b \left( f_R^m(Q) \right) = (1-\alpha) \nu + \alpha \eta.
\nonumber
\end{equation}
By then substituting \eqref{eq:proof12} and $z = b(P)$, we obtain
\begin{equation}
b \left( f_R^m(Q) \right) = \frac{\sigma - \sigma^\ell z}{\sigma - 1} \,\nu
+ \frac{\sigma^\ell z - 1}{\sigma - 1} \,\eta
= h_\ell(z;\eta,\nu,\sigma),
\label{eq:proof30}
\end{equation}
and notice $h_\ell = h$ because $\sigma^\ell z \in [1,\sigma)$.

Finally, by \eqref{eq:boundA} and \eqref{eq:proof7} we have
\begin{equation}
\left\| f_L^\ell(P) - Q \right\| \le (2 \ee)^c \| S - Y \|,
\nonumber
\end{equation}
using also the definition \eqref{eq:coordChange} of $a$ and $b$.
Thus
\begin{equation}
\left\| f_R^m \left( f_L^\ell(P) \right) - f_R^m(Q) \right\| \le (2 \ee)^c \| S - Y \| \| \rD f_R \|^m,
\nonumber
\end{equation}
and
\begin{equation}
\left| z' - b \left( f_R^m(Q) \right) \right| \le C_2 \ee^c,
\nonumber
\end{equation}
where $C_2 = \frac{2^c \| S - Y \| \| \rD f_R \|^m}{\| U - Y \|}$, using again \eqref{eq:coordChange}.
In view of \eqref{eq:proof30}, this verifies \eqref{eq:mainApprox}.
\end{proof}

\section{Examples}
\label{sec:examples}

In this section we first derive formulas for $\eta$ and $\nu$
for the example shown in Fig.~\ref{fig:bifSetBCNF}.
We then consider $\delta_L = 0$, which arises when $f$ captures the dynamics near grazing-sliding bifurcations \cite{DiKo02,Si25f},
and in this case $\lambda = 0$ so there is no error in the approximation \eqref{eq:mainApprox}.
Finally we derive an analogous one-dimensional approximation to the situation in Fig.~\ref{fig:subHCC}b
for which the homoclinic connection is associated with a period-three solution
instead of the fixed point $Y$.

For each example we show a bifurcation set of $f$
and a bifurcation set of the corresponding subfamily of $h$
(done already in Fig.~\ref{fig:bifSetBCNF} for the first example).
For the first example a more detailed examination of the bifurcation structure is provided in \S\ref{sec:fourStructures}.

\subsection{A minimal example}
\label{sub:exA}

Recall from \S\ref{sec:derivation} that when the homoclinic connection is associated with the fixed point $Y$,
the parameters $\eta$ and $\nu$ are given in terms of the forward orbit of $U$ by \eqref{eq:nuetaDefn}.
By iterating $U$, given by \eqref{eq:SU}, under the right piece of $f$, we obtain
\begin{align}
f_R(U) &= \left( \frac{1}{1-\lambda}, 0 \right), \label{eq:fRU} \\
f_R^2(U) &= \left( 1 + \frac{\tau_R}{1 - \lambda}, \frac{-\delta_R}{1 - \lambda} \right), \label{eq:fR2U} \\
f_R^3(U) &= \left( \tau_R + 1 + \frac{\tau_R^2 - \delta_R}{1 - \lambda},
-\delta_R - \frac{\tau_R \delta_R}{1 - \lambda} \right), \label{eq:fR3U} \\
f_R^4(U) &= \left( \tau_R^2 + \tau_R - \delta_R + 1 + \frac{\tau_R^3 - 2 \tau_R \delta_R}{1 - \lambda},
-\delta_R (\tau_R + 1) - \frac{\delta_R (\tau_R^2 - \delta_R)}{1 - \lambda} \right). \label{eq:fR4U}
\end{align}
With $m = 2$, we have $\eta = b \left( f_R^3(U) \right)$ and $\nu = b \left( f_R^2(U) \right)$.
So by substituting \eqref{eq:fR2U} and \eqref{eq:fR3U} into the formula \eqref{eq:b} for $b$, we obtain
\begin{align}
\eta &= \frac{1}{\sigma - \lambda} \Big( \delta_R (\tau_R - \lambda + 1)
+ \big( (\tau_R + \delta_R) \lambda - \tau_R (\tau_R + \delta_R + 1) \big) \sigma \nonumber \\
&\quad+ \big( \tau_R^2 + \tau_R - \delta_R + 1 - (1 + \tau_R) \lambda \big) \sigma^2 \Big), \label{eq:eta2} \\
\nu &= \frac{\delta_R - (\delta_R + \tau_R) \sigma + (\tau_R - \lambda + 1) \sigma^2}{\sigma - \lambda}. \label{eq:nu2}
\end{align}
Fig.~\ref{fig:bifSetBCNF}a shows a bifurcation set of $f$ with $\tau_L = 2$ and $\delta_L = 0.75$.
In this case the stability multipliers of $Y$ are $\lambda = 0.5$ and $\sigma = 1.5$,
so $\eta$ and $\nu$ are given by \eqref{eq:eta2} and \eqref{eq:nu2} with these values of $\lambda$ and $\sigma$.
The curves $\eta = 0$ and $\nu = 0$ intersect at $(\delta_R,\tau_R) = (1.5,-0.5)$
where $f$ has a subsumed homoclinic connection, Fig.~\ref{fig:subHCC}a.
In the region where $\eta > 0$ and $\nu > 0$, the map has an attractor, see Remark \ref{re:homoclinicCorners}.
This attractor is destroyed along $\eta = 0$ and $\nu = 0$
where the stable and unstable sets of $Y$ develop a non-trivial intersection.

Fig.~\ref{fig:bifSetBCNF}b shows that with
\eqref{eq:eta2}, \eqref{eq:nu2}, $\lambda = 0.5$, and $\sigma = 1.5$,
the bifurcation set of $h$ is broadly the same as that of $f$, and matching better nearer the codimension-two point.
Thus the one-dimensional family $h$ is able reproduce the general behaviour
of the two-dimensional normal form $f$ near the codimension-two subsumed homoclinic connection.
A more detailed comparison is provided in \S\ref{sec:fourStructures}.

\subsection{An example with $\delta_L = 0$}
\label{sub:exB}

Now consider $f$ with $\delta_L = 0$, in which case $\lambda = 0$ and $\sigma = \tau_L$.
Fig.~\ref{fig:schemexB} shows a typical phase portrait.
Here we illustrate the one-dimensional approximation with $m = 3$,
so $\eta = b \left( f_R^4(U) \right)$ and $\nu = b \left( f_R^3(U) \right)$.
By substituting \eqref{eq:fR3U} and \eqref{eq:fR4U} into \eqref{eq:b}, with also $\lambda = 0$, we obtain
\begin{align}
\eta &= \frac{\delta_R \left( \tau_R^2 + \tau_R - \delta_R + 1 \right)}{\sigma}
-\tau_R^3 - \tau_R^2 (\delta_R + 1) + \tau_R (\delta_R - 1) + \delta_R^2 \nonumber \\
&\quad+ \left( \tau_R^3 + \tau_R^2 + \tau_R - 2 \delta_R \tau_R - \delta_R + 1 \right) \sigma, \label{eq:eta3zero} \\
\nu &= \frac{\delta_R (\tau_R + 1)}{\sigma}
- \tau_R (\tau_R + \delta_R + 1)
+ \left( \tau_R^2 + \tau_R - \delta_R + 1 \right) \sigma. \label{eq:nu3zero}
\end{align}

Fig.~\ref{fig:bifSetBCNFexB}a shows a bifurcation set of $f$ with $\delta_L = 0$ and $\delta_R = 2$.
Again $f$ has an attractor where $\eta > 0$ and $\nu > 0$,
and again this attractor is destroyed along the curves $\eta = 0$ and $\nu = 0$.
These curves intersect at
$(\tau_L,\tau_R) = \left( 1 + \frac{1}{\sqrt{5}}, \frac{\sqrt{5} - 1}{2} \right) \approx (1.4472,0.6180)$,
where $f$ has a subsumed homoclinic connection with $f^3(U) = S$.

\begin{figure}[b!]
\begin{center}
\includegraphics[width=8cm]{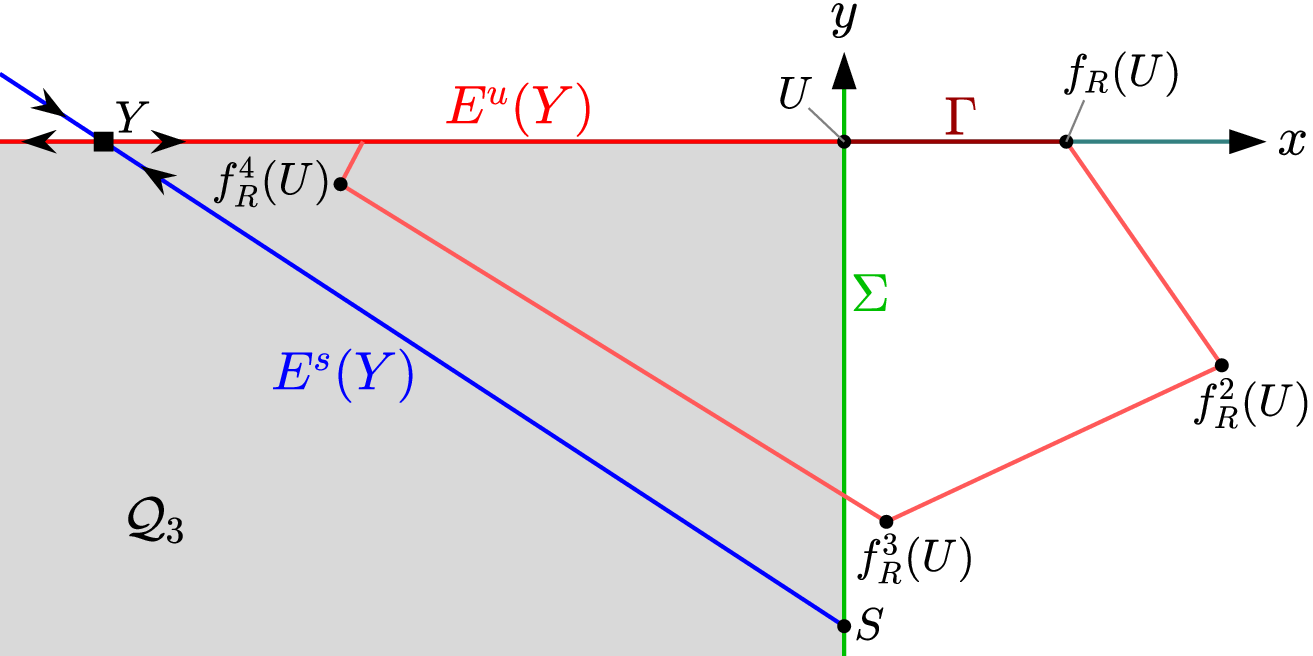}
\caption{
A sketch of the phase space of $f$ 
with $(\tau_L,\delta_L,\tau_R,\delta_R) = (1.3,0,0.7,2)$.
\label{fig:schemexB}
} 
\end{center}
\end{figure}

Fig.~\ref{fig:bifSetBCNFexB}b shows the corresponding bifurcation set of $h$.
As with the previous example, the bifurcation sets of $f$ and $h$ are similar with
differences increasing with the distance from the codimension-two point.
However $\lambda = 0$, so each branch of $h$ captures the first return map $F$ with no error.
Thus bifurcation curves where periodic solutions lose stability
(e.g.~the upper boundaries of $\cP_k'$ and $\cP_k$) follow the same paths in the two bifurcation sets.
This is because the non-zero stability multiplier of a periodic solution of $f$
is the slope of the corresponding branch, or composition of branches, of $h$.
However, border-collision bifurcation curves (e.g.,~the left and right boundaries of $\cP_k'$ and $\cP_k$)
where a periodic solution is destroyed by colliding
with the switching line are different in the two bifurcation sets.
This is because the values of $\ell(P)$ and $r(P)$ for points $P \in \cQ_3$
have a dependency on the value of $a(P)$
that is not incorporated into the one-dimensional approximation.

\begin{figure}[b!]
\begin{center}
\setlength{\unitlength}{1cm}
\begin{picture}(17,7)
\put(.7,0){\includegraphics[height=7cm]{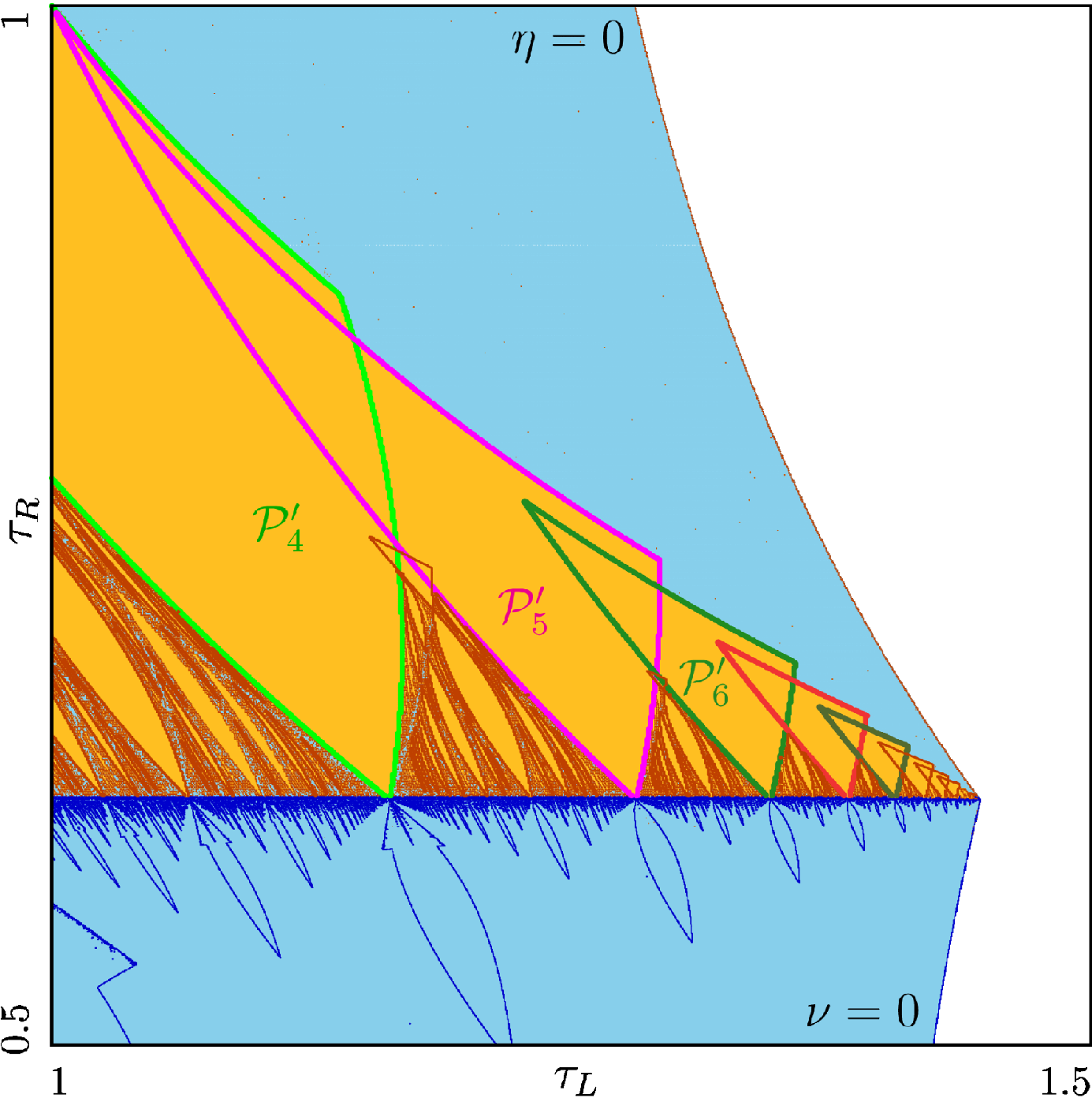}}
\put(10,0){\includegraphics[height=7cm]{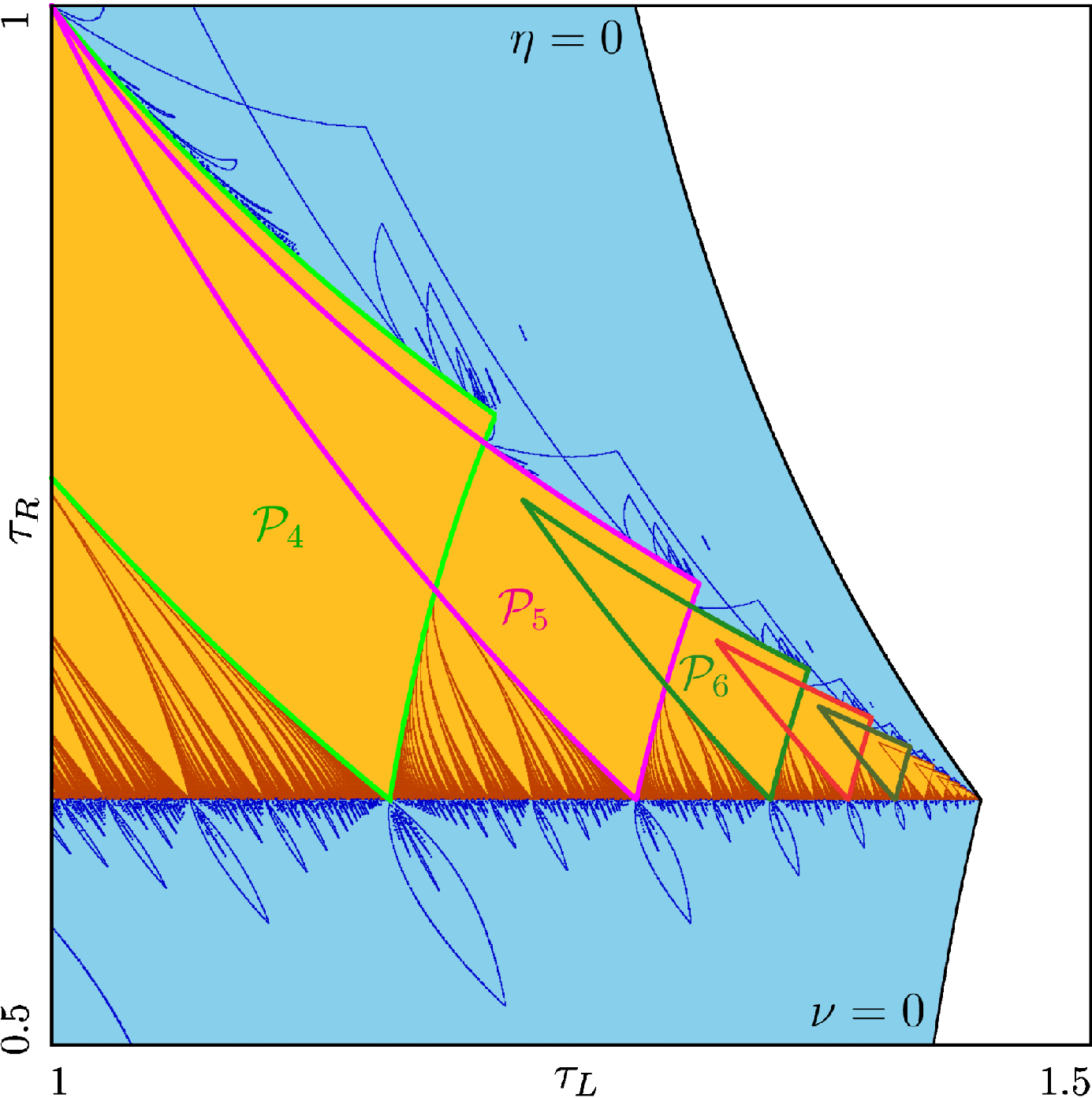}}
\put(0,6.6){{\bf a)}}
\put(9.3,6.6){{\bf b)}}
\end{picture}
\caption{
Panel (a) is a bifurcation set of $f$ with $\delta_L = 0$ and $\delta_R = 2$;
panel (b) is a bifurcation set of the corresponding one-dimensional family $h$
for which $\sigma = \tau_L$ and $\eta$ and $\nu$ are given by \eqref{eq:eta3zero} and \eqref{eq:nu3zero}.
The colour conventions are the same as in Fig.~\ref{fig:bifSetBCNF}.
In each $\cP_k'$ of panel (a), $f$ has a stable $L^k R^3$-cycle.
\label{fig:bifSetBCNFexB}
} 
\end{center}
\end{figure}

\subsection{An example for a saddle periodic solution}
\label{sub:exD}

Here we fix $\tau_R = -2.5$ and $\delta_R = 2$, and vary $\tau_L$ and $\delta_L$.
In this case $f$ has a subsumed homoclinic connection to a period-three solution ($RLR$-cycle)
when $\left( \tau_L, \delta_L \right) = \left( -\frac{23}{33}, \frac{13}{66} \right) \approx \left( -0.6970, 0.1970 \right)$ \cite{Si20}.
Fig.~\ref{fig:subHCC}b shows a phase portrait of $f$ at this point,
while Fig.~\ref{fig:schemexD} shows part of a phase portrait of $f$ at a typical nearby parameter point.
Fig.~\ref{fig:bifSetBCNFexD}a shows a bifurcation set of $f$,
while Fig.~\ref{fig:bifSetBCNFexD}b shows a bifurcation set of the corresponding subfamily of $h$.
Again, the bifurcation structure matches well with minor differences developing with the distance from the codimension-two point.

\begin{figure}[b!]
\begin{center}
\includegraphics[width=8cm]{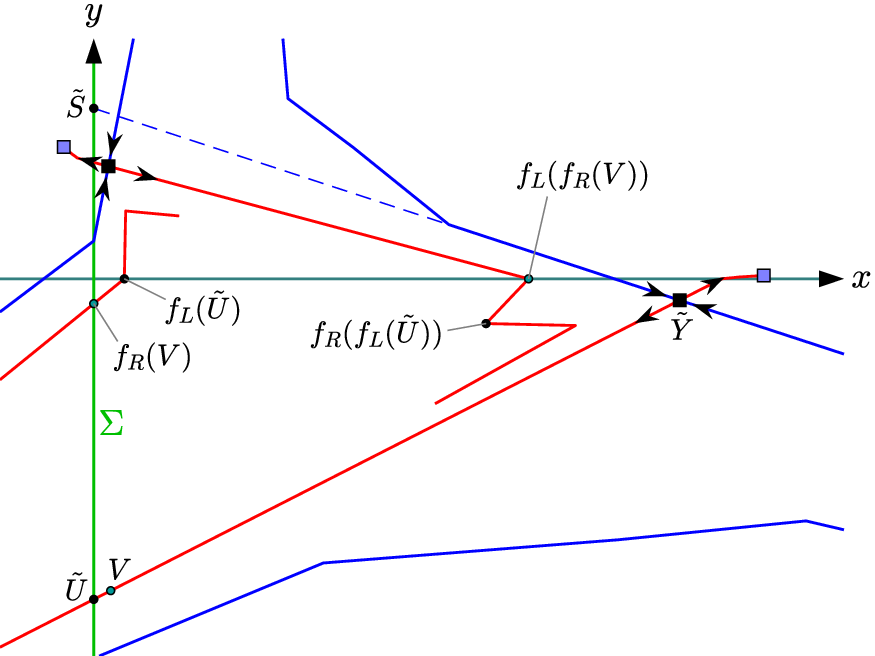}
\caption{
A sketch of the phase space of $f$ 
with $(\tau_L,\delta_L,\tau_R,\delta_R) = (-0.7,0.15,-2.5,2)$.
Two points of a saddle $RLR$-cycle are visible, shown with black squares, and the right-most point is labelled $\tilde{Y}$.
The stable and unstable sets of the $RLR$-cycle are coloured blue and red respectively.
Two points of a stable $RLL$-cycle are visible, shown with blue squares.
\label{fig:schemexD}
} 
\end{center}
\end{figure}

\begin{figure}[b!]
\begin{center}
\setlength{\unitlength}{1cm}
\begin{picture}(17,7)
\put(.7,0){\includegraphics[height=7cm]{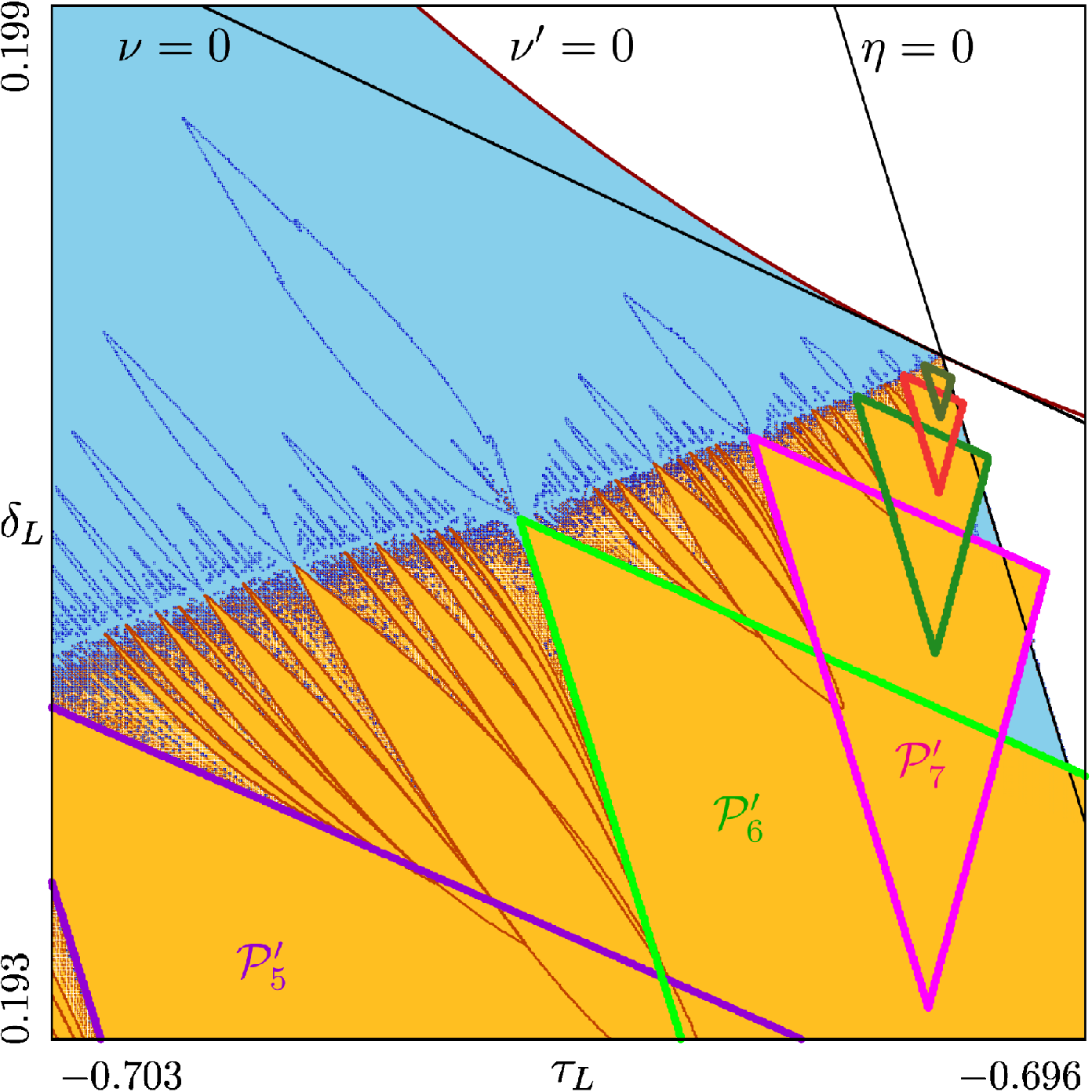}}
\put(10,0){\includegraphics[height=7cm]{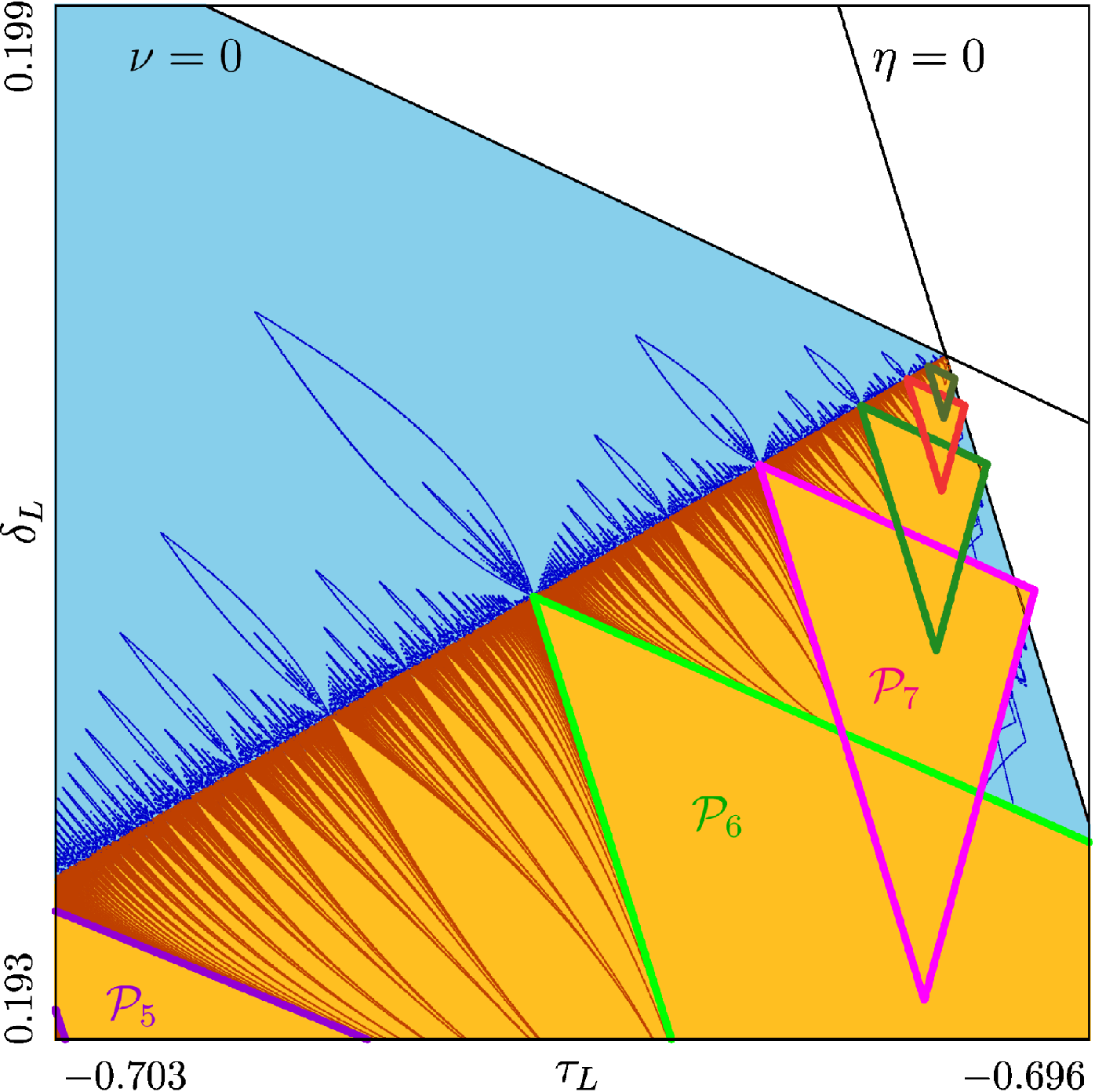}}
\put(0,6.6){{\bf a)}}
\put(9.3,6.6){{\bf b)}}
\end{picture}
\caption{
Panel (a) is a bifurcation set of $f$ with $\tau_R = -2.5$ and $\delta_R = 2$;
panel (b) is a bifurcation set of the corresponding one-dimensional family $h$
for which $\sigma$ is the unstable stability multiplier of the $RLR$-cycle,
and $\eta$ and $\nu$ are given by \eqref{eq:etaPeriod3} and \eqref{eq:nuPeriod3}.
The colour conventions are the same as in Fig.~\ref{fig:bifSetBCNF}.
In each $\cP_k'$ of panel (a), $f$ has a stable $(RLR)^k LR$-cycle.
\label{fig:bifSetBCNFexD}
} 
\end{center}
\end{figure}

We now outline the manner by which $\eta$, $\nu$, and $\sigma$ can be defined
for a general subsumed homoclinic connection of $f$, following \cite{Si20},
and apply this to the present example to explain how Fig.~\ref{fig:bifSetBCNFexD}b was produced.
We first pick a suitable point $\tilde{Y}$ in the saddle periodic solution associated with homoclinic connection.
We also identify an interval $\Gamma \subset E^u(\tilde{Y})$, with one endpoint on $\Sigma$, that forms a fundamental domain for
the branch of the unstable set of the periodic solution associated with the homoclinic connection.
By assumption, $\Gamma$ maps to $E^s(\tilde{Y})$ under some sequence of iterations of $f_L$ and $f_R$.
Using $(a,b)$ coordinates aligned with the stable and unstable directions of $\tilde{Y}$,
we define $\eta$ and $\nu$ to be $b$-coordinates of the images of the endpoints of $\Gamma$ under this sequence,
and let $\sigma$ be the unstable eigenvalue associated with the saddle periodic solution.

For the example here, we take $\tilde{Y}$ to be the right-most point of the $RLR$-cycle,
which is the unique fixed point of $f_R \circ f_L \circ f_R$.
Let $\tilde{S}$ and $\tilde{U}$ denote the intersections
of the stable and unstable subspaces $E^s(\tilde{Y})$ and $E^u(\tilde{Y})$ with $\Sigma$, see Fig.~\ref{fig:schemexD}.
Then let $\tilde{U}' = f_R(f_L(f_R(\tilde{U})))$ and
$\Gamma = \left\{ (1-\alpha) \tilde{U} + \alpha \tilde{U}' \,\middle|\, 0 \le \alpha < 1 \right\}$.
At the codimension-two point,
$f_R(f_L(\Gamma))$ belongs to $E^s(\tilde{Y})$
forming a subsumed homoclinic connection.
Thus
\begin{align}
\eta &= b \big( f_R(f_L(\tilde{U}')) \big), \label{eq:etaPeriod3} \\
\nu &= b \big( f_R(f_L(\tilde{U})) \big), \label{eq:nuPeriod3}
\end{align}
where, analogous to \eqref{eq:coordChange},
$(a,b)$-coordinates are defined by
\begin{equation}
P = \tilde{Y} + a (\tilde{S}-\tilde{Y}) + b (\tilde{U}-\tilde{Y}),
\label{eq:coordChangePeriod3}
\end{equation}
for any $P \in \mathbb{R}^2$.
Finally, $\sigma > 1$ is the unstable stability multiplier of the $RLR$-cycle
(an eigenvalue of $(\rD f_R)(\rD f_L)(\rD f_R)$).
For this example, $\sigma = \frac{13}{6}$ at the codimension-two point.

Unlike the previous two examples, the curves $\eta = 0$ and $\nu = 0$
are not bifurcations where the attractor of \eqref{eq:f} is destroyed.
Instead, the attractor is destroyed along the curve $\nu' = 0$, see Fig.~\ref{fig:bifSetBCNFexD}a,
where $\nu' = b \left( f_L(f_R(V)) \right)$
and $V$ is the unique point in $E^u(\tilde{Y})$ for which $f_R(V) \in \Sigma$.
The attractor is also destroyed along a similar curve close to $\eta = 0$
(not shown as it is almost indistinguishable from $\eta = 0$).
These curves are where the stable and unstable sets of the $RLR$-cycle
first attain a non-trivial intersection.

As a final remark, notice that the main periodicity regions $\cP_k'$ and $\cP_k$ overlap the curve $\eta = 0$.
This occurs because $\sigma > 2$, as shown in the next section.

\section{Basic properties of the one-dimensional family}
\label{sec:1d}

In this section we establish basic properties of the one-dimensional family $h$, given by \eqref{eq:h}.
We first establish a scaling property, whereby, as we step in parameter space towards $(\eta,\nu) = (0,0)$,
the dynamics is unchanged except occurs on different branches of $h$.
Invariant sets occur on an {\em absorbing interval} $J$ having endpoints $\eta$ and $\nu$,
and we characterise the number of branches $N$ that the map has over this interval.
This guides our more detailed analysis of $h$ in \S\ref{sec:fourStructures},
because where $N \le 2$ the dynamics of $h$ on $J$ are covered by existing theory \cite{AvGa19}.
Lastly in this section we analyse the fixed points of $h$.

\subsection{Scale invariance}
\label{sub:scaleInvariance}

\begin{figure}[b!]
\begin{center}
\includegraphics[width=9cm]{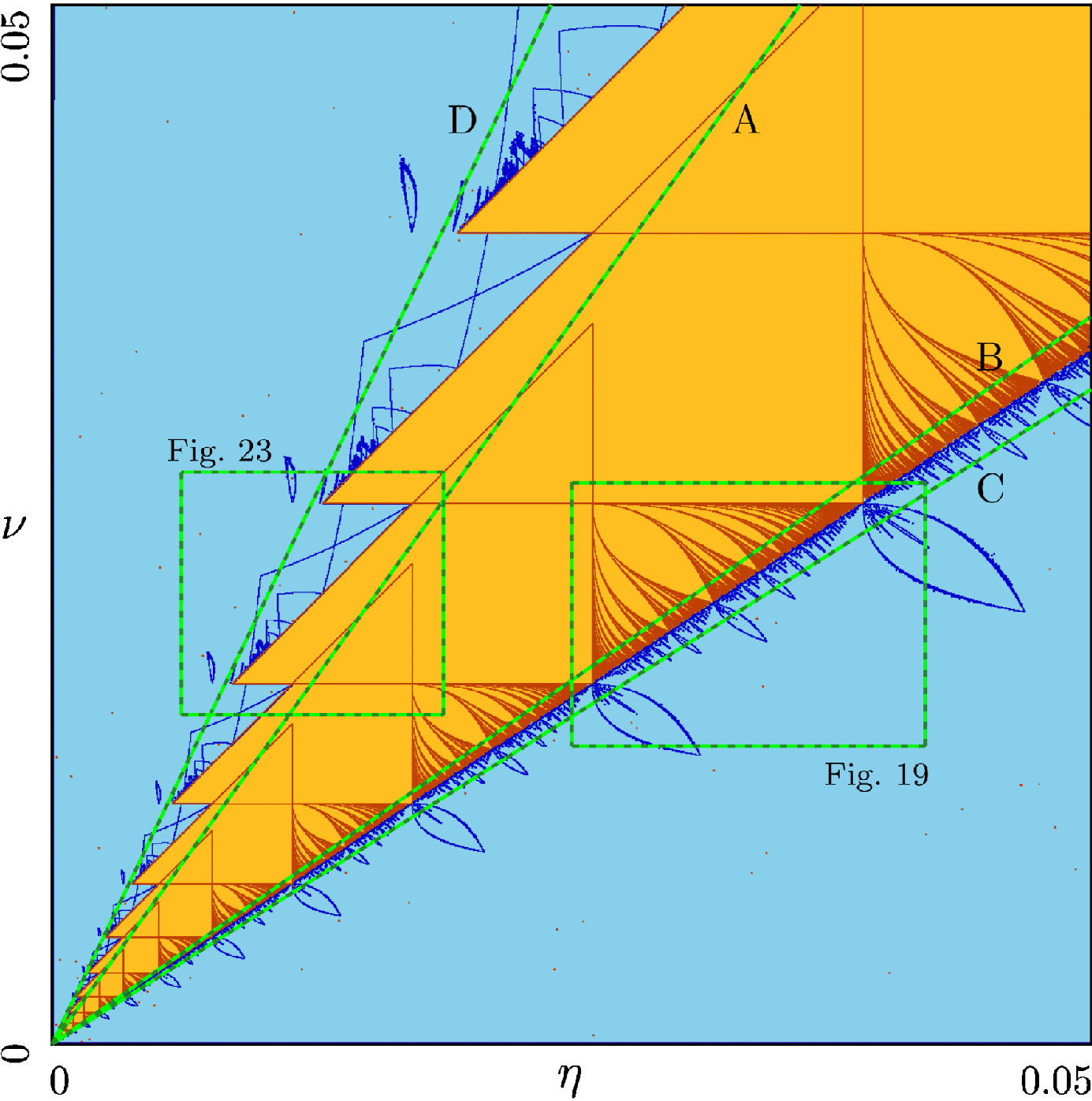}
\caption{
A bifurcation set of $h$ with $\sigma = 1.5$
using the same colour conventions as Fig.~\ref{fig:bifSetBCNF}.
The lines A--D correspond to the one-parameter slices
shown in Figs.~\ref{fig:periodInc}a, \ref{fig:periodAdd}a, \ref{fig:bandcountAdd}a, and \ref{fig:bandcountInc}a.
These are defined by:
A: $\frac{\nu}{\eta} = \frac{0.05}{0.0359} \approx 1.3928$;
B: $\frac{\nu}{\eta} = \frac{0.035}{0.05} = 0.7$;
C: $\frac{\nu}{\eta} = \frac{0.0315}{0.05} = 0.63$;
D: $\frac{\nu}{\eta} = \frac{0.05}{0.0262} \approx 1.9084$.
\label{fig:2D:num}
} 
\end{center}
\end{figure}

Fig.~\ref{fig:2D:num} is a bifurcation set of $h$ with $\sigma = 1.5$,
and reveals a pattern that repeats over smaller and smaller areas converging to $(\eta,\nu) = (0,0)$.
This occurs because $h$ enjoys the following scaling property.

\begin{proposition}
The map $h$ satisfies
\begin{equation}
h \left( \tfrac{z}{\sigma}; \tfrac{\eta}{\sigma}, \tfrac{\nu}{\sigma}, \sigma \right) = \tfrac{1}{\sigma} h(z;\eta,\nu,\sigma), \qquad
\text{for all $z > 0$, $\eta, \nu \in \mathbb{R}$, and $\sigma > 1$}.
\label{eq:scaleInvariance}
\end{equation}
\label{pr:scaleInvariance}
\end{proposition}

\begin{proof}
Fix $z > 0$, $\eta, \nu \in \mathbb{R}$, and $\sigma > 1$,
and let $k \in \mathbb{Z}$ be such that $z \in I_k$.
The right-hand side of \eqref{eq:scaleInvariance}
is $\frac{1}{\sigma} h_k(z;\eta,\nu,\sigma)$,
while the left-hand side of \eqref{eq:scaleInvariance}
is $h_{k+1} \left( \frac{z}{\sigma}; \frac{\eta}{\sigma}, \frac{\nu}{\sigma}, \sigma \right)$.
By \eqref{eq:hk} these are identical.
\end{proof}

Thus if $\mathcal{X} \subset (0,\infty)$ is an orbit of $h(z;\eta,\nu,\sigma)$,
then $\frac{\mathcal{X}}{\sigma}$ is an orbit of $h \left( z; \frac{\eta}{\sigma}, \frac{\nu}{\sigma}, \sigma \right)$.
So when $\eta$ and $\nu$ are divided by $\sigma$,
the dynamics are unchanged except they occur at values of $z$ scaled by $\frac{1}{\sigma}$.
These dynamics occur on the branches of $h$ shifted one place to the left,
so correspond to values of $k$ that are increased by one.

In the context of the dynamics of $f$,
each increment in the value of $k$ corresponds to additional iterate of $f$ close to the saddle (see Remark \ref{re:k}).
More precisely, if $h(z;\eta,\nu;\sigma)$ has a period-$M$ solution corresponding to a period-$N$ solution of $f$,
then for any $j \ge 1$ the map $h \left( z; \frac{\eta}{\sigma^j}, \frac{\nu}{\sigma^j}, \sigma \right)$
has a period-$M$ solution corresponding to a period-$(j M + N)$ solution of $f$.

\subsection{Number of branches over the absorbing interval}
\label{sub:numberOfBranches}

If $\eta > \nu$ then each piece $h_k$ is increasing,
while if $\eta < \nu$ then each $h_k$ is decreasing.
Each $h_k$ has the same range, namely the set
\begin{equation}
J = \begin{cases} (\eta,\nu], & \eta < \nu, \\
\{ \eta \}, & \eta = \nu, \\
[\nu, \eta), & \eta > \nu.
\end{cases}
\label{eq:absorbingInterval}
\end{equation}
Any compact neighbourhood of $J$ in $(0,\infty)$ is a {\em trapping region} of $h$ \cite{Ro04,Ru17},
while if $\eta \ne \nu$, then $J$ is an {\em absorbing interval} \cite[pg.~12]{AvGa19}. 

Since any invariant set of $h$ must belong to $J$,
it suffices to consider $h$ on $J$.
Let $N(\eta,\nu,\sigma)$ be the number of intervals $I_k$
that have a non-empty intersection with $J$.
Fig.~\ref{fig:2D:count} shows how $N$ varies over the bifurcation set,
and the following result provides a formula for $N$.
At most parameter points $N$ is simply
the value of $k$ at ${\rm max}[\eta,\nu]$,
minus the value of $k$ at ${\rm min}[\eta,\nu]$, plus one,
where the $k$-values are given explicitly by the formula \eqref{eq:kFormula}.
An adjustment is needed to handle special cases where
$\frac{-\ln(\eta)}{\ln(\sigma)}$ or $\frac{-\ln(\nu)}{\ln(\sigma)}$ are integers.

\begin{figure}[b!]
\begin{center}
\includegraphics[width=9cm]{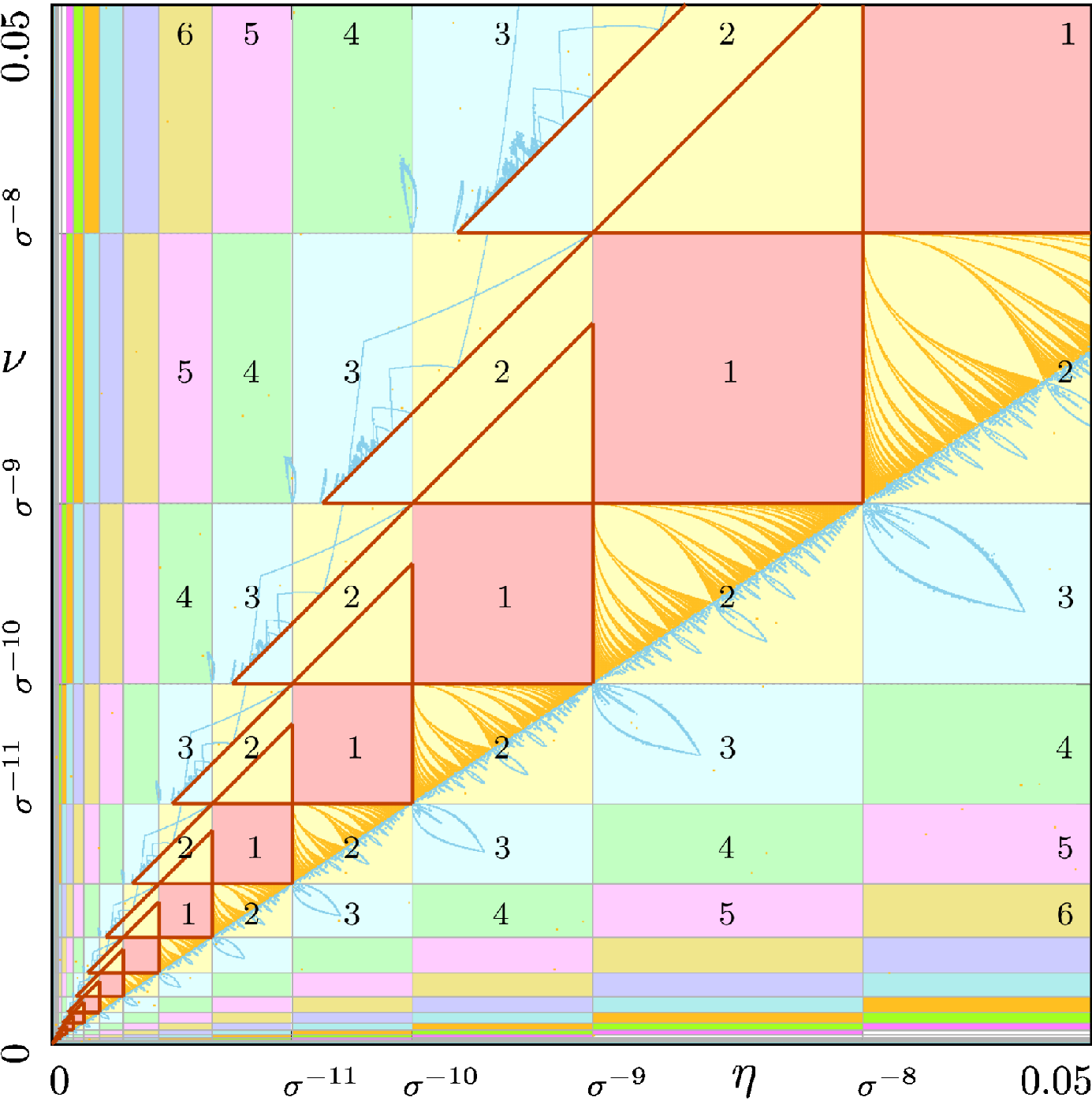}
\caption{
A division of the $(\eta,\nu)$-plane
by the number of pieces $N$ of $h$ over the absorbing interval $J$.
For example $N=1$ in the central pink rectangles.
This diagram was drawn using $\sigma = 1.5$ and we have overlaid the bifurcation curves of Fig.~\ref{fig:2D:num}.
\label{fig:2D:count}
} 
\end{center}
\end{figure}

\begin{proposition}
For any $\eta > 0$, $\nu > 0$ and $\sigma > 1$,
\begin{equation}
N(\eta,\nu,\sigma) = \begin{cases}
\left\lceil \frac{-\ln(\eta)}{\ln(\sigma)} \right\rceil - \left\lceil \frac{-\ln(\nu)}{\ln(\sigma)} \right\rceil + 1,
& \eta \le \nu, \\[2mm]
\left\lceil \frac{-\ln(\nu)}{\ln(\sigma)} \right\rceil - \left\lfloor \frac{-\ln(\eta)}{\ln(\sigma)} \right\rfloor,
& \eta > \nu.
\end{cases}
\label{eq:N}
\end{equation}
\label{pr:N}
\end{proposition}

\begin{proof}
We have $N = k_{\rm max} - k_{\rm min} + 1$, where
$k_{\rm max}$ is the $k$-value of the left-most branch of $h$ over $J$,
and $k_{\rm min}$ is the $k$-value of the right-most branch of $h$ over $J$.

First suppose $\eta > \nu$.
Then $J = [\nu,\eta)$, hence $k_{\rm max} = K(\nu,\sigma) = \left\lceil \frac{-\ln(\nu)}{\ln(\sigma)} \right\rceil$.
Also $k_{\rm min} = \lim_{z \to \eta^-} K(z,\sigma)$,
so if $\frac{-\ln(\eta)}{\ln(\sigma)} \notin \mathbb{Z}$
then $k_{\rm min} = K(\eta,\sigma) = \left\lceil \frac{-\ln(\nu)}{\ln(\sigma)} \right\rceil$,
which is the same as $\left\lfloor \frac{-\ln(\eta)}{\ln(\sigma)} \right\rfloor + 1$.
If instead $\frac{-\ln(\eta)}{\ln(\sigma)} \in \mathbb{Z}$,
then $k_{\rm min} = K(\eta,\sigma) + 1$,
so again $k_{\rm min} = \left\lfloor \frac{-\ln(\eta)}{\ln(\sigma)} \right\rfloor + 1$ as required.

Now suppose $\eta < \nu$, so then $J = (\eta,\nu]$.
Then $k_{\rm min} = K(\nu,\sigma)$ and $k_{\rm max} = \lim_{z \to \eta^+} K(z,\sigma)$.
In this case $k_{\rm max} = K(\eta,\sigma)$ regardless of whether or not
the value of $\frac{-\ln(\eta)}{\ln(\sigma)}$ is an integer, verifying \eqref{eq:N} in this case.
Finally, \eqref{eq:N} holds when $\eta = \nu$, because here $J$ is a single point.
\end{proof}

\subsection{Fixed points}
\label{sub:fixedPoints}

The slope of each $h_k$ is
\begin{equation}
s_k = \frac{\eta - \nu}{\sigma - 1} \,\sigma^k.
\label{eq:sk}
\end{equation}
If $s_k \ne 1$ then the fixed point equation $z = h_k(z)$ has the unique solution
\begin{equation}
z_k^* = \frac{-\eta + \sigma \nu}{\sigma - 1 - (\eta - \nu) \sigma^k}.
\label{eq:zkStar}
\end{equation}
This is a fixed point of $h$ if it belongs to the interval $I_k$,
with stability governed by the slope $s_k$.
These considerations lead to the following result
(which ignores the degenerate case $\eta = \nu$ for which $z_k^* = \eta$ for all $k \in \mathbb{Z}$).

\begin{proposition}
Consider $\eta, \nu \in \mathbb{R}$ with $\eta \ne \nu$, $\sigma > 1$, and $k \in \mathbb{Z}$.
Then $z_k^*$ is an asymptotically stable fixed point of $h$ if and only if $(\eta,\nu) \in \cP_k$, where
\begin{equation}
\cP_k = \left\{ (\eta,\nu) \in \mathbb{R}^2 \,\middle|\,
\eta < \sigma^{-k+1},\, \sigma^{-k} < \nu < \eta + \sigma^{-k+1} - \sigma^{-k} \right\}.
\label{eq:Pk}
\end{equation}
\label{pr:Pk}
\end{proposition}

\begin{proof}
If $z_k^*$ belongs to an endpoint of $I_k$,
then $z_k^*$ cannot be Lyapunov stable (so certainly not asymptotically stable)
because \eqref{eq:h} is discontinuous at the endpoints of $I_k$.
Thus $z_k^*$ is an asymptotically stable fixed point of $h$ if and only if it belongs to the interior of $I_k$ and $|s_k| < 1$.
If $s_k \ge 1$, then by \eqref{eq:sk} $(\eta,\nu) \notin \cP_k$, so it remains to consider $s_k < 1$.
In this case $z_k^* \in {\rm int}(I_k)$ is equivalent to
$\frac{-\eta + \sigma \nu}{\sigma - 1 - (\eta - \nu) \sigma^k} > \sigma^{-k}$ and
$\frac{-\eta + \sigma \nu}{\sigma - 1 - (\eta - \nu) \sigma^k} < \sigma^{-k+1}$,
which are equivalent to $\nu > \sigma^{-k}$ and $\eta < \sigma^{-k+1}$ respectively.
Also $s_k > -1$ is equivalent to $\nu < \eta + \sigma^{-k+1} - \sigma^{-k}$.
\end{proof}

For any fixed $\sigma > 1$, each $\cP_k$ is a triangle in the $(\eta,\nu)$-plane. 
The edges $\eta = \sigma^{-k+1}$ and $\nu = \sigma^{-k}$
are border-collision bifurcations where $z_k^*$ meets an endpoint of $I_k$.
Along the third edge $\nu = \eta + \sigma^{-k+1} - \sigma^{-k}$
the fixed point $z_k^*$ loses stability because its stability multiplier becomes $-1$.
This edge does not correspond to a flip (or period-doubling) bifurcation because $h_k$ is linear
so cannot support an isolated period-two solution.
Instead it is a {\em degenerate flip bifurcation},
which is a standard bifurcation for piecewise-linear families
that typically generates a multi-band chaotic attractor \cite{AvGa19}.

The bottom-right vertex of $\cP_k$ is $\left( \sigma^{-k+1}, \sigma^{-k} \right)$,
the top vertex of $\cP_k$ is $\left( \sigma^{-k+1}, (2 \sigma - 1) \sigma^{-k} \right)$,
and the left vertex of $\cP_k$ is $\left( (2-\sigma) \sigma^{-k}, \sigma^{-k} \right)$.
Thus if $1 < \sigma < 2$, as in Fig.~\ref{fig:2D:count},
each $\cP_k$ is contained in the first quadrant of the $(\eta,\nu)$-plane,
while if $\sigma > 2$, as in Fig.~\ref{fig:bifSetBCNFexD}b,
the left vertices of $\cP_k$ protrude over $\eta = 0$.

The next result characterises intersections between the triangles $\cP_k$.

\begin{proposition}
Let $\sigma > 1$ and $k_1, k_2 \in \mathbb{Z}$ be distinct.
Then $\cP_{k_1} \cap \cP_{k_2} \ne \varnothing$ if and only if $|k_1 - k_2| = 1$.
\label{pr:triangleIntersections}
\end{proposition}

\begin{proof}
For any $k \in \mathbb{Z}$ the top-right corner of $\cP_{k+1}$
is $(\eta,\nu) = \left( \sigma^{-k}, (2 \sigma - 1) \sigma^{-k-1} \right)$.
By \eqref{eq:Pk} this point lies in the interior of $\cP_k$,
thus $\cP_{k+1} \cap \cP_k \ne \varnothing$.
The top-right corner of $\cP_{k+j}$
is $(\eta,\nu) = \left( \sigma^{-k-j+1}, (2 \sigma - 1) \sigma^{-k-j} \right)$.
For all $j \ge 2$,
\begin{equation}
(2 \sigma - 1) \sigma^{-k-j} \le (2 \sigma - 1) \sigma^{-k-2}
= -(\sigma-1)^2 \sigma^{-k-2} + \sigma^{-k}
< \sigma^{-k},
\nonumber
\end{equation}
so this corner lies below the bottom edge $\nu = \sigma^{-k}$ of $\cP_k$,
thus $\cP_{k+j} \cap \cP_k = \varnothing$.
\end{proof}

We now prove that $h$ has no other attractors at parameter points in $\cP_k$.
In the special case that $(\eta,\nu) \in \cP_k$ belongs to the
$\nu = \eta + \sigma^{-k} - \sigma^{-k-1}$ boundary of $\cP_{k+1}$,
the map has an uncountable family of period-two solutions in $J \cap I_{k+1}$,
but these do not produce an attractor.

\begin{proposition}
Let $\sigma > 1$, $k \in \mathbb{Z}$, and $(\eta,\nu) \in \cP_k$.
The only possible topological attractors of $h$ are the fixed points
$z_{k-1}^*$, $z_k^*$, and $z_{k+1}^*$.
\label{pr:fpGlobalAttractor}
\end{proposition}

\begin{proof}
If $\eta \ge \sigma^{-k}$ and $\nu < \sigma^{-k+1}$,
then $N(\eta,\nu,\sigma) = 1$ by \eqref{eq:N}.
In this case $h|_J$ is affine with asymptotically stable fixed point $z_k^*$
so all forward orbits of $h$ converge to $z_k^*$.

If $\nu \ge \sigma^{-k+1}$ then $N(\eta,\nu,\sigma) = 2$ and $(\eta,\nu) \in \cP_{k-1}$.
In this case $h$ is affine and forward invariant on $J \cap I_k$ with asymptotically stable fixed point $z_k^*$,
and affine and forward invariant on $J \cap I_{k-1}$ with asymptotically stable fixed point $z_{k-1}^*$.
Thus every forward orbit of $h$ enters one of these intervals and converges to $z_k^*$ or $z_{k-1}^*$.

Finally suppose $\eta < \sigma^{-k}$.
Again $h$ is affine and forward invariant on $J \cap I_k$ and all forward orbits in $J \cap I_k$ converge to $z_k^*$.
If $(\eta,\nu) \in \cP_{k+1}$,
then $N(\eta,\nu,\sigma) = 2$ and
$z_k^*$ and $z_{k+1}^*$ are asymptotically stable and the only attractors of $h$
by the previous argument (with $k+1$ in place of $k$).
If $\nu = \eta + \sigma^{-k} - \sigma^{-k-1}$,
then $z_{k+1}^*$ is neutrally stable and surrounded by period-two solutions.
Again $N(\eta,\nu,\sigma) = 2$ so these solutions fill $J \setminus I_k$, hence $z_k^*$ is the only attractor.
If $(\eta,\nu) \notin \cP_{k+1}$,
then $N(\eta,\nu,\sigma) \ge 2$ and every branch in $J$ besides $h_k$ has slope less than $-1$, so is expanding.
In this case for any open interval $I \subset J$
the sets $I, h(I), h^2(I), \ldots$ cannot all be contained in an interval where $h$ is expanding,
thus there exists $p \ge 0$ such that either $h^p(I) \cap I_k \ne \varnothing$,
or $h^p(I)$ contains a point of discontinuity of $h$ and $h^{p+1}(I) \cap I_k \ne \varnothing$.
Thus there exists $z \in I$ whose forward orbit reaches $I_k$, and hence converges to $z_k^*$.
Since $I$ is an arbitrary open subset of $J$, this shows that $z_k^*$ is the only attractor of $h$.
\end{proof}

Finally observe that at parameter points in $\cP_k$ sufficiently close to the left corner of $\cP_k$,
the map has $N \ge 3$ branches over $J$.
Each branch has a fixed point in $J$, so the stable fixed point $z_k^*$ coexists
with two or more unstable fixed points.
This fixed points may have homoclinic or heteroclinic connections, in which case
$z_k^*$ coexists with a chaotic repeller.

\section{Four bifurcation structures}
\label{sec:fourStructures}

In this section we explore the bifurcation structure of
the one-dimensional family $h$, given by \eqref{eq:h},
and compare it to the bifurcation structure of the two-dimensional
family $f$, given by \eqref{eq:f}, near subsumed homoclinic connections.

For $h$, there are four basic bifurcation structures and these
are realised by one-parameter bifurcation diagrams
defined by fixing $\sigma$ and the slope $\frac{\nu}{\eta}$.
Examples are provided by the slices through the bifurcation set Fig.~\ref{fig:2D:num} that correspond to:
(A) period-incrementing, \S\ref{sub:periodInc};
(B) period-adding, \S\ref{sub:periodAdd};
(C) bandcount-adding, \S\ref{sub:bandcountAdd}; and
(D) bandcount-incrementing, \S\ref{sub:bandcountInc}.
In each slice the structure repeats on a smaller and smaller scale as $(\eta,\nu) \to (0,0)$
due to the scaling property, Proposition \ref{pr:scaleInvariance}.
For $f$, the bifurcation set Fig.~\ref{fig:bifSetBCNFZoom} is a magnification of Fig.~\ref{fig:bifSetBCNF}
and indicates four analogous slices.

\begin{figure}[b!]
\begin{center}
\includegraphics[height=7cm]{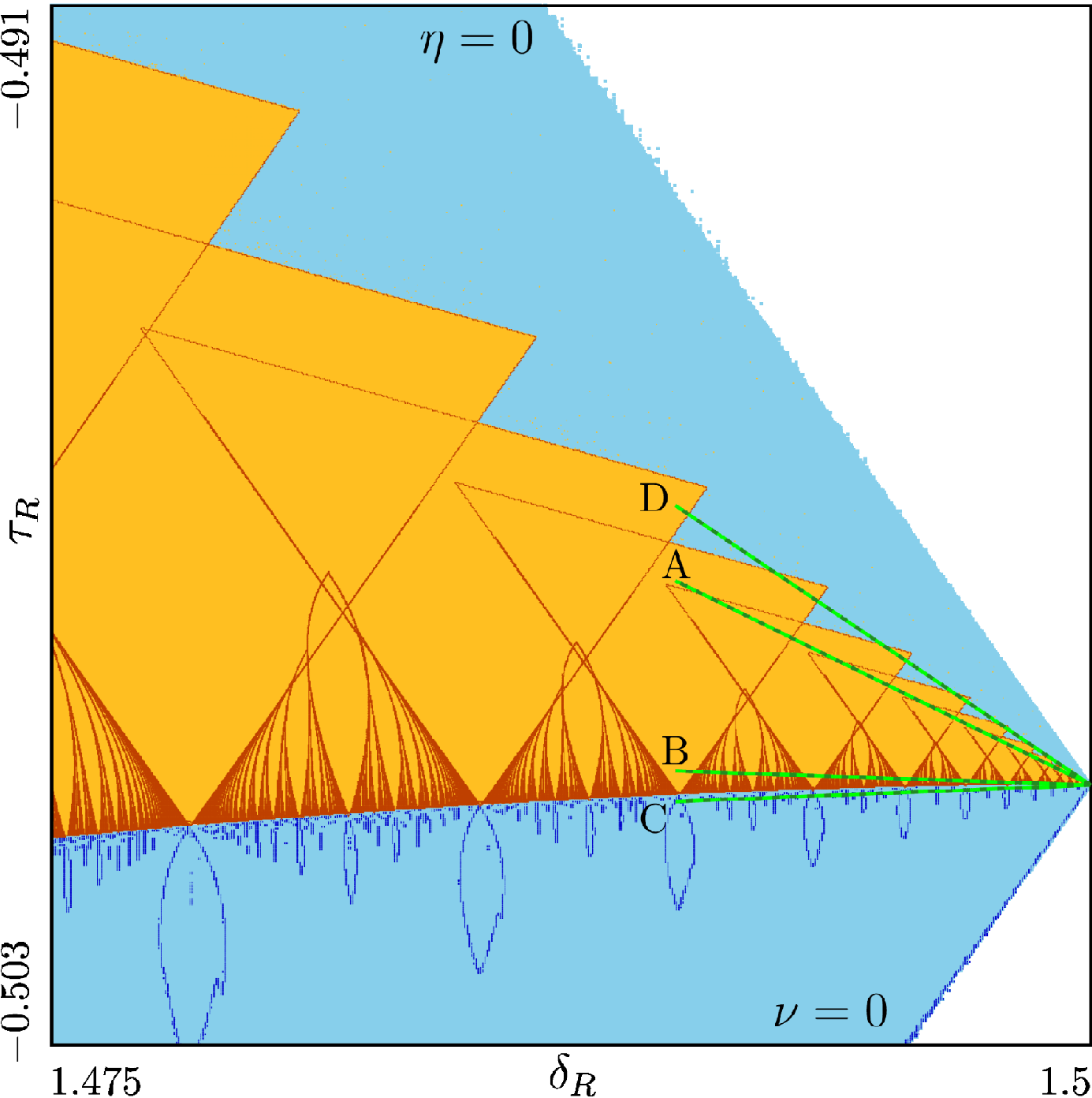}
\caption{
A magnification of Fig.~\ref{fig:bifSetBCNF} showing the bifurcation set of the two-dimensional
border-collision normal form $f$ near the codimension-two point $(\delta_R,\tau_R) = (1.5,-0.5)$.
The lines A--D indicate the one-parameter slices
shown in Figs.~\ref{fig:periodInc}b, \ref{fig:periodAdd}b, \ref{fig:bandcountAdd}b, and \ref{fig:bandcountInc}b.
These correspond approximately to the slices A--D of Fig.~\ref{fig:2D:num} for the one-dimensional family $h$.
\label{fig:bifSetBCNFZoom}
} 
\end{center}
\end{figure}

\subsection{Period-incrementing ($1 < \frac{\nu}{\eta} < \sigma$)}
\label{sub:periodInc}

Here we suppose the slope $\frac{\nu}{\eta}$ belongs to the interval $(1,\sigma)$.
In this case $h$ has either one piece or two pieces over the absorbing interval $J$ by Proposition \ref{pr:N}.
Every point in the slice belongs to a triangle $\cP_k$ \eqref{eq:Pk}
where $h$ has a fixed point $z_k^*$.
Thus the slice shows a sequence of intervals $M_k$ where it intersects $\cP_k$.
For the example shown in Fig.~\ref{fig:periodInc}a we have
$M_k = \left( \xi_k^\ell, \theta_k \right)$,
where $\xi_k^\ell$ is the border-collision bifurcation $\nu = \sigma^{-k}$
corresponding to the bottom edge of $\cP_k$,
and $\theta_k$ is the degenerate flip bifurcation $\nu = \eta + \sigma^{-k+1} - \sigma^{-k}$
corresponding to the diagonal edge of $\cP_k$.
This occurs because $\frac{\nu}{\eta} \in \left( 2 - \frac{1}{\sigma}, \sigma \right)$,
and beyond $\theta_k$ the fixed point persists but is unstable.
If instead $\frac{\nu}{\eta} \in \left( 1, 2 - \frac{1}{\sigma} \right)$
then the right endpoint of each $M_k$ is the border-collision bifurcation $\eta = \sigma^{-k+1}$, denoted $\xi_k^r$,
corresponding to the right edge of $\cP_k$.
Each interval $M_k$ overlaps $M_{k-1}$ and $M_{k+1}$, by Proposition \ref{pr:triangleIntersections}.
The bifurcation diagram shows no other attractors by Proposition \ref{pr:fpGlobalAttractor}.

\begin{figure}[b!]
\begin{center}
\setlength{\unitlength}{1cm}
\begin{picture}(17,4.3)
\put(.7,0){\includegraphics[height=4.3cm]{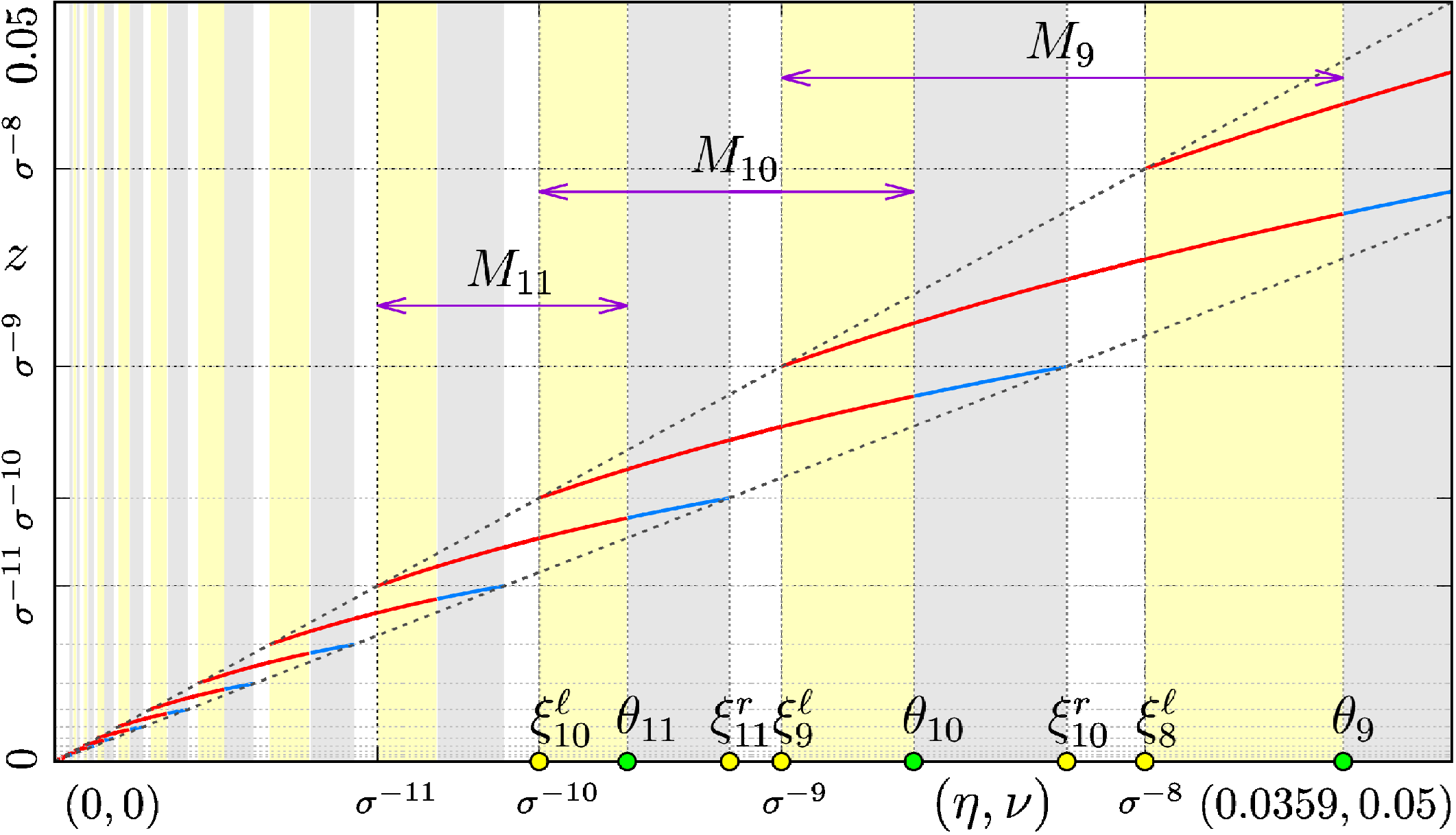}}
\put(9.5,0){\includegraphics[height=4.3cm]{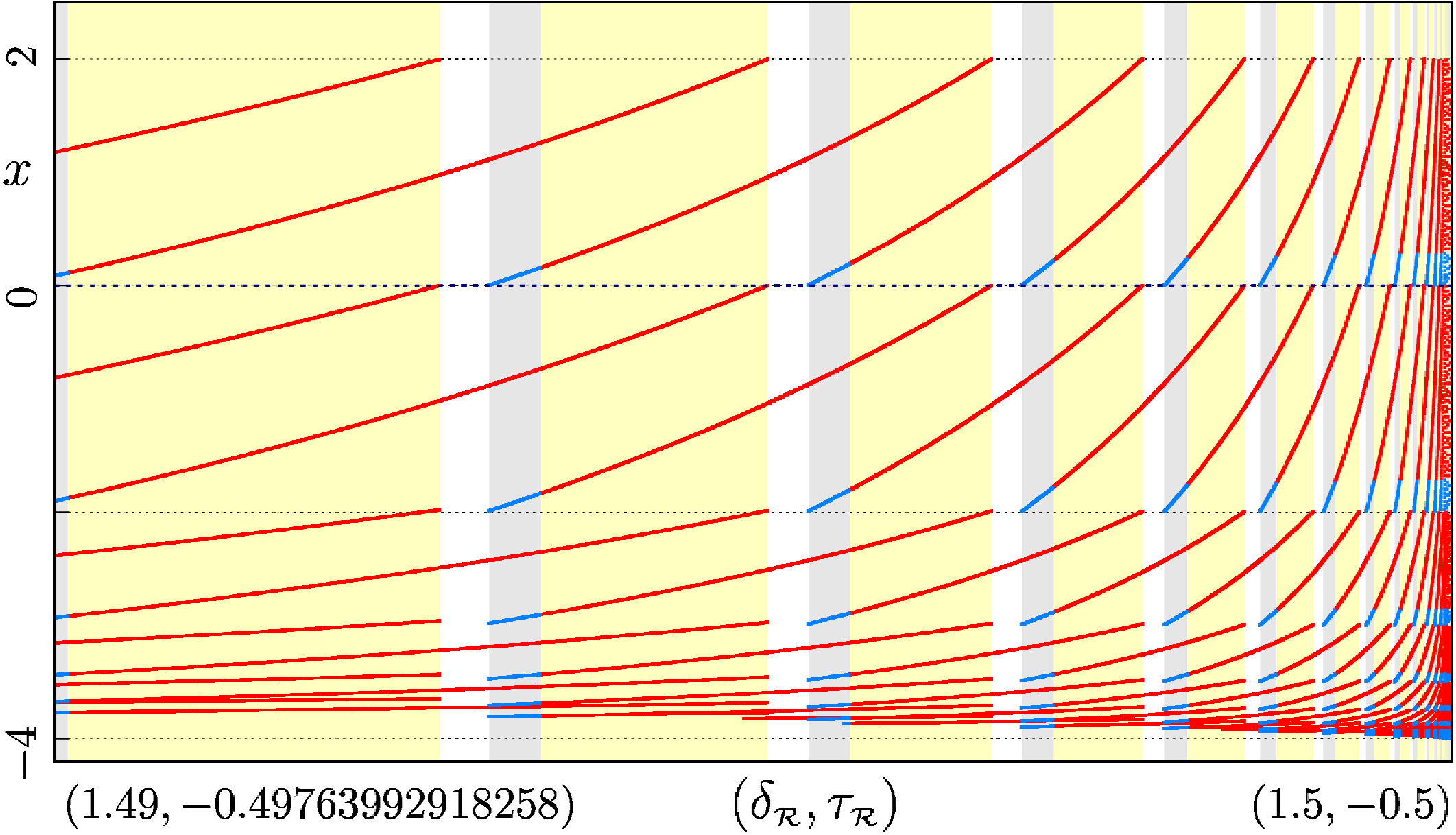}}
\put(0,3.9){{\bf a)}}
\put(8.8,3.9){{\bf b)}}
\end{picture}
\caption{
Bifurcation diagrams of $h$ (panel (a)) and $f$ (panel (b)) corresponding to the
lines A of Figs.~\ref{fig:2D:num} and \ref{fig:bifSetBCNFZoom} respectively.
In (a) the attractors are the fixed points $z_k^*$ (coloured red where they are stable and blue where they are unstable).
In yellow intervals two stable fixed points coexist,
in grey intervals two fixed points coexist but only one is stable,
and in white intervals $h$ has one fixed point (stable).
In (b) the attractors are $L^k R^2$-cycles and the same colour scheme is used.
We have only plotted points for which $y < {\rm max} \left[ -\frac{3 x}{4}, 0 \right]$ so that the diagram is not too cluttered.
\label{fig:periodInc}
} 
\end{center}
\end{figure}

Fig.~\ref{fig:periodInc}b shows the corresponding bifurcation diagram for $f$.
Recall, each $z_k^*$ corresponds to a period-$(k+m)$ solution of $f$, and for this example $m=2$.
These solutions exist and are stable on intervals $M_k'$ corresponding to the intervals $M_k$ of panel (a).
Thus as we move left to right across the bifurcation diagram the period of the solution sequentially increases by one.
For this reason this bifurcation structure is referred to as {\em period-incrementing}.

The intervals $M_k'$ are small perturbations of the intervals $M_k$,
thus for sufficiently large values of $k$ each $M_k'$ overlaps $M_{k-1}'$ and $M_{k+1}'$. 
For smaller values of $k$, i.e.~further from the codimension-two point,
other overlapping arrangements are possible.
For example at parameter point (i) in Fig.~\ref{fig:bifSetBCNF}a,
the regions $\cP_9'$, $\cP_{10}'$, and $\cP_{11}'$ overlap.
Fig.~\ref{fig:basins}a provides a phase portrait showing the corresponding
stable $L^k R^2$-cycles (for $k = 7,8,9$)
and their basins of attraction.

\begin{figure}[b!]
\begin{center}
\includegraphics[width=17cm]{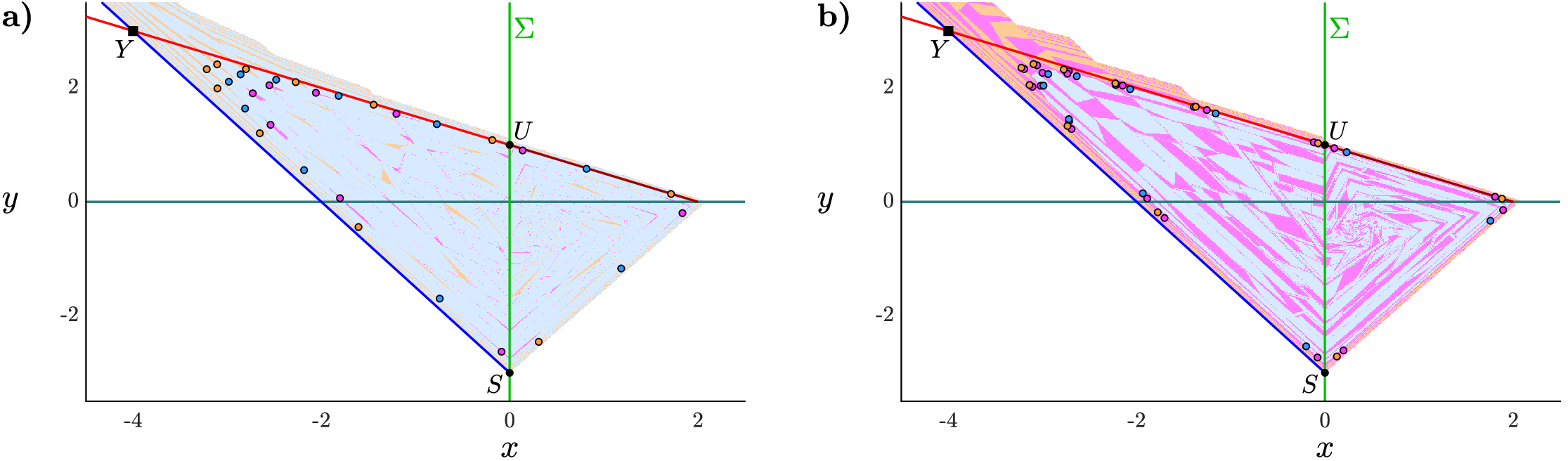}
\caption{
Phase portraits of $f$ with $(\tau_L,\delta_L,\tau_R,\delta_R) = (2,0.75,-0.484,1.433)$ in panel (a)
and $(\tau_L,\delta_L,\tau_R,\delta_R) = (2,0.75,-0.494,1.443)$ in panel (b).
These parameter points are labelled (i) and (ii) respectively in Fig.~\ref{fig:bifSetBCNF}.
In (a) $f$ has a stable $L^7 R^2$-cycle (pink),
a stable $L^8 R^2$-cycle (blue),
and a stable $L^9 R^2$-cycle (orange).
In (b) $f$ has a stable $L^8 R^2$-cycle (blue),
a stable $L^9 R^2$-cycle (orange),
and a stable $L^8 R^2 L^9 R^2$-cycle (pink).
Both panels include numerically computed basins of attraction.
\label{fig:basins}
} 
\end{center}
\end{figure}

Proposition \ref{pr:fpGlobalAttractor}, which concerns the absence of other attractors,
does not easily extend from $h$ to $f$.
For example the parameter point labelled (ii) in Fig.~\ref{fig:bifSetBCNF}a
belongs to $\cP_{10}'$ and $\cP_{11}'$,
but in addition to the corresponding stable $L^8 R^2$ and $L^9 R^2$-cycles,
there also exists a stable $L^8 R^2 L^9 R^2$-cycle, see Fig.~\ref{fig:basins}b.

The period-incrementing structure
is common for families of one-dimensional piecewise-smooth maps with two pieces,
one piece with positive slope and the other piece with negative slope \cite{GrAl17}.
Here the incrementing structure arises
in a scenario where all pieces of the map have negative slope \cite[\S 8.3.2]{AvGa19}.

\subsection{Period-adding ($\frac{1}{\sigma} < \frac{\nu}{\eta} < 1$)}
\label{sub:periodAdd}

Now suppose the slope $\frac{\nu}{\eta}$ belongs to the interval $\left( \frac{1}{\sigma}, 1 \right)$.
Again the number of branches $N$ that $h$ has over $J$ is either $1$ or $2$ by Proposition \ref{pr:N}.
Intervals where $N = 1$ correspond to intersections with exactly one $\cP_k$,
where the attractor is unique and equal to the fixed point $z_k^*$.
Intervals where $N = 2$ intersect none of the $\cP_k$, and $h$ appears as in Fig.~\ref{fig:circleMaps}a.
The two pieces of $h$ are $h_k$ and $h_{k+1}$, where $k$ is such that $\sigma^{-k} \in (\nu,\eta)$.
These pieces have positive slope, and it is important to consider the value
\begin{equation}
\Delta = h_{k}(\eta) - h_{k+1}(\nu),
\label{eq:Delta}
\end{equation}
which is the difference between the value of $h$ at the left and right endpoints of $J$.
By \eqref{eq:hk},
\begin{equation}
\Delta = \frac{(\eta - \nu)(\eta - \sigma \nu)}{\sigma - 1},
\label{eq:Delta2}
\end{equation}
so with $\frac{\nu}{\eta} \in \left( \frac{1}{\sigma}, 1 \right)$ we have $\Delta < 0$.
In this case $h|_J$ is a one-to-one function, as in Fig.~\ref{fig:circleMaps}a.

\begin{figure}[b!]
\begin{center}
\includegraphics[width=11cm]{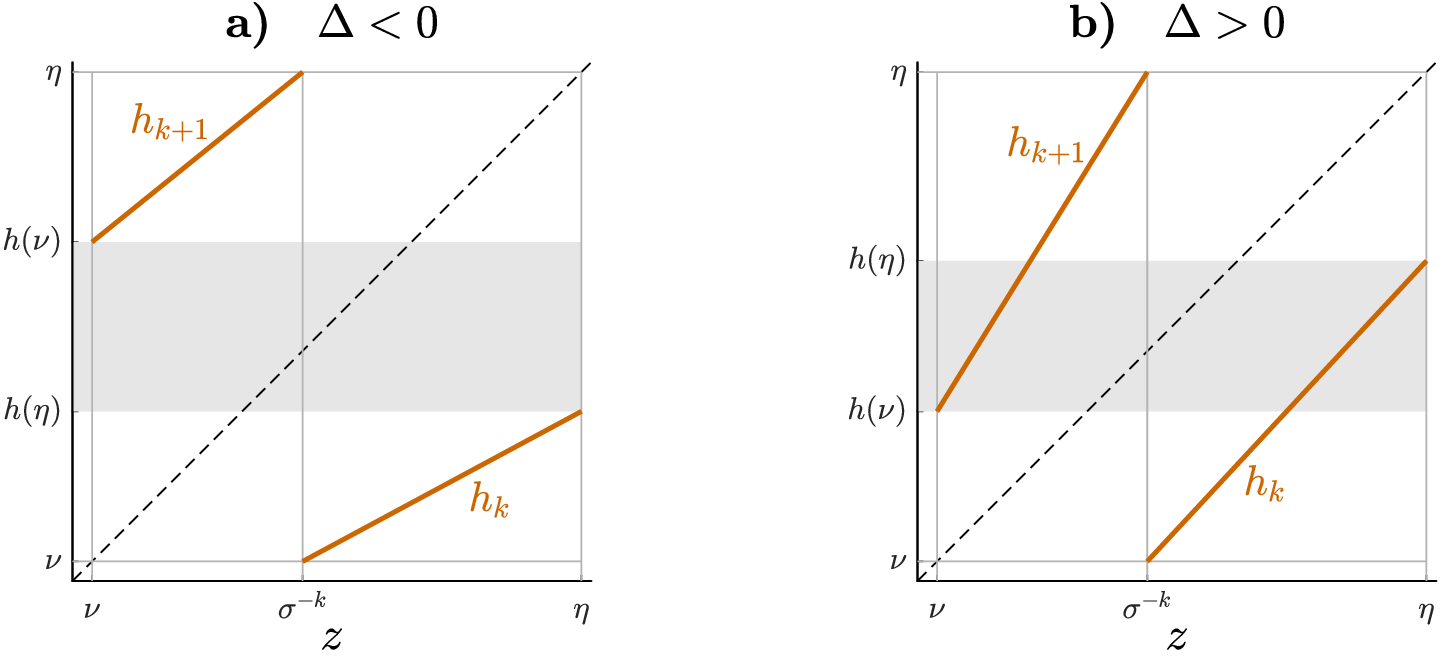}
\caption{
Sketches of the map $h$ when $\eta > \nu$ and $N = 2$.
In (a) the quantity $\Delta = h(\eta) - h(\nu)$ is negative and $h|_J$ is one-to-one but not onto,
while in (b) $\Delta$ is positive and $h|_J$ is onto but not one-to-one.
In the context of circle maps,
$h$ is {\em non-overlapping} if $\Delta < 0$
and {\em overlapping} if $\Delta > 0$ \cite{Ke80}.
\label{fig:circleMaps}
} 
\end{center}
\end{figure}

With $\Delta < 0$, the dynamics of $h|_J$ are non-chaotic
(in general, either periodic or quasiperiodic), while if $\Delta > 0$ its dynamics are chaotic.
This is because for two-piece piecewise-linear maps the invertibility condition $\Delta = 0$
is precisely the boundary between non-chaotic and chaotic dynamics, 
see \cite{BrSt91}, \cite[pg.~186]{DiBu08}, or \cite[pg.~393]{AvGa19}.
We stress that this is specific to the piecewise-linear setting;
for typical smooth or piecewise-smooth families of maps the boundary of chaos is far more complex \cite{MaTr86}.

With $\Delta < 0$, the map $h$ has stable periodic solutions
and quasiperiodic Cantor set attractors that form the structure shown in Fig.~\ref{fig:periodAdd}a.
With $h|_J$ treated as a circle map, it has a unique {\em rotation number}, $\rho$ \cite{GrAl17,RhTh86}.
If $\rho$ is rational, specifically $\rho = \frac{u}{q} \in (0,1)$ where $\frac{u}{q}$ is irreducible,
then $h$ has a stable period-$q$ solution with $u$ points in $I_k$.
If $\rho$ is irrational, then
$h$ has a Cantor set attractor given by the closure of quasiperiodic orbits.
Importantly, $\rho$ varies continuously with parameters \cite[Thm.~5.8]{RhTh91},
so between intervals where $\rho = \frac{u_1}{q_1}$ and $\rho = \frac{u_2}{q_2}$,
if $\frac{u_1}{q_1}$ and $\frac{u_2}{q_2}$ are {\em Farey neighbours}
then there exists an interval where $\rho = \frac{u_1 + u_2}{q_1 + q_2}$ (the {\em Farey sum}).
In this way the periods `add'
and the bifurcation structure is referred to as period-adding, see \cite{AvGa19,GrAl17} for details.

\begin{figure}[b!]
\begin{center}
\setlength{\unitlength}{1cm}
\begin{picture}(17,4.3)
\put(.7,0){\includegraphics[height=4.3cm]{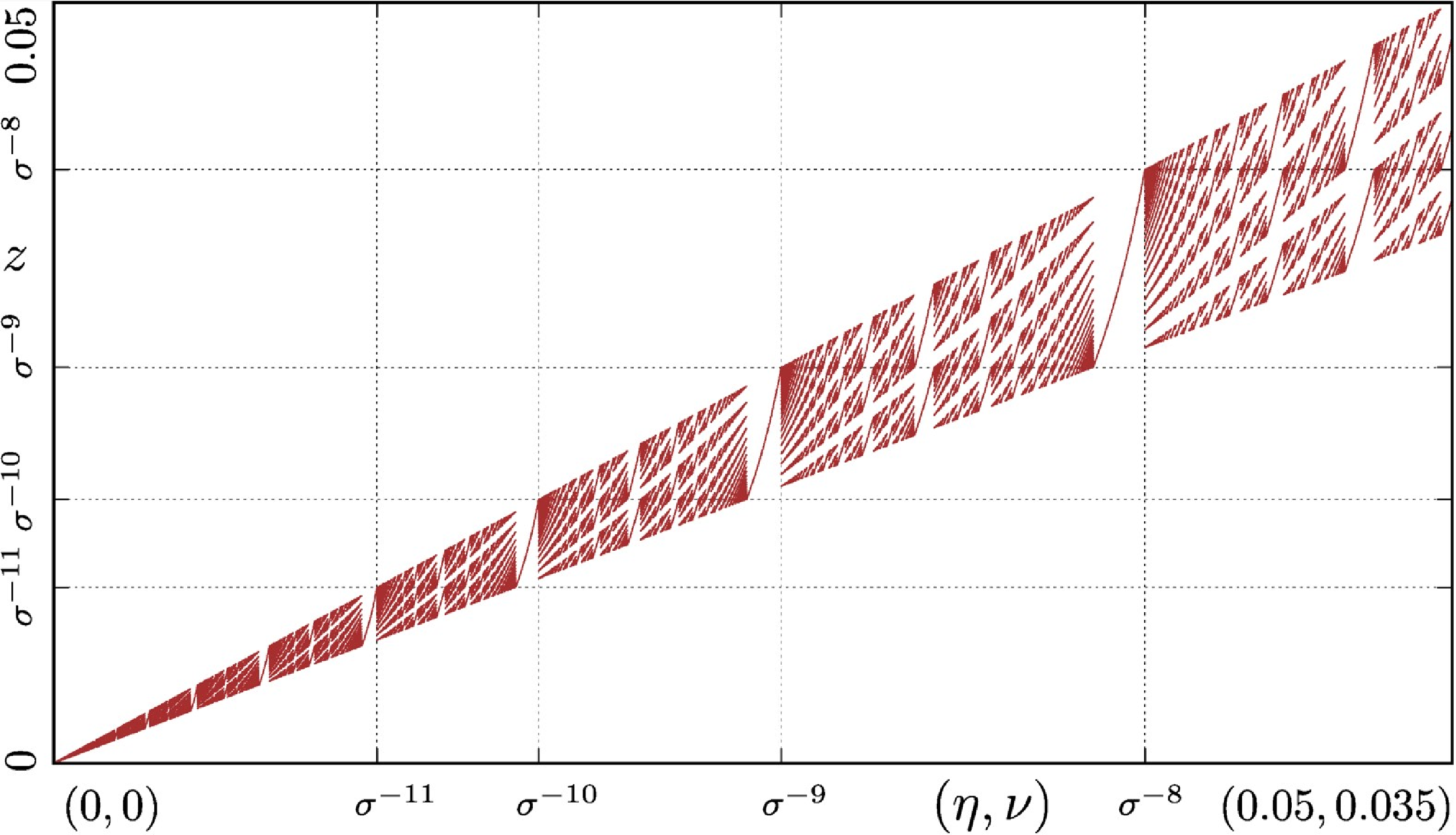}}
\put(9.5,0){\includegraphics[height=4.3cm]{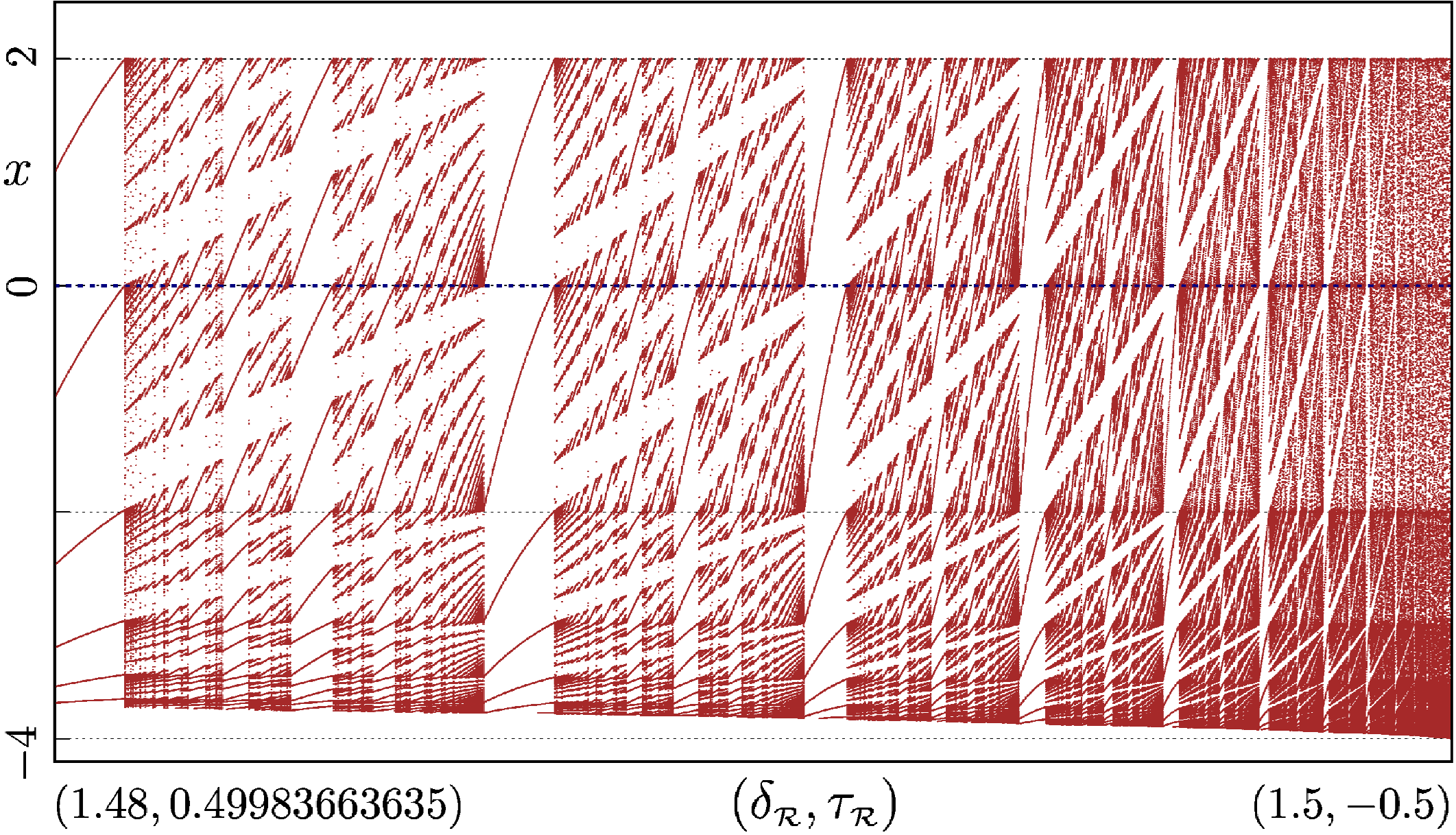}}
\put(0,3.9){{\bf a)}}
\put(8.8,3.9){{\bf b)}}
\end{picture}
\caption{
Bifurcation diagrams of $h$ (panel (a)) and $f$ (panel (b)) corresponding to the
lines B of Figs.~\ref{fig:2D:num} and \ref{fig:bifSetBCNFZoom} respectively.
For clarity panel (b) only shows points with $y < {\rm max} \left[ -\frac{3 x}{4}, 0 \right]$.
\label{fig:periodAdd}
} 
\end{center}
\end{figure}

The rotation number $\rho$ varies monotonically from
the limit $\rho = 0$ at $\eta = \sigma^{-k}$ to the limit $\rho = 1$ at $\nu = \sigma^{-k}$.
Thus over every interval where $N = 2$ there is a {\em full} period-adding structure.

Fig.~\ref{fig:Y} shows a magnification of Fig.~\ref{fig:2D:num} focussing on one box where $N = 2$ with $\eta > \nu$.
Rational values of $\rho$ occur in `tongues' that issue from the codimension-two point
$(\eta,\nu) = \left( \sigma^{-k}, \sigma^{-k} \right)$,
described more generally in \cite[\S 8.1.1]{AvGa19}.
The boundaries of the tongues are border-collision bifurcations
where one point of the periodic solution reaches the discontinuity $z = \sigma^{-k}$.

\begin{figure}[b!]
\begin{center}
\includegraphics[width=8cm]{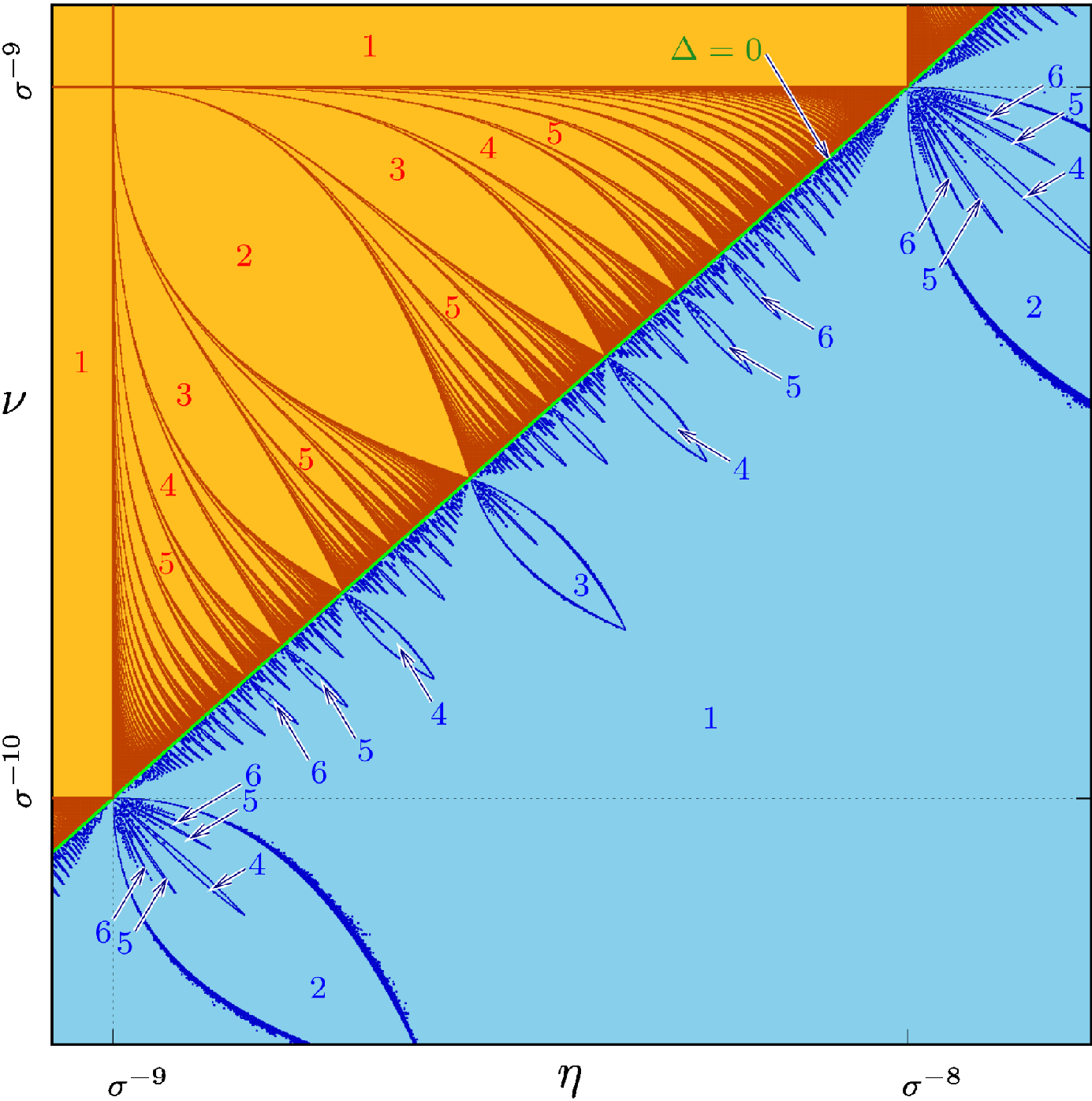}
\caption{
A magnification of the bifurcation set Fig.~\ref{fig:2D:num} of $h$ with $\sigma = 1.5$.
Throughout the rectangle $\sigma^{-9} < \eta < \sigma^{-8}$, $\sigma^{-10} < \nu < \sigma^{-9}$
the map $h$ has $N = 2$ pieces over $J$
and each piece is increasing, as in Fig.~\ref{fig:circleMaps}.
The yellow regions are numbered by the period of the corresponding stable periodic solution,
while the blue regions are numbered by the number of connected components of the chaotic attractor.
\label{fig:Y}
} 
\end{center}
\end{figure}

Now see Fig.~\ref{fig:periodAdd}b for the corresponding bifurcation structure for $f$.
Where $\rho = \frac{u}{q} \in \mathbb{Q}$ is irreducible,
the map $h$ has a stable periodic solution with $u$ points in $I_k$ and $q-u$ points in $I_{k+1}$,
thus the corresponding periodic solution of $f$ has period
\begin{align}
p &= u (k+m) + (q-u) (k+m+1), \nonumber \\
&= q (k+m+1) - u, \nonumber
\end{align}
where $m = 2$ for this example.

In the bifurcation set for $f$ (Fig.~\ref{fig:bifSetBCNFZoom}),
the tongues have the same appearance as in the bifurcation set of $h$ (Fig.~\ref{fig:2D:num}),
except they overlap slightly giving rise to coexisting attractors.
This may be due to the presence of two nested invariant circles in the phase space of $f$, see \cite{AvZh19},
but more analysis is needed.
In our example with $\delta_L = 0$, see Fig.~\ref{fig:bifSetBCNFexB}a,
the tongues instead develop shrinking points near their intersections with $\cP_k'$.
For both the coexisting attractors and the shrinking points,
we believe these occur arbitrarily close to the codimension-two points, $(\eta,\nu) = (0,0)$.

\subsection{Bandcount-adding ($\frac{\nu}{\eta} < \frac{1}{\sigma}$)}
\label{sub:bandcountAdd}

Next we consider $\frac{\nu}{\eta} < \frac{1}{\sigma}$.
With also $\frac{\nu}{\eta} > \frac{1}{\sigma^2}$, as in Fig.~\ref{fig:bandcountAdd}a,
the bifurcation diagram has intervals over which $N = 2$.
Here again $h|_J$ can be treated as a circle map,
but by \eqref{eq:Delta2} we now have $\Delta > 0$, so $h|_J$ is not a one-to-one function, see Fig.~\ref{fig:circleMaps}b.
In this case the orbits of $h$ in $J$ do not all share the same rotation number,
instead there is a {\em rotation interval} \cite{Ke80,Mi86b},
and the map has a chaotic attractor.

This attractor is an interval or a union of disjoint intervals.
As parameters are varied, the number of intervals changes,
producing the so-called {\em bandcount-adding} structure of Fig.~\ref{fig:bandcountAdd}a \cite{AvSc08}.
In the bifurcation set shown in Fig.~\ref{fig:Y},
the number of intervals is greater than one in the lobes
that issue from the line $\frac{\nu}{\eta} = \frac{1}{\sigma}$.
For any $\rho = \frac{u}{q} \in (0,1)$, the corresponding periodicity tongue (yellow) gives rise at
$\frac{\nu}{\eta} = \frac{1}{\sigma}$ to a lobe (blue) where $h$ has an attractor consisting of $q+1$ or more disjoint intervals.
The boundaries of the lobe are where this attractor collides with an unstable period-$q$ solution
that has a positive stability multiplier.
This is a type of homoclinic bifurcation known as an {\em interior crisis} \cite{GrOt83},
and may more specifically be called an {\em expansion bifurcation} because the size of the attractor increases suddenly.
Within each lobe there is also a nested substructure of smaller lobes where the attractor
is comprised of yet larger numbers of disjoint intervals \cite[\S 6.4]{AvGa19}.

\begin{figure}[b!]
\begin{center}
\setlength{\unitlength}{1cm}
\begin{picture}(17,4.3)
\put(.7,0){\includegraphics[height=4.3cm]{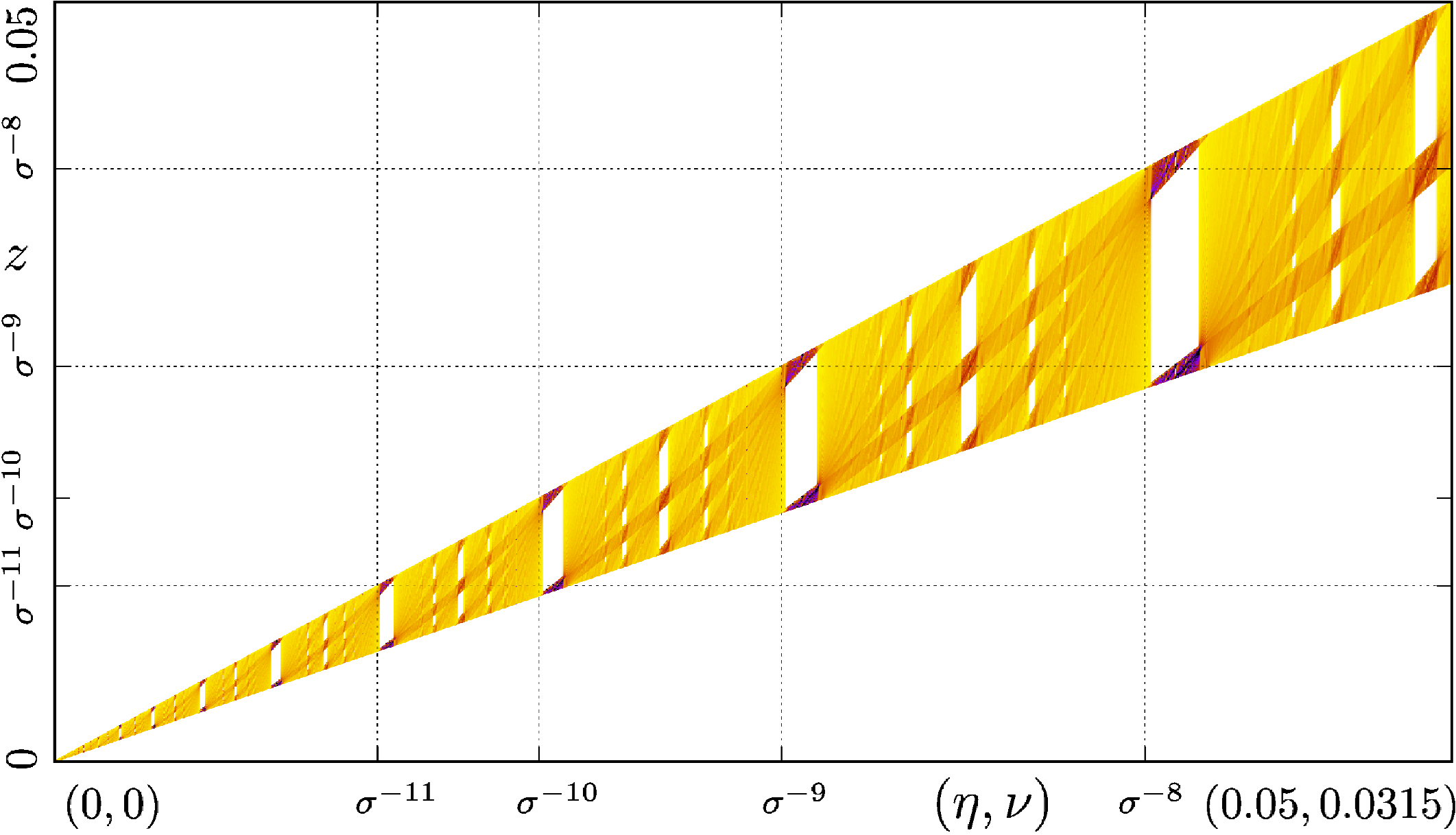}}
\put(9.5,0){\includegraphics[height=4.3cm]{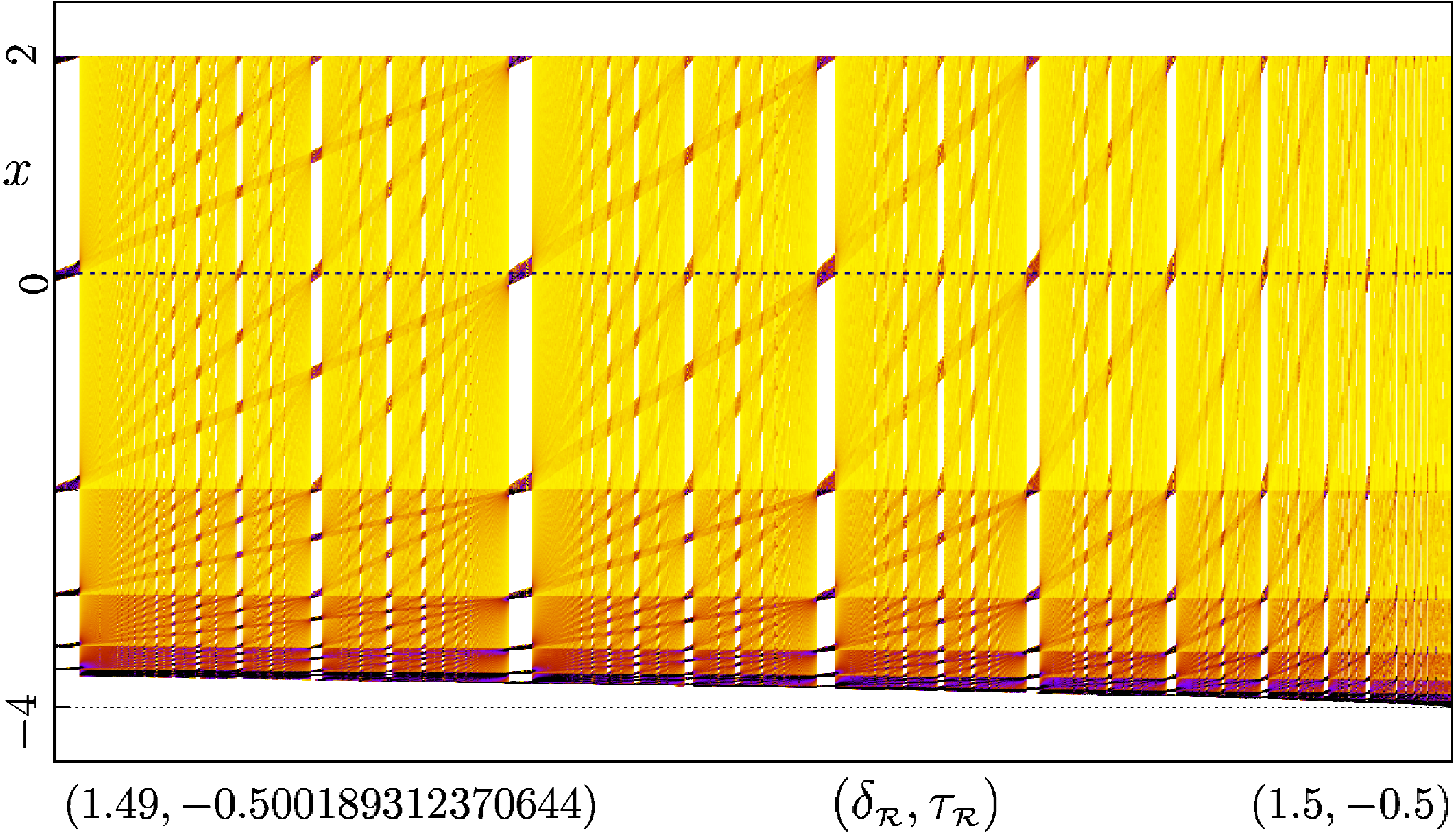}}
\put(0,3.9){{\bf a)}}
\put(8.8,3.9){{\bf b)}}
\end{picture}
\caption{
Bifurcation diagrams of $h$ (panel (a)) and $f$ (panel (b)) corresponding to the
lines C of Figs.~\ref{fig:2D:num} and \ref{fig:bifSetBCNFZoom} respectively.
Points are coloured according to density of the attractor.
For clarity panel (b) only shows points with $y < {\rm max} \left[ -\frac{3 x}{4}, 0 \right]$.
\label{fig:bandcountAdd}
} 
\end{center}
\end{figure}

With $\frac{\nu}{\eta} < \frac{1}{\sigma}$ the bifurcation diagram
includes parameter combinations for which the number of branches $N$ of the map over $J$ is greater than two.
To our knowledge the bifurcation structures of families of such maps has not been explored before,
but if each piece of $h_k$ has slope greater than $1$, then
one can apply the classical theory of piecewise-expanding maps \cite{LaYo73,LiYo78}
to obtain some information on the nature of the attractors.

Numerically we observe that from the lower-right corner of each $\cP_k$
there issues a lobe into the adjoining rectangle where $N = 3$
throughout which $h$ has an attractor comprised of two or more disjoint intervals, e.g.~Fig.~\ref{fig:bandexA}c.
The boundaries of these lobes are expansion bifurcations
where the attractor collides with the fixed point $z_k^*$, which here is unstable.
Each lobe has a nested substructure of smaller lobes.
The largest nested region seems to be confined by expansion bifurcation curves
of a period-two solution of $h$ that has points in $I_{k-1}$ and $I_{k+1}$.
It remains for future work to characterise the expansion bifurcations
that form the remainder of the substructure.

\begin{figure}[b!]
\begin{center}
\includegraphics[width=17cm]{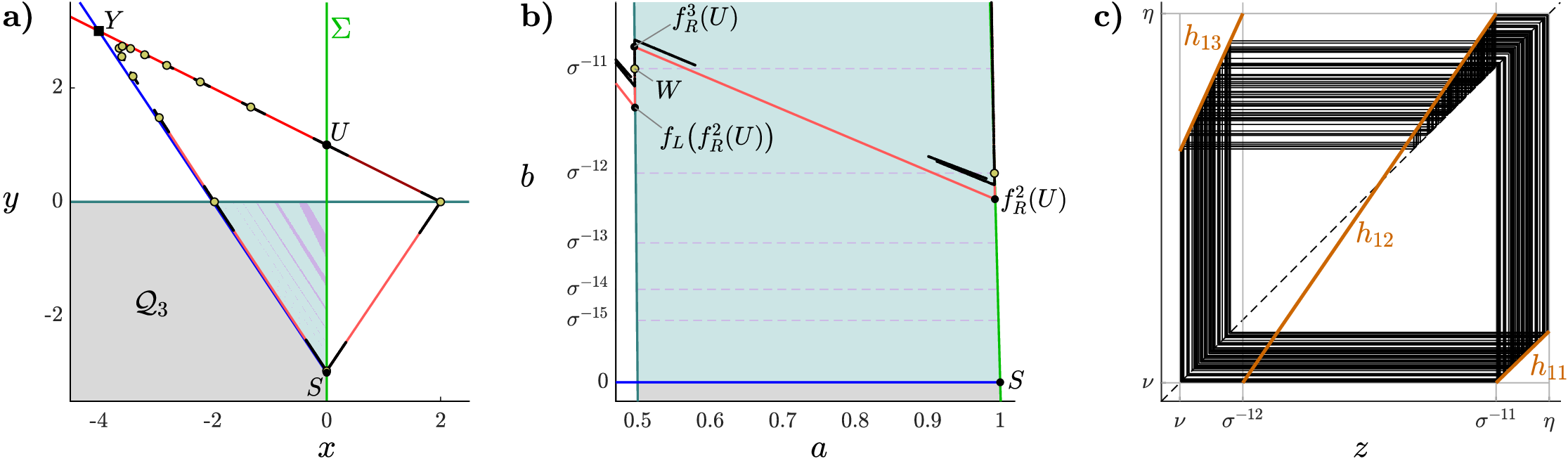}
\caption{
Panel (a) is a phase portrait of $f$ with $(\tau_L,\delta_L,\tau_R,\delta_R) = (2,0.75,-0.501,1.485)$
corresponding to the parameter point (iii) of Fig.~\ref{fig:bifSetBCNF}a.
The yellow circles indicate the points of a saddle $L^{13} R$-cycle,
the attractor is indicated in black,
and the third quadrant $\cQ_3$ is shaded by the value of $r$ as in Fig.~\ref{fig:chaoticAttr}.
Panel (b) shows part of the same plot in $(a,b)$-coordinates \eqref{eq:coordChange}
with dashed lines at $b = \sigma^{-k}$ for $k = 11,\ldots,15$
because the lavender regions are too narrow to see.
Panel (c) is a cobweb diagram of the corresponding map $h$
which uses $\sigma = 1.5$ and $\eta = 0.01236825$ and $\nu = 0.00675$
given by the formulas \eqref{eq:eta2} and \eqref{eq:nu2}.
\label{fig:bandexA}
} 
\end{center}
\end{figure}

Fig.~\ref{fig:bandcountAdd}b shows a corresponding bifurcation diagram of $f$.
We observe no significant differences in the arrangement of the bifurcations between this diagram and that of $h$.
However, for $f$ the bifurcations that form the boundaries of the lobes can in some places
be border-collision bifurcations instead of expansion bifurcations.
In this case the lobes have a more complicated geometry, as seen in the lower-left part of Fig.~\ref{fig:bifSetBCNFexB}a.
If $h$ has a multi-band attractor, then $f$ often has an attractor with multiple connected components,
as explained below, but this is not always the case (the attractor of $f$ may have positive but low density
in places where the attractor of $h$ has gaps).
The left-most and right-most intervals that comprise the attractor of $h$ correspond to the same connected component
of the attractor of $f$ (see Fig.~\ref{fig:bandexA}), whereas any other intervals correspond to distinct connected components.

If the attractor of $h$ is the single interval $J$,
then the corresponding attractor of the first return map $F$ 
stretches from the switching line $\Sigma$ near $S$
to its image $f(\Sigma)$ near $f(S)$.
By iterating the attractor of $F$ under $f$,
we obtain the attractor of $f$, which in this case is typically a connected set.

Now suppose the attractor of $h$ is a disjoint union of two intervals,
as in Fig.~\ref{fig:bandexA}c which corresponds to a parameter point in the lobe issuing from $\cP_{12}$.
The corresponding attractor of $F$ has two connected components in $\cQ_3$, see Fig.~\ref{fig:bandexA}b.
For this example, the corresponding attractor of $f$, see Fig.~\ref{fig:bandexA}a, appears to have $14$ connected components
(matching the period $k+m = 12+2 = 14$ of the periodic solution ($L^{12} R^2$-cycle)
of $f$ that corresponds to the fixed point $z_k^*$ of $h$).
Since the bifurcation set Fig.~\ref{fig:bifSetBCNFZoom} for $f$ has lobes
mimicking those in Fig.~\ref{fig:2D:num} for $h$,
we believe this behaviour is typical, and provide the following heuristic explanation.

The forward orbits of most points $P \in \cQ_3$ with $0 < b(P) \ll 1$
experience $r(P) = m$ iterations in the right half-plane before returning to $\cQ_3$.
Only if $b(P) \approx \sigma^{-k}$ for some $k$ does the value of $r(P)$ differ from $m$.
Indeed in Fig.~\ref{fig:bandexA} our numerical scan of the value of $r(P)$
detected $r(P) \ne m$ in the lavender regions shown in panel (a),
but in panel (b) these regions are too narrow to be seen,
so we have instead plotted dashed lines at $b = \sigma^{-k}$ for several values of $k$ to indicate the location of these regions.
For a path of points in $\cQ_3$ from some $b = \sigma^{-k}$
vertically downwards to $b = \sigma_{-k-1}$,
their images under $F$ stretch roughly from $\Sigma$ to $f(\Sigma)$.

The attractor of $f$, indicated with black dots in Fig.~\ref{fig:bandexA}a,
appears to be the closure of the unstable set of a $L^{13} R$-cycle, indicated with yellow circles.
To explain why the attractor is not a connected set,
consider the result of growing the unstable set outwards from the point $W$
of the $L^{13} R$-cycle, see Fig.~\ref{fig:bandexA}b.
This set is linear until developing kinks near $f_R^3(U)$ and $f_L \left( f_R^2(U) \right)$.
The $b$-values of $f_R^3(U)$ and $f_L \left( f_R^3(U) \right)$ are $\eta$ and $\sigma \nu$,
where $\frac{1}{\sigma^2} < \frac{\nu}{\eta} < \frac{1}{\sigma}$,
and thus the initial part of the unstable set is relatively short.
On a log-scale it is shorter than the distance from one dashed line $b = \sigma^{-k}$ to the next,
thus the image of this part of the unstable set does not stretch from $\Sigma$ to $f(\Sigma)$.
Subsequent images also do not stretch from $\Sigma$ to $f(\Sigma)$,
thus each of the 14 components of the unstable set are bounded away from one another.

\subsection{Bandcount-incrementing ($\frac{\nu}{\eta} > \sigma$)}
\label{sub:bandcountInc}

Finally suppose $\frac{\nu}{\eta} > \sigma$.
If also $\frac{\nu}{\eta} < \sigma^2$, as in Fig.~\ref{fig:bandcountInc}a,
the bifurcation diagram has intervals over which $N = 2$.
In the context of two-piece maps, the bifurcation structure is termed {\em bandcount-incrementing} \cite{AvEc08,AvEc08b,AvEc09}.
There are many possibilities for the precise nature of this structure,
as can be inferred from the bifurcation set shown in Fig.~\ref{fig:2D:num},
a magnification of which is provided by Fig.~\ref{fig:U}.

\begin{figure}[b!]
\begin{center}
\setlength{\unitlength}{1cm}
\begin{picture}(17,4.3)
\put(.7,0){\includegraphics[height=4.3cm]{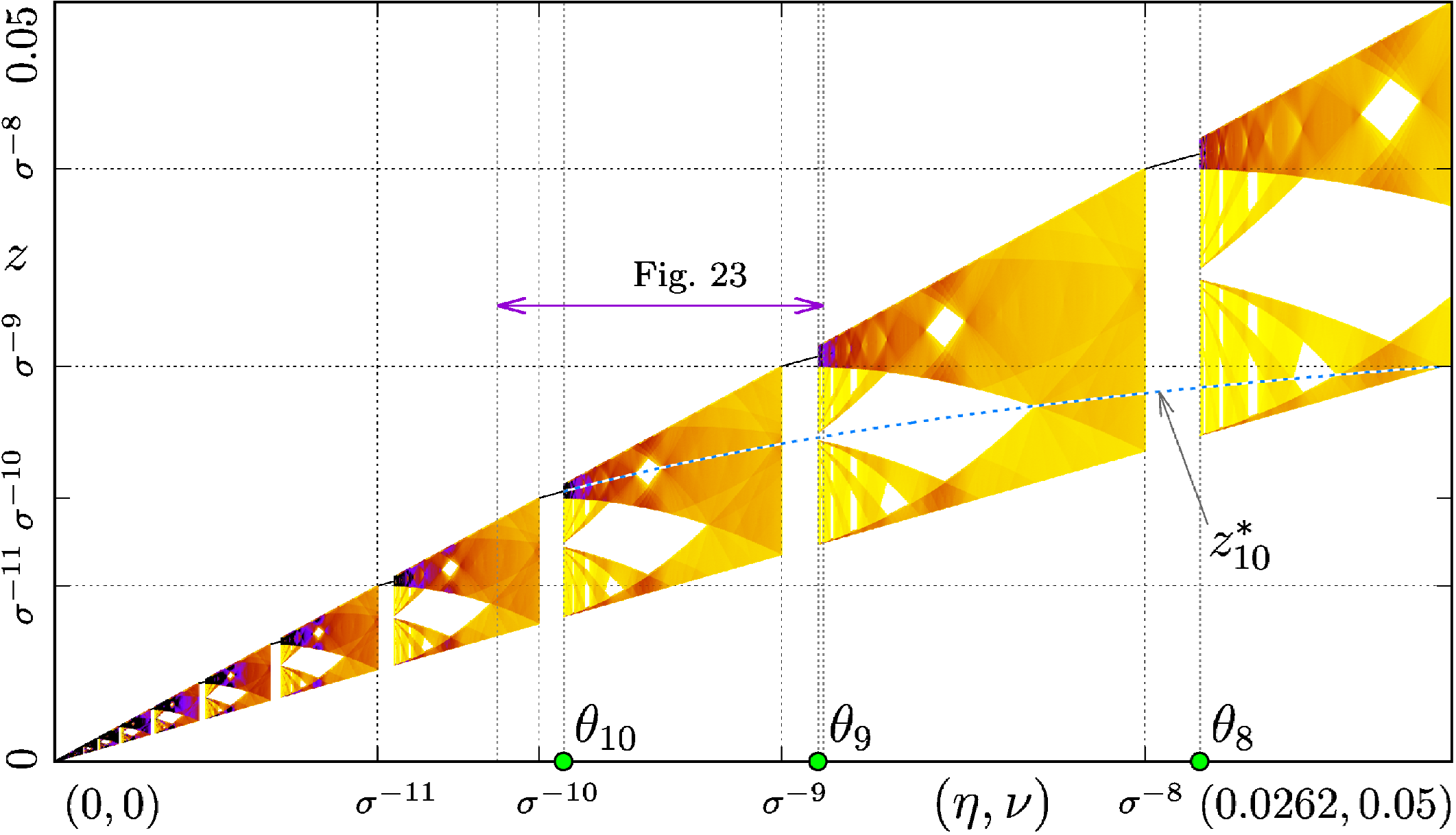}}
\put(9.5,0){\includegraphics[height=4.3cm]{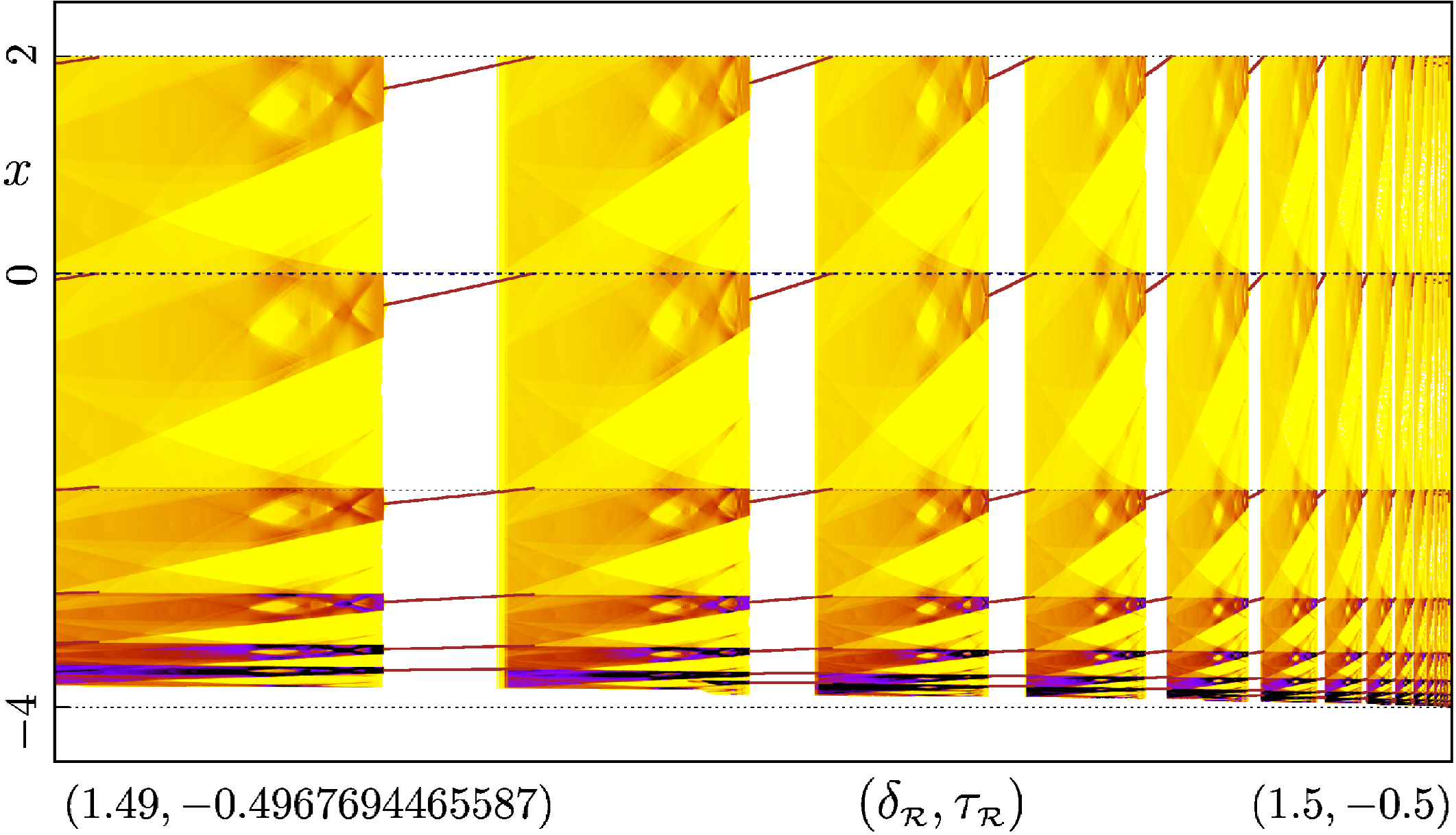}}
\put(0,3.9){{\bf a)}}
\put(8.8,3.9){{\bf b)}}
\end{picture}
\caption{
Bifurcation diagrams of $h$ (panel (a)) and $f$ (panel (b)) corresponding to the
lines D of Figs.~\ref{fig:2D:num} and \ref{fig:bifSetBCNFZoom} respectively.
Points are coloured according to density of the attractor.
For clarity panel (b) only shows points with $y < {\rm max} \left[ -\frac{3 x}{4}, 0 \right]$.
\label{fig:bandcountInc}
} 
\end{center}
\end{figure}

\begin{figure}[b!]
\begin{center}
\includegraphics[width=8cm]{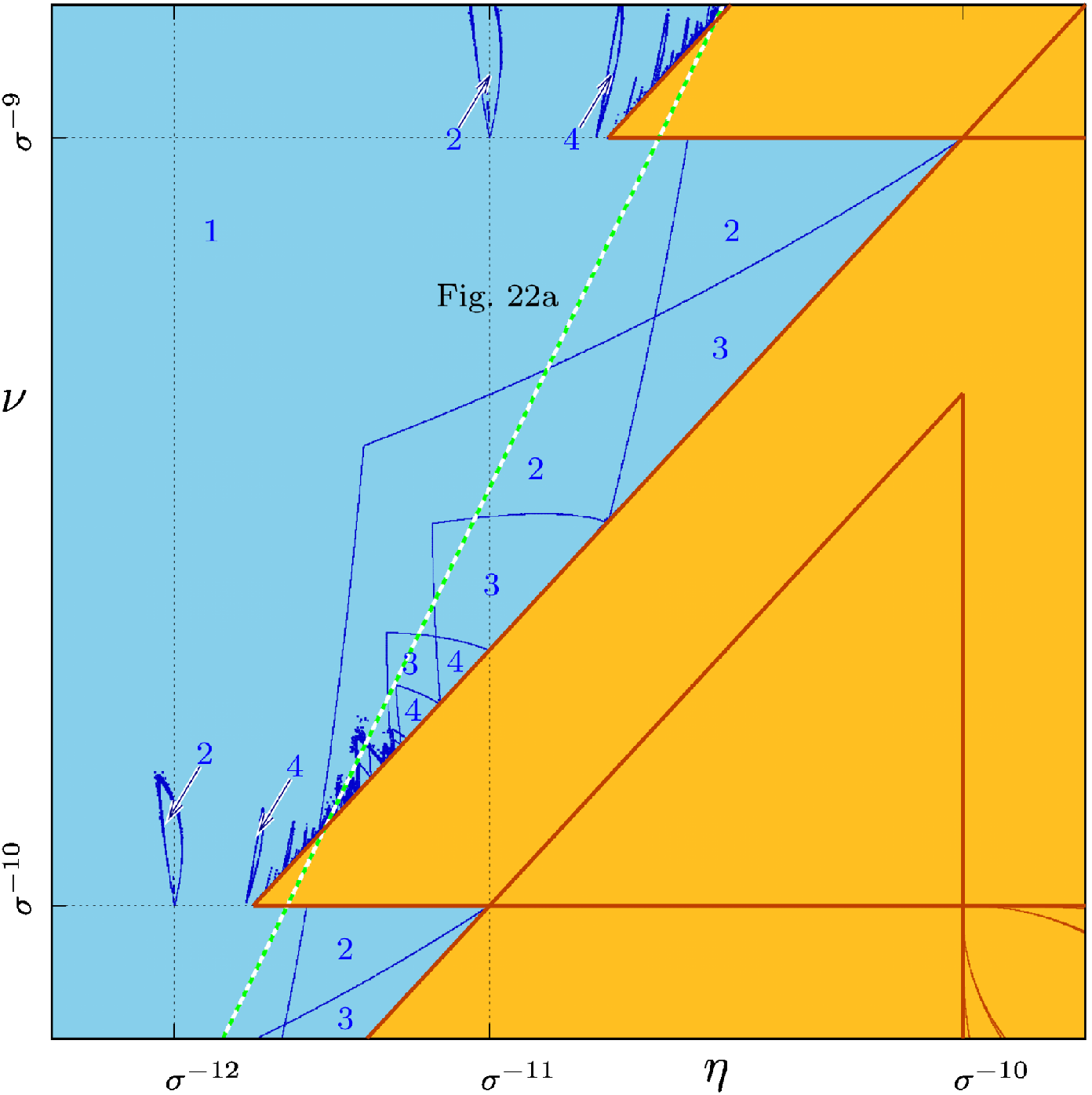}
\caption{
A magnification of the bifurcation set Fig.~\ref{fig:2D:num} of $h$ with $\sigma = 1.5$.
The blue regions are numbered by the number of connected components of the chaotic attractor.
\label{fig:U}
} 
\end{center}
\end{figure}

To the left of the triangles $\cP_k$, $h$ has a chaotic attractor.
Again this attractor is either an interval or a union of disjoint intervals.
The number of intervals comprising the attractor changes at expansion bifurcations (described above)
and {\em merging bifurcations}.
At a merging bifurcation a fixed point $z_k^*$ of $h$ with a negative stability multiplier
has a homoclinic orbit that includes a point of discontinuity of $h|_J$.
The point of discontinuity is not necessarily an endpoint of the interval $I_k$ associated with the fixed point.
Merging bifurcations cause
orbits in the attractor to gain or lose access to a neighbourhood of the fixed point.
Each piece of $h$ has the same range with endpoints $\eta$ and $\nu$,
thus merging bifurcations occur at parameter values
for which $\eta$ or $\nu$ is a rank-$j$ preimage of a fixed point for some $j \ge 1$
(that is, $\eta$ or $\nu$ maps to a fixed point under $j$ iterations of $h$).
It is worth noticing that the merging bifurcations occurring just above the diagonal edge
of $\cP_k$ where the fixed point $z_k^*$ loses stability by attaining a stability multiplier of $-1$,
do not necessarily involve this fixed point.
The majority of the merging bifurcations involve $z_j^*$ for some $j > k$.
For example, in Fig.~\ref{fig:bandcountInc}a several merging bifurcations involving $z_{10}^*$
occur just beyond $\theta_9$ and $\theta_8$ where $z_9^*$ and $z_8^*$ lose stability.

In part of Fig.~\ref{fig:bandcountInc}a, the attractor consists of a single interval.
The roughly rhombus-shaped holes are due to two merging bifurcations
that cause the attractor to develop a gap.
In Fig.~\ref{fig:bandcountInc}a there are also narrow parameter intervals
with four-interval attractors bounded by expansion bifurcations.
These intervals are caused by {\em floating regions}
that originate at the boundary between chaotic and non-chaotic attracting dynamics,
but under parameter variation have become disconnected from this boundary \cite{AvEc13}.

In Fig.~\ref{fig:U} we also observe additional merging bifurcation curves forming numerous substructures.
The precise arrangement of the curves differs for different values of $\sigma$
because each $\cP_k$ extends further to the left of the $(\eta,\nu)$-plane as the value of $\sigma$ is increased.
Some of the merging bifurcation curves extend into $\cP_k$ where they appear to give birth to chaotic repellers.
A careful description of this behaviour is left for future work.

Fig.~\ref{fig:bandcountInc}b shows the corresponding bifurcation diagram of $f$.
As expected $f$ has a stable period-$(k+2)$ solution
at parameter values belonging to the regions $\cP_k'$, and a chaotic attractor otherwise.
However, the chaotic attractors appear to be connected sets
in places where the corresponding attractor of $h$ is a union of several disjoint intervals.
Gaps in the attractors of $h$ correspond to low density areas for the attractors of $f$.
Consequently, roughly rhombus-shaped regions are still recognisable in Fig.~\ref{fig:bandcountInc}b.

\begin{figure}[b!]
\begin{center}
\includegraphics[width=17cm]{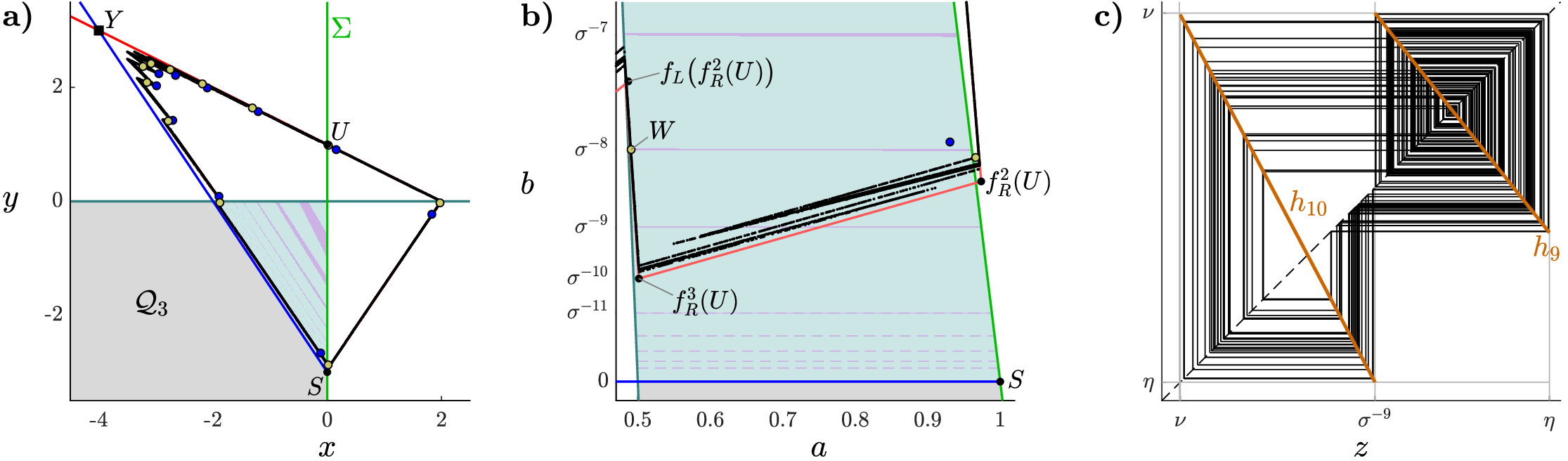}
\caption{
Panel (a) is a phase portrait of $f$ with $(\tau_L,\delta_L,\tau_R,\delta_R) = (2,0.75,-0.485,1.455)$
corresponding to the parameter point (iv) of Fig.~\ref{fig:bifSetBCNF}a.
The blue circles indicate the points of a stable $L^8 R^2$-cycle,
the yellow circles indicate the points of a saddle $L^8 R^3$-cycle,
and the numerically computed chaotic attractor is indicated in black.
Panel (b) shows part of the same plot in $(a,b)$-coordinates \eqref{eq:coordChange}.
Panel (c) is a cobweb diagram of the corresponding map $h$
which uses $\sigma = 1.5$ and $\eta = 0.01738125$ and $\nu = 0.03375$
given by \eqref{eq:eta2} and \eqref{eq:nu2}.
\label{fig:bandexB}
} 
\end{center}
\end{figure}

To explain why this occurs, we provide a representative example in Fig.~\ref{fig:bandexB}.
Here $h$ has a two-interval attractor, panel (c),
while $f$ appears to have a connected chaotic attractor equal to the closure of the unstable set of a saddle $L^8 R^3$-cycle, panel (a).
The map $f$ also a stable $L^8 R^2$-cycle because the parameter point (iv) belongs to $\cP_8'$, see Fig.~\ref{fig:bifSetBCNF}a.

Let $W$ be the point of the saddle $L^8 R^3$-cycle that belongs to $\cQ_3$.
As we grow the unstable set of this solution outwards from $W$,
the set is linear until developing kinks near $f_R^3(U)$ and $f_L \left( f_R^2(U) \right)$.
But $\frac{\nu}{\eta} > \sigma$, so the line segment from $f_R^3(U)$ and $f_L \left( f_R^2(U) \right)$
that approximates the initial part of the unstable set
intersects at least two horizontal lines $b = \sigma^{-k}$.
Hence the image of this part of the unstable set stretches from $\Sigma$ to $f(\Sigma)$.
For this reason the 11 branches of the unstable set are intertwined,
and so their closure (the chaotic attractor of $f$) is a connected set.

\section{Summary and outlook}
\label{sec:conc}

The bifurcation structures of the two-dimensional border-collision normal form $f$ are not fully understood.
In this paper we have shown how a common structure involving periodicity and chaos
is underpinned by a three-parameter family $h$ of one-dimensional discontinuous maps.
Theorem \ref{th:main} shows that $h$ approximates $f$ functionally,
while our numerical explorations show that $h$ captures the dynamics of $f$ near
the codimension-two subsumed homoclinic connections from which $h$ is derived.

As we move away from the codimension-two points, the bifurcation structure of $f$ develops
features not exhibited by the bifurcation structure of $h$.
For example, any parameter point of $h$ belongs to at most two periodicity triangles $\cP_k$,
but this is not true for $f$ sufficiently far from the codimension-two point, e.g.~at the point (i) of Fig.~\ref{fig:bifSetBCNF}a.
The periodicity tongues forming the period-adding structure of $h$, see Fig.~\ref{fig:Y},
do not overlap due to the uniqueness of the rotation number for non-overlapping circle maps.
But for $f$, some tongues do overlap,
and we believe overlaps occur arbitrarily close to the codimension-two point, see Fig.~\ref{fig:bifSetBCNFZoom}.
Also the tongues of $f$ can develop pinch points, as in Fig.~\ref{fig:bifSetBCNFexB}a.

The main difference between the bifurcation structures of $f$ and $h$ visible in Fig.~\ref{fig:bifSetBCNF}
concerns the connectedness of chaotic attractor above the regions $\cP_k'$ and $\cP_k$.
For $h$ the attractor is disconnected in some regions, shown more clearly in Fig.~\ref{fig:U},
while for $f$ the attractor is connected but has low density
where the corresponding attractor of $h$ contains gaps, see Fig.~\ref{fig:bandcountInc}.

From an applied viewpoint these differences are minor.
The bifurcation structure that emanates from a subsumed homoclinic tangency
often extends to a large area of parameter space, as in Fig.~\ref{fig:bifSetLargeScale}.
Our main point is that the family $h$ provides a simple and striking first-order approximation
to the bifurcation structure of $f$ over large areas.

For piecewise-linear maps of dimension greater than two,
we expect the same family $h$ applies near subsumed homoclinic connections.
However, it is likely that such maps will not have any attractors near the homoclinic connections
because in more than two dimensions the unstable set of the saddle is not forced to spiral inwards, as in Fig.~\ref{fig:schem}.
Additional dimensions allow the stable and unstable sets to intersect more readily,
and these intersections destroy the attractor.

It remains to consider families of two-dimensional maps that are discontinuous
or involve more than two linear pieces, as in \cite{Si20},
and study the analogous reduction to one dimension.
It also remains to further study some details of the bifurcation structure of $h$,
as several features relating to bandcount-incrementing
involve more than two pieces of $h$ and are not accounted for by the existing two-piece theory.

\section*{Acknowledgements}

DJWS was supported by Marsden Fund contract MAU2209 managed by Royal Society Te Ap\={a}rangi.
VA was supported by the German Research Foundation within the scope of the project AV 111/3-1.
The authors thank Lukas Ott for assistance with the numerical explorations.


\end{document}